\documentclass[a4paper,onecolumn, unpublished]{compositionalityarticle}
\pdfoutput=1

\usepackage[utf8]{inputenc}
\usepackage{amsmath}
\usepackage{amssymb}
\usepackage{bbold}
\usepackage{comment}
\usepackage{graphicx}
\usepackage{scalerel}
\usepackage{braket}
\usepackage{leftidx}
\usepackage[title]{appendix}
\usepackage{tikz}
\usepackage{tikzit}
\usepackage{tikz-cd}
\usepackage{cite}
\usepackage{geometry}
\usepackage{mathtools}
\usepackage{ stmaryrd }
\usepackage{circuitikz}
\usepackage{textcomp}
\usepackage{accents}
\usepackage{comment}
\usepackage{ mathrsfs }
\usepackage{authblk}
\usepackage{breakurl,amsmath,amsthm,amssymb,xspace,xypic,enumerate,color}
\usepackage{float}

\tikzcdset{scale cd/.style={every label/.append style={scale=#1},
    cells={nodes={scale=#1}}}}

\usetikzlibrary{shapes.multipart}

\tikzstyle{bwSpider}=[
       rectangle split,
       rectangle split parts=2,
       rectangle split part fill={black,white},
 minimum size=3.6 mm, inner sep=-2mm, draw=black,scale=0.5,rounded corners=0.8 mm
       ]
 \tikzstyle{wbSpider}=[
       rectangle split,
       rectangle split parts=2,
       rectangle split part fill={white,black},
 minimum size=3.6 mm, inner sep=-2mm, draw=black,scale=0.5,rounded corners=0.8 mm
       ]
\tikzstyle{cWire}=[densely dotted, thick]
\tikzstyle{env}=[copoint,regular polygon rotate=0,minimum width=0.2cm, fill=black]

\tikzstyle{probs}=[shape=semicircle,fill=white,draw=black,shape border rotate=180,minimum width=1.2cm]

\tikzstyle{every picture}=[baseline=-0.25em,scale=0.5]
\tikzstyle{dotpic}=[] 
\tikzstyle{diredges}=[every to/.style={diredge}]
\tikzstyle{math matrix}=[matrix of math nodes,left delimiter=(,right delimiter=),inner sep=2pt,column sep=1em,row sep=0.5em,nodes={inner sep=0pt},text height=1.5ex, text depth=0.25ex]

\tikzstyle{inline text}=[text height=1.5ex, text depth=0.25ex,yshift=0.5mm]
\tikzstyle{label}=[font=\footnotesize,text height=1.5ex, text depth=0.25ex,yshift=0.5mm]
\tikzstyle{left label}=[label,anchor=east,xshift=1.5mm]
\tikzstyle{right label}=[label,anchor=west,xshift=-1.5mm]

\tikzstyle{braceedge}=[decorate,decoration={brace,amplitude=2mm,raise=-1mm}]
\tikzstyle{small braceedge}=[decorate,decoration={brace,amplitude=1mm,raise=-1mm}]

\tikzstyle{doubled}=[line width=1.6pt] 
\tikzstyle{boldedge}=[doubled,shorten <=-0.17mm,shorten >=-0.17mm]
\tikzstyle{boldedgegray}=[doubled,gray,shorten <=-0.17mm,shorten >=-0.17mm]
\tikzstyle{singleedgegray}=[gray]

\tikzstyle{semidoubled}=[line width=1.4pt] 
\tikzstyle{semiboldedgegray}=[semidoubled,gray,shorten <=-0.17mm,shorten >=-0.17mm]

\tikzstyle{boxedge}=[semiboldedgegray]

\tikzstyle{dottededge}=[dashed,shorten <=-0.17mm,shorten >=-0.17mm]

\tikzstyle{boldedgedashed}=[very thick,dashed,shorten <=-0.17mm,shorten >=-0.17mm]
\tikzstyle{vboldedgedashed}=[doubled,dashed,shorten <=-0.17mm,shorten >=-0.17mm]
\tikzstyle{left hook arrow}=[left hook-latex]
\tikzstyle{right hook arrow}=[right hook-latex]
\tikzstyle{sembracket}=[line width=0.5pt,shorten <=-0.07mm,shorten >=-0.07mm]

\tikzstyle{causal edge}=[->,thick,gray]
\tikzstyle{causal nondir}=[thick,gray]
\tikzstyle{timeline}=[thick,gray, dashed]

\tikzstyle{cedge}=[<->,thick,gray!70!white]

\tikzstyle{empty diagram}=[draw=gray!40!white,dashed,shape=rectangle,minimum width=1cm,minimum height=1cm]
\tikzstyle{empty diagram small}=[draw=gray!50!white,dashed,shape=rectangle,minimum width=0.6cm,minimum height=0.5cm]

\tikzstyle{dot}=[inner sep=0mm,minimum width=2mm,minimum height=2mm,draw,shape=circle]

\tikzstyle{phase dot}=[pdot,phase dimensions]
\tikzstyle{wphase dot}=[dot, phase dimensions]

\tikzstyle{leak}=[white dot, shape=regular polygon, minimum size=300mm, regular polygon sides=3, outer sep=-0.2mm, regular polygon rotate=270]
\tikzstyle{preleak}=[trapezium, trapezium angle=67.5, draw, inner sep=0.1pt, outer sep=0pt, minimum height=2mm, minimum width=4pt,rotate=270]
\tikzstyle{proj}=[white dot, shape=regular polygon, minimum size=3.3 mm, regular polygon sides=4, outer sep=-0.2mm]
\tikzstyle{Vleak}=[white dot, shape=regular polygon, minimum size=3.3 mm, regular polygon sides=3, outer sep=-0.2mm, regular polygon rotate=90]
\tikzstyle{dleak}=[white dot, line width=1.6pt, shape=regular polygon, minimum size=3.3 mm, regular polygon sides=3, outer sep=-0.2mm, regular polygon rotate=270]

\tikzstyle{Wsquare}=[white dot, shape=regular polygon, rounded corners=0.8 mm, minimum size=3.3 mm, regular polygon sides=3, outer sep=-0.2mm]
\tikzstyle{Wsquareadj}=[white dot, shape=regular polygon, rounded corners=0.8 mm, minimum size=3.3 mm, regular polygon sides=3, outer sep=-0.2mm, regular polygon rotate=180]
\tikzstyle{ddot}=[inner sep=0mm, doubled, minimum width=2.5mm,minimum height=2.5mm,draw,shape=circle]

\tikzstyle{black dot}=[dot,fill=black]
\tikzstyle{white dot}=[dot,fill=white,,text depth=-0.2mm]
\tikzstyle{white Wsquare}=[Wsquare,fill=gray,,text depth=-0.2mm]
\tikzstyle{white Wsquareadj}=[Wsquareadj,fill=white,,text depth=-0.2mm]
\tikzstyle{green dot}=[white dot] 
\tikzstyle{gray dot}=[dot,fill=gray!40!white,,text depth=-0.2mm]
\tikzstyle{red dot}=[gray dot] 

\tikzstyle{black ddot}=[ddot,fill=black]
\tikzstyle{white ddot}=[ddot,fill=white]
\tikzstyle{gray ddot}=[ddot,fill=gray!40!white]

\tikzstyle{gray edge}=[gray!60!white]

\tikzstyle{small dot}=[inner sep=0.5mm,minimum width=0pt,minimum height=0pt,draw,shape=circle]

\tikzstyle{small black dot}=[small dot,fill=black]
\tikzstyle{small white dot}=[small dot,fill=white]
\tikzstyle{small gray dot}=[small dot,fill=gray!40!white]

\tikzstyle{causal dot}=[inner sep=0.4mm,minimum width=0pt,minimum height=0pt,draw=white,shape=circle,fill=gray!40!white]

\tikzstyle{phase dimensions}=[minimum size=5mm,font=\footnotesize,rectangle,rounded corners=2.5mm,inner sep=0.2mm,outer sep=-2mm]
\tikzstyle{dphase dimensions}=[minimum size=5mm,font=\footnotesize,rectangle,rounded corners=2.5mm,inner sep=0.2mm,outer sep=-2mm]

\tikzstyle{white phase dot}=[dot,fill=white,phase dimensions]
\tikzstyle{white phase ddot}=[ddot,fill=white,dphase dimensions]

\tikzstyle{white rect ddot}=[draw=black,fill=white,doubled,minimum size=5mm,font=\footnotesize,rectangle,rounded corners=2.5mm,inner sep=0.2mm]
\tikzstyle{gray rect ddot}=[draw=black,fill=gray!40!white,doubled,minimum size=6mm,font=\footnotesize,rectangle,rounded corners=3mm]

\tikzstyle{gray phase dot}=[dot,fill=gray!40!white,phase dimensions]
\tikzstyle{gray phase ddot}=[ddot,fill=gray!40!white,dphase dimensions]
\tikzstyle{grey phase dot}=[gray phase dot]
\tikzstyle{grey phase ddot}=[gray phase ddot]

\tikzstyle{small phase dimensions}=[minimum size=4mm,font=\tiny,rectangle,rounded corners=2mm,inner sep=0.2mm,outer sep=-2mm]
\tikzstyle{small dphase dimensions}=[minimum size=4mm,font=\tiny,rectangle,rounded corners=2mm,inner sep=0.2mm,outer sep=-2mm]

\tikzstyle{small gray phase dot}=[dot,fill=gray!40!white,small phase dimensions]
\tikzstyle{small gray phase ddot}=[ddot,fill=gray!40!white,small dphase dimensions]

\tikzstyle{small map}=[draw,shape=rectangle,minimum height=4mm,minimum width=4mm,fill=white]

\tikzstyle{cnot}=[fill=white,shape=circle,inner sep=-1.4pt]

\tikzstyle{asym hadamard}=[fill=white,draw,shape=NEbox,inner sep=0.6mm,font=\footnotesize,minimum height=4mm]
\tikzstyle{asym hadamard conj}=[fill=white,draw,shape=NWbox,inner sep=0.6mm,font=\footnotesize,minimum height=4mm]
\tikzstyle{asym hadamard dag}=[fill=white,draw,shape=SEbox,inner sep=0.6mm,font=\footnotesize,minimum height=4mm]

\tikzstyle{hadamard}=[fill=white,draw,inner sep=0.6mm,font=\footnotesize,minimum height=4mm,minimum width=4mm]
\tikzstyle{small hadamard}=[fill=white,draw,inner sep=0.6mm,minimum height=1.5mm,minimum width=1.5mm]
\tikzstyle{small hadamard rotate}=[small hadamard,rotate=45]
\tikzstyle{dhadamard}=[hadamard,doubled]
\tikzstyle{small dhadamard}=[small hadamard,doubled]
\tikzstyle{small dhadamard rotate}=[small hadamard rotate,doubled]
\tikzstyle{antipode}=[white dot,inner sep=0.3mm,font=\footnotesize]

\tikzstyle{scalar}=[diamond,draw,inner sep=0.5pt,font=\small]
\tikzstyle{dscalar}=[diamond,doubled, draw,inner sep=0.5pt,font=\small]

\tikzstyle{small box}=[rectangle,inline text,fill=white,draw,minimum height=5mm,yshift=-0.5mm,minimum width=5mm,font=\small]
\tikzstyle{small gray box}=[small box,fill=gray!30]
\tikzstyle{medium box}=[rectangle,inline text,fill=white,draw,minimum height=5mm,yshift=-0.5mm,minimum width=10mm,font=\small]
\tikzstyle{square box}=[small box] 
\tikzstyle{medium gray box}=[small box,fill=gray!30]
\tikzstyle{semilarge box}=[rectangle,inline text,fill=white,draw,minimum height=5mm,yshift=-0.5mm,minimum width=12.5mm,font=\small]
\tikzstyle{large box}=[rectangle,inline text,fill=white,draw,minimum height=5mm,yshift=-0.5mm,minimum width=15mm,font=\small]
\tikzstyle{large gray box}=[small box,fill=gray!30]

\tikzstyle{Bayes box}=[rectangle,fill=black,draw, minimum height=3mm, minimum width=3mm]

\tikzstyle{gray square point}=[small box,fill=gray!50]

\tikzstyle{dphase box white}=[dhadamard]
\tikzstyle{dphase box gray}=[dhadamard,fill=gray!50!white]
\tikzstyle{phase box white}=[hadamard]
\tikzstyle{phase box gray}=[hadamard,fill=gray!50!white]

\tikzstyle{point}=[regular polygon,regular polygon sides=3,draw,scale=0.75,inner sep=-0.5pt,minimum width=9mm,fill=white,regular polygon rotate=180]
\tikzstyle{point nosep}=[regular polygon,regular polygon sides=3,draw,scale=0.75,inner sep=-2pt,minimum width=9mm,fill=white,regular polygon rotate=180]
\tikzstyle{copoint}=[regular polygon,regular polygon sides=3,draw,scale=0.75,inner sep=-0.5pt,minimum width=9mm,fill=white]
\tikzstyle{dpoint}=[point,doubled]
\tikzstyle{dcopoint}=[copoint,doubled]

\tikzstyle{pointgrow}=[shape=cornerpoint,kpoint common,scale=0.75,inner sep=3pt]
\tikzstyle{pointgrow dag}=[shape=cornercopoint,kpoint common,scale=0.75,inner sep=3pt]

\tikzstyle{wide copoint}=[fill=white,draw,shape=isosceles triangle,shape border rotate=90,isosceles triangle stretches=true,inner sep=0pt,minimum width=1.5cm,minimum height=6.12mm]
\tikzstyle{wide point}=[fill=white,draw,shape=isosceles triangle,shape border rotate=-90,isosceles triangle stretches=true,inner sep=0pt,minimum width=1.5cm,minimum height=6.12mm,yshift=-0.0mm]
\tikzstyle{wide point plus}=[fill=white,draw,shape=isosceles triangle,shape border rotate=-90,isosceles triangle stretches=true,inner sep=0pt,minimum width=1.74cm,minimum height=7mm,yshift=-0.0mm]

\tikzstyle{wide dpoint}=[fill=white,doubled,draw,shape=isosceles triangle,shape border rotate=-90,isosceles triangle stretches=true,inner sep=0pt,minimum width=1.5cm,minimum height=6.12mm,yshift=-0.0mm]

\tikzstyle{tinypoint}=[regular polygon,regular polygon sides=3,draw,scale=0.55,inner sep=-0.15pt,minimum width=6mm,fill=white,regular polygon rotate=180]

\tikzstyle{white point}=[point]
\tikzstyle{white dpoint}=[dpoint]
\tikzstyle{green point}=[white point] 
\tikzstyle{white copoint}=[copoint]
\tikzstyle{gray point}=[point,fill=gray!40!white]
\tikzstyle{gray dpoint}=[gray point,doubled]
\tikzstyle{red point}=[gray point] 
\tikzstyle{gray copoint}=[copoint,fill=gray!40!white]
\tikzstyle{gray dcopoint}=[gray copoint,doubled]

\tikzstyle{white point guide}=[regular polygon,regular polygon sides=3,font=\scriptsize,draw,scale=0.65,inner sep=-0.5pt,minimum width=9mm,fill=white,regular polygon rotate=180]

\tikzstyle{black point}=[point,fill=black,font=\color{white}]
\tikzstyle{black copoint}=[copoint,fill=black,font=\color{white}]

\tikzstyle{tiny gray point}=[tinypoint,fill=gray!40!white]

\tikzstyle{diredge}=[->]
\tikzstyle{ddiredge}=[<->]
\tikzstyle{rdiredge}=[<-]
\tikzstyle{thickdiredge}=[->, very thick]
\tikzstyle{pointer edge}=[->,very thick,gray]
\tikzstyle{pointer edge part}=[very thick,gray]
\tikzstyle{dashed edge}=[dashed]
\tikzstyle{thick dashed edge}=[very thick,dashed]
\tikzstyle{thick gray dashed edge}=[thick dashed edge,gray!40]
\tikzstyle{thick map edge}=[very thick,|->]

\makeatletter
\newcommand{\boxshape}[3]{%
\pgfdeclareshape{#1}{
\inheritsavedanchors[from=rectangle] 
\inheritanchorborder[from=rectangle]
\inheritanchor[from=rectangle]{center}
\inheritanchor[from=rectangle]{north}
\inheritanchor[from=rectangle]{south}
\inheritanchor[from=rectangle]{west}
\inheritanchor[from=rectangle]{east}
\backgroundpath{
\southwest \pgf@xa=\pgf@x \pgf@ya=\pgf@y
\northeast \pgf@xb=\pgf@x \pgf@yb=\pgf@y

\@tempdima=#2
\@tempdimb=#3

\pgfpathmoveto{\pgfpoint{\pgf@xa - 5pt + \@tempdima}{\pgf@ya}}
\pgfpathlineto{\pgfpoint{\pgf@xa - 5pt - \@tempdima}{\pgf@yb}}
\pgfpathlineto{\pgfpoint{\pgf@xb + 5pt + \@tempdimb}{\pgf@yb}}
\pgfpathlineto{\pgfpoint{\pgf@xb + 5pt - \@tempdimb}{\pgf@ya}}
\pgfpathlineto{\pgfpoint{\pgf@xa - 5pt + \@tempdima}{\pgf@ya}}
\pgfpathclose
}
}}

\boxshape{NEbox}{0pt}{5pt}
\boxshape{SEbox}{0pt}{-5pt}
\boxshape{NWbox}{5pt}{0pt}
\boxshape{SWbox}{-5pt}{0pt}
\boxshape{EBox}{-3pt}{3pt}
\boxshape{WBox}{3pt}{-3pt}
\makeatother

\tikzstyle{cloud}=[shape=cloud,draw,minimum width=1.5cm,minimum height=1.5cm]

\tikzstyle{map}=[draw,shape=NEbox,inner sep=2pt,minimum height=6mm,fill=white]
\tikzstyle{dashedmap}=[draw,dashed,shape=NEbox,inner sep=2pt,minimum height=6mm,fill=white]
\tikzstyle{mapdag}=[draw,shape=SEbox,inner sep=2pt,minimum height=6mm,fill=white]
\tikzstyle{mapadj}=[draw,shape=SEbox,inner sep=2pt,minimum height=6mm,fill=white]
\tikzstyle{maptrans}=[draw,shape=SWbox,inner sep=2pt,minimum height=6mm,fill=white]
\tikzstyle{mapconj}=[draw,shape=NWbox,inner sep=2pt,minimum height=6mm,fill=white]

\tikzstyle{medium map}=[draw,shape=NEbox,inner sep=2pt,minimum height=6mm,fill=white,minimum width=7mm]
\tikzstyle{medium map dag}=[draw,shape=SEbox,inner sep=2pt,minimum height=6mm,fill=white,minimum width=7mm]
\tikzstyle{medium map adj}=[draw,shape=SEbox,inner sep=2pt,minimum height=6mm,fill=white,minimum width=7mm]
\tikzstyle{medium map trans}=[draw,shape=SWbox,inner sep=2pt,minimum height=6mm,fill=white,minimum width=7mm]
\tikzstyle{medium map conj}=[draw,shape=NWbox,inner sep=2pt,minimum height=6mm,fill=white,minimum width=7mm]
\tikzstyle{semilarge map}=[draw,shape=NEbox,inner sep=2pt,minimum height=6mm,fill=white,minimum width=9.5mm]
\tikzstyle{semilarge map trans}=[draw,shape=SWbox,inner sep=2pt,minimum height=6mm,fill=white,minimum width=9.5mm]
\tikzstyle{semilarge map adj}=[draw,shape=SEbox,inner sep=2pt,minimum height=6mm,fill=white,minimum width=9.5mm]
\tikzstyle{semilarge map dag}=[draw,shape=SEbox,inner sep=2pt,minimum height=6mm,fill=white,minimum width=9.5mm]
\tikzstyle{semilarge map conj}=[draw,shape=NWbox,inner sep=2pt,minimum height=6mm,fill=white,minimum width=9.5mm]
\tikzstyle{large map}=[draw,shape=NEbox,inner sep=2pt,minimum height=6mm,fill=white,minimum width=12mm]
\tikzstyle{large map conj}=[draw,shape=NWbox,inner sep=2pt,minimum height=6mm,fill=white,minimum width=12mm]
\tikzstyle{very large map}=[draw,shape=NEbox,inner sep=2pt,minimum height=6mm,fill=white,minimum width=17mm]

\tikzstyle{medium dmap}=[draw,doubled,shape=NEbox,inner sep=2pt,minimum height=6mm,fill=white,minimum width=7mm]
\tikzstyle{medium dmap dag}=[draw,doubled,shape=SEbox,inner sep=2pt,minimum height=6mm,fill=white,minimum width=7mm]
\tikzstyle{medium dmap adj}=[draw,doubled,shape=SEbox,inner sep=2pt,minimum height=6mm,fill=white,minimum width=7mm]
\tikzstyle{medium dmap trans}=[draw,doubled,shape=SWbox,inner sep=2pt,minimum height=6mm,fill=white,minimum width=7mm]
\tikzstyle{medium dmap conj}=[draw,doubled,shape=NWbox,inner sep=2pt,minimum height=6mm,fill=white,minimum width=7mm]
\tikzstyle{semilarge dmap}=[draw,doubled,shape=NEbox,inner sep=2pt,minimum height=6mm,fill=white,minimum width=9.5mm]
\tikzstyle{semilarge dmap trans}=[draw,doubled,shape=SWbox,inner sep=2pt,minimum height=6mm,fill=white,minimum width=9.5mm]
\tikzstyle{semilarge dmap adj}=[draw,doubled,shape=SEbox,inner sep=2pt,minimum height=6mm,fill=white,minimum width=9.5mm]
\tikzstyle{semilarge dmap dag}=[draw,doubled,shape=SEbox,inner sep=2pt,minimum height=6mm,fill=white,minimum width=9.5mm]
\tikzstyle{semilarge dmap conj}=[draw,doubled,shape=NWbox,inner sep=2pt,minimum height=6mm,fill=white,minimum width=9.5mm]
\tikzstyle{large dmap}=[draw,doubled,shape=NEbox,inner sep=2pt,minimum height=6mm,fill=white,minimum width=12mm]
\tikzstyle{large dmap conj}=[draw,doubled,shape=NWbox,inner sep=2pt,minimum height=6mm,fill=white,minimum width=12mm]
\tikzstyle{large dmap trans}=[draw,doubled,shape=SWbox,inner sep=2pt,minimum height=6mm,fill=white,minimum width=12mm]
\tikzstyle{large dmap adj}=[draw,doubled,shape=SEbox,inner sep=2pt,minimum height=6mm,fill=white,minimum width=12mm]
\tikzstyle{large dmap dag}=[draw,doubled,shape=SEbox,inner sep=2pt,minimum height=6mm,fill=white,minimum width=12mm]
\tikzstyle{very large dmap}=[draw,doubled,shape=NEbox,inner sep=2pt,minimum height=6mm,fill=white,minimum width=19.5mm]

\tikzstyle{muxbox}=[draw,shape=rectangle,minimum height=3mm,minimum width=3mm,fill=white]
\tikzstyle{dmuxbox}=[muxbox,doubled]

\tikzstyle{box}=[draw,shape=rectangle,inner sep=2pt,minimum height=6mm,minimum width=6mm,fill=white]
\tikzstyle{dbox}=[draw,doubled,shape=rectangle,inner sep=2pt,minimum height=6mm,minimum width=6mm,fill=white]
\tikzstyle{dmap}=[draw,doubled,shape=NEbox,inner sep=2pt,minimum height=6mm,fill=white]
\tikzstyle{dmapdag}=[draw,doubled,shape=SEbox,inner sep=2pt,minimum height=6mm,fill=white]
\tikzstyle{dmapadj}=[draw,doubled,shape=SEbox,inner sep=2pt,minimum height=6mm,fill=white]
\tikzstyle{dmaptrans}=[draw,doubled,shape=SWbox,inner sep=2pt,minimum height=6mm,fill=white]
\tikzstyle{dmapconj}=[draw,doubled,shape=NWbox,inner sep=2pt,minimum height=6mm,fill=white]

\tikzstyle{ddmap}=[draw,doubled,dashed,shape=NEbox,inner sep=2pt,minimum height=6mm,fill=white]
\tikzstyle{ddmapdag}=[draw,doubled,dashed,shape=SEbox,inner sep=2pt,minimum height=6mm,fill=white]
\tikzstyle{ddmapadj}=[draw,doubled,dashed,shape=SEbox,inner sep=2pt,minimum height=6mm,fill=white]
\tikzstyle{ddmaptrans}=[draw,doubled,dashed,shape=SWbox,inner sep=2pt,minimum height=6mm,fill=white]
\tikzstyle{ddmapconj}=[draw,doubled,dashed,shape=NWbox,inner sep=2pt,minimum height=6mm,fill=white]

\boxshape{sNEbox}{0pt}{3pt}
\boxshape{sSEbox}{0pt}{-3pt}
\boxshape{sNWbox}{3pt}{0pt}
\boxshape{sSWbox}{-3pt}{0pt}
\tikzstyle{smap}=[draw,shape=sNEbox,fill=white]
\tikzstyle{smapdag}=[draw,shape=sSEbox,fill=white]
\tikzstyle{smapadj}=[draw,shape=sSEbox,fill=white]
\tikzstyle{smaptrans}=[draw,shape=sSWbox,fill=white]
\tikzstyle{smapconj}=[draw,shape=sNWbox,fill=white]

\tikzstyle{dsmap}=[draw,dashed,shape=sNEbox,fill=white]
\tikzstyle{dsmapdag}=[draw,dashed,shape=sSEbox,fill=white]
\tikzstyle{dsmaptrans}=[draw,dashed,shape=sSWbox,fill=white]
\tikzstyle{dsmapconj}=[draw,dashed,shape=sNWbox,fill=white]

\boxshape{mNEbox}{0pt}{10pt}
\boxshape{mSEbox}{0pt}{-10pt}
\boxshape{mNWbox}{10pt}{0pt}
\boxshape{mSWbox}{-10pt}{0pt}
\tikzstyle{mmap}=[draw,shape=mNEbox]
\tikzstyle{mmapdag}=[draw,shape=mSEbox]
\tikzstyle{mmaptrans}=[draw,shape=mSWbox]
\tikzstyle{mmapconj}=[draw,shape=mNWbox]

\tikzstyle{mmapgray}=[draw,fill=gray!40!white,shape=mNEbox]
\tikzstyle{smapgray}=[draw,fill=gray!40!white,shape=sNEbox]

\makeatletter

\pgfdeclareshape{cornerpoint}{
\inheritsavedanchors[from=rectangle] 
\inheritanchorborder[from=rectangle]
\inheritanchor[from=rectangle]{center}
\inheritanchor[from=rectangle]{north}
\inheritanchor[from=rectangle]{south}
\inheritanchor[from=rectangle]{west}
\inheritanchor[from=rectangle]{east}
\backgroundpath{
\southwest \pgf@xa=\pgf@x \pgf@ya=\pgf@y
\northeast \pgf@xb=\pgf@x \pgf@yb=\pgf@y

\pgfmathsetmacro{\pgf@shorten@left}{\pgfkeysvalueof{/tikz/shorten left}}
\pgfmathsetmacro{\pgf@shorten@right}{\pgfkeysvalueof{/tikz/shorten right}}

\pgfpathmoveto{\pgfpoint{0.5 * (\pgf@xa + \pgf@xb)}{\pgf@ya - 5pt}}
\pgfpathlineto{\pgfpoint{\pgf@xa - 8pt + \pgf@shorten@left}{\pgf@yb - 1.5 * \pgf@shorten@left}}
\pgfpathlineto{\pgfpoint{\pgf@xa - 8pt + \pgf@shorten@left}{\pgf@yb}}
\pgfpathlineto{\pgfpoint{\pgf@xb + 8pt - \pgf@shorten@right}{\pgf@yb}}
\pgfpathlineto{\pgfpoint{\pgf@xb + 8pt - \pgf@shorten@right}{\pgf@yb - 1.5 * \pgf@shorten@right}}
\pgfpathclose
}
}

\pgfdeclareshape{cornercopoint}{
\inheritsavedanchors[from=rectangle] 
\inheritanchorborder[from=rectangle]
\inheritanchor[from=rectangle]{center}
\inheritanchor[from=rectangle]{north}
\inheritanchor[from=rectangle]{south}
\inheritanchor[from=rectangle]{west}
\inheritanchor[from=rectangle]{east}
\backgroundpath{
\southwest \pgf@xa=\pgf@x \pgf@ya=\pgf@y
\northeast \pgf@xb=\pgf@x \pgf@yb=\pgf@y

\pgfmathsetmacro{\pgf@shorten@left}{\pgfkeysvalueof{/tikz/shorten left}}
\pgfmathsetmacro{\pgf@shorten@right}{\pgfkeysvalueof{/tikz/shorten right}}

\pgfpathmoveto{\pgfpoint{0.5 * (\pgf@xa + \pgf@xb)}{\pgf@yb + 5pt}}
\pgfpathlineto{\pgfpoint{\pgf@xa - 8pt + \pgf@shorten@left}{\pgf@ya + 1.5 * \pgf@shorten@left}}
\pgfpathlineto{\pgfpoint{\pgf@xa - 8pt + \pgf@shorten@left}{\pgf@ya}}
\pgfpathlineto{\pgfpoint{\pgf@xb + 8pt - \pgf@shorten@right}{\pgf@ya}}
\pgfpathlineto{\pgfpoint{\pgf@xb + 8pt - \pgf@shorten@right}{\pgf@ya + 1.5 * \pgf@shorten@right}}
\pgfpathclose
}
}

\makeatother

\pgfkeyssetvalue{/tikz/shorten left}{0pt}
\pgfkeyssetvalue{/tikz/shorten right}{0pt}

\tikzstyle{kpoint common}=[draw,fill=white,inner sep=1pt,minimum height=4mm]
\tikzstyle{kpoint sc}=[shape=cornerpoint,kpoint common]
\tikzstyle{kpoint adjoint sc}=[shape=cornercopoint,kpoint common]
\tikzstyle{kpoint}=[shape=cornerpoint,shorten left=5pt,kpoint common]
\tikzstyle{kpoint adjoint}=[shape=cornercopoint,shorten left=5pt,kpoint common]
\tikzstyle{kpoint conjugate}=[shape=cornerpoint,shorten right=5pt,kpoint common]
\tikzstyle{kpoint transpose}=[shape=cornercopoint,shorten right=5pt,kpoint common]
\tikzstyle{kpoint symm}=[shape=cornerpoint,shorten left=5pt,shorten right=5pt,kpoint common]

\tikzstyle{wide kpoint sc}=[shape=cornerpoint,kpoint common, minimum width=1 cm]
\tikzstyle{wide kpointdag sc}=[shape=cornercopoint,kpoint common, minimum width=1 cm]

\tikzstyle{black kpoint}=[shape=cornerpoint,shorten left=5pt,kpoint common,fill=black,font=\color{white}]

\tikzstyle{black kpoint sm}=[shape=cornerpoint,shorten left=5pt,kpoint common,fill=black,font=\color{white},scale=0.75]

\tikzstyle{black kpoint adjoint}=[shape=cornercopoint,shorten left=5pt,kpoint common,fill=black,font=\color{white}]
\tikzstyle{black kpointadj}=[shape=cornercopoint,shorten left=5pt,kpoint common,fill=black,font=\color{white}]

\tikzstyle{black kpointadj sm}=[shape=cornercopoint,shorten left=5pt,kpoint common,fill=black,font=\color{white},scale=0.75]

\tikzstyle{black dkpoint}=[shape=cornerpoint,shorten left=5pt,kpoint common,fill=black, doubled,font=\color{white}]
\tikzstyle{black dkpoint adjoint}=[shape=cornercopoint,shorten left=5pt,kpoint common,fill=black, doubled,font=\color{white}]
\tikzstyle{black dkpointadj}=[shape=cornercopoint,shorten left=5pt,kpoint common,fill=black, doubled,font=\color{white}]

\tikzstyle{black dkpoint sm}=[shape=cornerpoint,shorten left=5pt,kpoint common,fill=black, doubled,font=\color{white},scale=0.75]
\tikzstyle{black dkpointadj sm}=[shape=cornercopoint,shorten left=5pt,kpoint common,fill=black, doubled,font=\color{white},scale=0.75]

\tikzstyle{kpointdag}=[kpoint adjoint]
\tikzstyle{kpointadj}=[kpoint adjoint]
\tikzstyle{kpointconj}=[kpoint conjugate]
\tikzstyle{kpointtrans}=[kpoint transpose]

\tikzstyle{big kpoint}=[kpoint, minimum width=1.2 cm, minimum height=8mm, inner sep=4pt, text depth=3mm]

\tikzstyle{wide kpoint}=[kpoint, minimum width=1 cm, inner sep=2pt]
\tikzstyle{wide kpointdag}=[kpointdag, minimum width=1 cm, inner sep=2pt]
\tikzstyle{wide kpointconj}=[kpointconj, minimum width=1 cm, inner sep=2pt]
\tikzstyle{wide kpointtrans}=[kpointtrans, minimum width=1 cm, inner sep=2pt]

\tikzstyle{wider kpoint}=[kpoint, minimum width=1.25 cm, inner sep=2pt]
\tikzstyle{wider kpointdag}=[kpointdag, minimum width=1.25 cm, inner sep=2pt]
\tikzstyle{wider kpointconj}=[kpointconj, minimum width=1.25 cm, inner sep=2pt]
\tikzstyle{wider kpointtrans}=[kpointtrans, minimum width=1.25 cm, inner sep=2pt]

\tikzstyle{gray kpoint}=[kpoint,fill=gray!50!white]
\tikzstyle{gray kpointdag}=[kpointdag,fill=gray!50!white]
\tikzstyle{gray kpointadj}=[kpointadj,fill=gray!50!white]
\tikzstyle{gray kpointconj}=[kpointconj,fill=gray!50!white]
\tikzstyle{gray kpointtrans}=[kpointtrans,fill=gray!50!white]

\tikzstyle{gray dkpoint}=[kpoint,fill=gray!50!white,doubled]
\tikzstyle{gray dkpointdag}=[kpointdag,fill=gray!50!white,doubled]
\tikzstyle{gray dkpointadj}=[kpointadj,fill=gray!50!white,doubled]
\tikzstyle{gray dkpointconj}=[kpointconj,fill=gray!50!white,doubled]
\tikzstyle{gray dkpointtrans}=[kpointtrans,fill=gray!50!white,doubled]

\tikzstyle{white label}=[draw,fill=white,rectangle,inner sep=0.7 mm]
\tikzstyle{gray label}=[draw,fill=gray!50!white,rectangle,inner sep=0.7 mm]
\tikzstyle{black label}=[draw,fill=black,rectangle,inner sep=0.7 mm]

\tikzstyle{dkpoint}=[kpoint,doubled]
\tikzstyle{wide dkpoint}=[wide kpoint,doubled]
\tikzstyle{dkpointdag}=[kpoint adjoint,doubled]
\tikzstyle{wide dkpointdag}=[wide kpointdag,doubled]
\tikzstyle{dkcopoint}=[kpoint adjoint,doubled]
\tikzstyle{dkpointadj}=[kpoint adjoint,doubled]
\tikzstyle{dkpointconj}=[kpoint conjugate,doubled]
\tikzstyle{dkpointtrans}=[kpoint transpose,doubled]

\tikzstyle{kscalar}=[kpoint common, shape=EBox, inner xsep=-1pt, inner ysep=3pt,font=\small]
\tikzstyle{kscalarconj}=[kpoint common, shape=WBox, inner xsep=-1pt, inner ysep=3pt,font=\small]

\tikzstyle{spekpoint}=[kpoint sc,minimum height=5mm,inner sep=3pt]
\tikzstyle{spekcopoint}=[kpoint adjoint sc,minimum height=5mm,inner sep=3pt]

\tikzstyle{dspekpoint}=[spekpoint,doubled]
\tikzstyle{dspekcopoint}=[spekcopoint,doubled]

 \tikzstyle{upground}=[circuit ee IEC,thick,ground,rotate=90,scale=2.5]
 \tikzstyle{downground}=[circuit ee IEC,thick,ground,rotate=-90,scale=2.5]
 \tikzstyle{bigground}=[regular polygon,regular polygon sides=3,draw=gray,scale=0.50,inner sep=-0.5pt,minimum width=10mm,fill=gray]

\tikzstyle{arrs}=[-latex,font=\small,auto]
\tikzstyle{arrow plain}=[arrs]
\tikzstyle{arrow dashed}=[dashed,arrs]
\tikzstyle{arrow bold}=[very thick,arrs]
\tikzstyle{arrow hide}=[draw=white!0,-]
\tikzstyle{arrow reverse}=[latex-]
\tikzstyle{cdnode}=[]

\tikzstyle{discarding}=[fill=white, draw=black, shape=circle, style=upground]
\tikzstyle{smalldiscarding}=[fill=white, draw=black, style=upground, scale=0.5]
\tikzstyle{backdiscard}=[fill=white, draw=black, shape=circle, style=downground, scale=0.5]
\tikzstyle{smallbackdiscard}=[fill=white, draw=black, shape=circle, style=downground, scale=0.5]
\tikzstyle{state}=[fill=white, draw=black, style=triang, tikzit shape=rectangle]
\tikzstyle{kstate}=[fill=white, draw=black, style=kpoint, tikzit shape=rectangle]
\tikzstyle{kstateconj}=[fill=white, draw=black, style=kpoint conjugate, tikzit shape=rectangle]
\tikzstyle{kstateBIG}=[fill=white, draw=black, style=big kpoint, tikzit shape=rectangle]
\tikzstyle{effect}=[fill=white, draw=black, style=triangdag]
\tikzstyle{keffect}=[fill=white, draw=black, style=kpoint adjoint]
\tikzstyle{keffectconj}=[fill=white, draw=black, style=kpoint transpose]
\tikzstyle{morphdag}=[style=mapdag]
\tikzstyle{morph}=[style=hadamard]
\tikzstyle{WIDEmorph}=[style=hadamard, minimum width=14mm]
\tikzstyle{morphtrans}=[style=maptrans]
\tikzstyle{morphconj}=[style=mapconj]
\tikzstyle{CPMmorph}=[style=dmap]
\tikzstyle{CPMmorphconj}=[style=dmapconj]
\tikzstyle{CPMmorphdag}=[style=dmapdag]
\tikzstyle{CPMmorphtrans}=[style=dmaptrans]
\tikzstyle{CPMstate}=[fill=white, draw=black, style=triang, doubled]
\tikzstyle{CPMstateBIG}=[fill=white, draw=black, style={triang_lesssep}, doubled]
\tikzstyle{CPMkstate}=[fill=white, draw=black, style=kpoint, tikzit shape=rectangle, doubled]
\tikzstyle{CPMkstateconj}=[fill=white, draw=black, style=kpoint conjugate, tikzit shape=rectangle, doubled]
\tikzstyle{CPMkstateBIG}=[fill=white, draw=black, style=big kpoint, tikzit shape=rectangle, doubled]
\tikzstyle{CPMkeffect}=[fill=white, draw=black, style=kpoint adjoint, doubled]
\tikzstyle{CPMkeffectconj}=[fill=white, draw=black, style=kpoint transpose, doubled]
\tikzstyle{UHfB}=[fill=white, draw=black, style=triangdag, doubled, inner sep=-2pt]
\tikzstyle{leak}=[style=tinypoint, regular polygon rotate=-90]
\tikzstyle{leakfill}=[style=tinypoint, regular polygon rotate=-90, fill=black]
\tikzstyle{Z}=[style=dot, fill=green]
\tikzstyle{X}=[style=dot, fill=red]
\tikzstyle{black_dot}=[style=dot, fill=black]
\tikzstyle{white_dot}=[style=dot, fill=white]
\tikzstyle{qblack_dot}=[style=ddot, fill=black]
\tikzstyle{qwhite_dot}=[style=ddot, fill=white]
\tikzstyle{whitephase}=[style=wphase dot, fill=white]
\tikzstyle{qredphase}=[style=phase dot, fill=red]
\tikzstyle{qgreenphase}=[style=phase dot, fill=green]
\tikzstyle{had}=[style=hadamard, doubled]
\tikzstyle{box}=[style=hadamard]
\tikzstyle{classhad}=[style=hadamard]
\tikzstyle{antipode}=[style=anti]

\tikzstyle{dottededge}=[-, dotted]
\tikzstyle{double edge}=[-, style=doubled, draw=black, tikzit draw={rgb,255: red,18; green,168; blue,191}]
\tikzstyle{new edge style 0}=[<-]
\tikzstyle{new edge style 1}=[-, draw={rgb,255: red,150; green,186; blue,125}, fill={rgb,255: red,150; green,186; blue,125}]
\tikzstyle{new edge style 2}=[-, draw={rgb,255: red,66; green,5; blue,188}]
\tikzstyle{new edge style 3}=[<-, draw={rgb,255: red,223; green,66; blue,126}]
\tikzstyle{new edge style 4}=[<-, draw={rgb,255: red,0; green,128; blue,128}]
\tikzstyle{new edge style 5}=[-, draw={rgb,255: red,227; green,64; blue,9}]
\tikzstyle{new edge style 6}=[-, draw={rgb,255: red,174; green,20; blue,174}]
\tikzstyle{new edge style 7}=[<-]
\tikzstyle{new edge style 8}=[<-, draw={rgb,255: red,227; green,64; blue,9}]

\usetikzlibrary{calc, positioning, shapes.geometric}
\usetikzlibrary{
	arrows,
	shapes,
	decorations,
	intersections,
	backgrounds,
	positioning,
	circuits.ee.IEC
	}

\newcommand{\tr}{\mathrm}

\newcommand{\tikzfigscale}[2]{\scalebox{#1}{\input{#2.tikz}}}

\def\be{\begin{equation}}
\def\ee{\end{equation}}
\def\ba{\begin{align}}
\def\ea{\end{align}}

\newcommand{\cat}[1]{\ensuremath{\mathbf{#1}}\xspace}
\newcommand{\FHilb}{\cat{FHilb}}
\newcommand{\Rel}{\cat{Rel}}

\newcommand{\CPM}{\ensuremath{\mathrm{CPM}}\xspace}

\newcommand{\id}[1][]{\ensuremath{1_{#1}}}

\newcommand{\ce}{\mathcal E}

\newcommand{\ch}{\mathcal H}

\newcommand{\cl}{\mathcal L}

\newcommand{\cp}{\mathcal P}

\def\tl{\tilde}

\usepackage{bbm}

\newcommand{\lalg}[1]{\mathfrak{#1}}  

\renewcommand{\sl}{\lalg{sl}}

\newcommand{\discard}{ 
\mathbin{ \begin{tikzpicture}[scale=0.85, {every node/.style}={scale=0.4},baseline=0.3mm]
	\begin{pgfonlayer}{nodelayer}
		\node [style=smalldiscarding] (140) at (0, 0.25) {};
		\node [style=none] (143) at (0, 0) {};
	\end{pgfonlayer}
	\begin{pgfonlayer}{edgelayer}
		\draw [style=double edge] (140) to (143.center);
	\end{pgfonlayer}
\end{tikzpicture}}  }

\newcommand{\identity}{ 
\mathbin{ \begin{tikzpicture}[scale=0.85, {every node/.style}={scale=0.4},baseline=0.3mm]
	\begin{pgfonlayer}{nodelayer}
		\node [style=none] (140) at (0, 0.325) {};
		\node [style=none] (143) at (0, 0) {};
	\end{pgfonlayer}
	\begin{pgfonlayer}{edgelayer}
		\draw [style=double edge] (140.center) to (143.center);
	\end{pgfonlayer}
\end{tikzpicture}}  }

\def\tl{\tilde}

\newtheorem{definition}{Definition}
\newtheorem{theorem}{Theorem}

\newtheorem{lemma}{Lemma}

\newtheorem{example}{Example}

\tikzstyle{every picture}=[baseline=-0.25em,shorten <=-0.1pt]
\tikzstyle{dotpic}=[scale=0.5]
\tikzstyle{braceedge}=[decorate,decoration={brace,amplitude=1mm,raise=-1mm}]
\tikzstyle{dot}=[inner sep=0.7mm,minimum width=0pt,minimum height=0pt,fill=black,draw=black,shape=circle]
\tikzstyle{small dot}=[inner sep=0.1mm,minimum width=0pt,minimum height=0pt,fill=black,draw=black,shape=circle]
\tikzstyle{black dot}=[dot]
\tikzstyle{white dot}=[dot,fill=white]
\tikzstyle{gray dot}=[dot,fill=gray!40!white]
\tikzstyle{alt white dot}=[white dot,label={[xshift=3mm,yshift=-0.05mm,font=\tiny]left:$*$}]
\tikzstyle{alt gray dot}=[gray dot,label={[xshift=3mm,yshift=-0.05mm,font=\tiny]left:$*$}]
\tikzstyle{white norm}=[rectangle,fill=white,draw=black,minimum height=2mm,minimum width=2mm,inner sep=0pt,font=\small]
\tikzstyle{gray norm}=[white norm,fill=gray!40!white]
\tikzstyle{square box}=[rectangle,fill=white,draw=black,minimum height=5mm,minimum width=5mm,font=\small]
\tikzstyle{square gray box}=[rectangle,fill=gray!30,draw=black,minimum height=6mm,minimum width=6mm]
\tikzstyle{diredge}=[->]
\tikzstyle{rdiredge}=[<-]
\tikzstyle{dashed edge}=[dashed]
\tikzstyle{cross}=[preaction={draw=white, -, line width=3pt}]

\newcommand{\dotdualmult}[1]{%
\!\begin{tikzpicture}[dotpic]
    \node [style=white dot] (0) at (0, 0.3) {};
    \node [style=none] (1) at (-0.5, -0.4) {};
    \node [style=none] (2) at (0.5, -0.4) {};
    \node [style=none] (3) at (0, 0.8) {};
    \draw [style=diredge] (3.center) to (0);
    \draw [style=diredge, in=15, out=-30, looseness=1.50] (0) to (1.center);
    \draw [style=diredge, in=165, out=-150, looseness=1.50] (0) to (2.center);
\end{tikzpicture}\!}

\newcommand{\dotconorm}[1]{%
\,\begin{tikzpicture}[dotpic,yshift=0.4mm]
    \node [style=none] (0) at (0, -0.4) {};
    \node [style=white norm] (1) at (0, 0.1) {};
    \node [style=none] (2) at (0, 0.5) {};
    \draw [style=diredge] (1) to (0.center);
    \draw (2.center) to (1);
\end{tikzpicture}\,}

\newcommand{\astfootnote}[1]{
\let\oldthefootnote=\thefootnote
\setcounter{footnote}{0}
\renewcommand{\thefootnote}{\fnsymbol{footnote}}
\footnote{#1}
\let\thefootnote=\oldthefootnote
}

\title{Composable constraints}
\author{Matt Wilson}
\email{matthew.wilson@cs.ox.ac.uk}
\affiliation{Quantum Group, Department of Computer Science, University of Oxford}
\affiliation{HKU-Oxford Joint Laboratory for Quantum Information and Computation}
\thanks{Both authors contributed equally to this paper; their ordering was decided by the measurement of a superposed state, using the IBM Q machine. In the absence of any access to the IBM Q machine, this measurement was in turn simulated via a classical coin flip, using a one-pound coin.}
	
\author{Augustin Vanrietvelde}
\email{a.vanrietvelde18@imperial.ac.uk}
\affiliation{Quantum Group, Department of Computer Science, University of Oxford}
\affiliation{HKU-Oxford Joint Laboratory for Quantum Information and Computation}
\affiliation{Blackett Laboratory, Department of Physics, Imperial College London}

\begin{document} \emergencystretch 3em

\maketitle

\begin{abstract}
We introduce a notion of compatibility between constraint encoding and compositional structure. Phrased in the language of category theory, it is given by a \textit{composable constraint encoding}. We show that every composable constraint encoding can be used to construct an equivalent notion of a \textit{constrained category} in which morphisms are supplemented with the constraints they satisfy. We further describe how to express the compatibility of constraints with additional categorical structures of their targets, such as parallel composition, compactness, and time-symmetry. We present a variety of concrete examples. Some are familiar in the study of quantum protocols and quantum foundations, such as signalling and sectorial constraints; others arise by construction from basic categorical notions. We use the language developed to discuss the notion of intersectability of constraints and the simplifications it allows for when present, and to show that any time-symmetric theory of relational constraints admits a faithful notion of intersection.
\end{abstract}

\tableofcontents

\section{Introduction}

A constraint on a thing is a statement that it is in fact a member of some restricted set of things. Whenever the thing lives in a theory with some structure, it is then natural to ask about the compatibility of a given method of constraint encoding with that structure. This paper is concerned with developing a language for talking about such a compatibility in the specific case where the structure at hand is compositional.

A prevalent mathematical language for the expression of compositional structure is that of category theory. Consequently, our goal is to provide a general notion on a category \cat{C} that 1) captures the possibility to impose \textit{constraints} on morphisms of \cat{C}, and 2) allows to compose these constraints, in a way that is compatible with the composition of the morphisms of \cat{C}. By a set of constraints on possible maps $A \overset{f}{\to} B$ in \cat{C}, we mean a piece of data $\lambda$ singling out, among the hom-set $\cat{C}(A,B)$, a subset $\cat{C}_\lambda(A,B)$, whose elements are said to be the maps that \textit{satisfy} these constraints. Now, given two constraints $\lambda$ and $\sigma$ on possible maps $B \overset{g}{\to} C$, we are concerned with the case in which there is a sense in which $\sigma$ and $\lambda$ can be composed to form a set of constraints `$\sigma \circ \lambda$' that is satisfied by $g \circ f$ when $g$ satisfies $\sigma$ and $f$ satisfies $\lambda$, i.e.,
\be \tikzfigscale{1.0}{scheme_1} . \ee 
As we shall see, formally capturing this kind of `structure of constraints', compatible with the structure of morphisms themselves, helps to model scenarios in which constraints have to be taken into account, and allows to pin down the structural features of different types of constraints, defining them in a general way. It can also be used to unlock a handy `constraint calculus': a calculus only performed on the constraints that morphisms satisfy, allowing to deduce properties about their composition while `bypassing' the handling of the (usually more intricate) data about the morphisms themselves.

The need for a formal theory of the compositionality of constraint encoding has recently arisen in a variety of contexts. First, constraints appear in the description of physical, communicational or computational scenarios in which some key operations can be freely chosen, yet can only be picked among a subset of the possible operations due to restrictions arising e.g. from physical constraints, rules of a game, or technological limitations. This is for instance the case for the study of superpositions of channels in quantum theory \cite{chiribella2019shannon,kristjansson2020singleparticle,kristjansson2020resources} -- and more generally for that of the coherent control of gates and channels \cite{aaberg2004subspace, aaberg2004operations, oi2003interference, araujo2014, friis2014, Friis_2015, Dunjko_2015, thompson2018, dong2019controlled, abbott2020communication}  --, a notion whose formal definition is a subtle matter, and for which a recently proposed formalism \cite{vanrietvelde2020routed, vanrietvelde2021coherent} makes a crucial use of so-called \textit{sectorial constraints} on morphisms. In another line of research \cite{Allen2017, barrett2019, lorenz2020,barrett2020cyclic}, it has been proposed to describe the causal structure of unitary channels using sets of `no-influence relations' used to encode constraints which forbid certain input factors of channels from signalling to certain output factors. In each of these cases, the constraints imposed are encoded in some way by relations as depicted in the following diagrams:
\begin{equation}\label{eq:sector1}
\tikzfigscale{1}{sector_signal} \,.
\end{equation}
Moreover, when processes satisfy such constraints, their compositions are guaranteed to satisfy the constraints encoded by the composition of the relations encoding those constraints. One may wonder whether there exists a formal notion of composable constraint which encompasses them both, and if there might exist a single construction method from which either can be produced. We will find that both questions can be answered positively. A further notion, familiar in theoretical computer science, is that of a constraint satisfaction problem. Given a pair of functions, each of which presents a solution to a constraint satisfaction problem, one may ask about the existence of a constraint satisfaction problem which is guaranteed to be solved by the composition of those functions.

Furthermore, introducing constraints can be viewed as a way of enriching the structure of a given category, in order to make it expressive enough to capture some notions in an elegant and consistent way; this has been the case in the study of so-called \textit{causal decompositions}, i.e. diagrammatic decompositions of unitary channels that are equivalent to these channels' causal structure \cite{Allen2017, barrett2019, lorenz2020,barrett2020cyclic} (see in particular \cite{lorenz2020}). Indeed, some causal decompositions cannot be written in terms of standard circuits, but only using more elaborate circuits (later called \textit{index-matching circuits} \cite{vanrietvelde2020routed}), which relied on constraints and whose exact semantics remained unclear\footnote{See Appendix B of \cite{vanrietvelde2020routed} for a discussion of why other standard categorical constructions, CP*[\FHilb] and Karoubi[CPM[\FHilb]], cannot be used either to model superpositions of paths and causal decompositions.}. 

In order to illustrate in more detail an example of the kind of composable constraint behaviour we want to describe and give some intuition, let us quickly present the study of sectorial constraints. The point of Ref.\ \cite{vanrietvelde2020routed} is, in the context of the category \FHilb of linear maps on finite-dimensional Hilbert spaces, to build a theory encompassing \textit{sectorial constraints} on these linear maps. Sectorial constraints express the fact that a given linear map is forbidden to relate some sectors (i.e. orthogonal subspaces) of its domain with some sectors of its codomain. For example, the set of links in the left of the following figure expresses a set of sectorial constraints on a map $f \in \FHilb(A,B)$, where $A$ and $B$ are partitioned into a direct sum of sectors:

\begin{equation}\label{eq:sector1}
\tikzfigscale{1}{scheme_2} \, .
\end{equation}

In this figure, the sectorial constraints correspond to the absence of links between some sectors. For instance, the absence of links from $A_1$ to $B_2$, $B_3$ and $B_4$ means that, for a $f$ following these constraints, one has $f(A_1) \subseteq B_1$; the same goes with the other sectors.

However, as we already mentioned, when we talk about a framework `including constraints', we do not just mean to allow for the possibility to add these constraints `by hand' on some morphisms; what we want is a theory in which one can \textit{compose} these constraints in various ways, compatibly with the structure of the original category. An example will make this point clearer. For any two linear maps $A \overset{f}{\to} B \overset{g}{\to} C$ following sectorial constraints given by
\begin{equation}\label{eq:sector1}
\tikzfigscale{1}{scheme_2_1} \,,
\end{equation}
it is easy to deduce a set of sectorial constraints which $g \circ f$ will necessarily follow by constructing a link from each connected path, as depicted in the following figure:
\begin{equation}\label{eq:sector1}
\tikzfigscale{1}{scheme_3} \, .
\end{equation}


In other words, the constraints themselves feature some structure (here, a composition), and, crucially, this structure is compatible with that of the underlying category: if two maps each follow a set of constraints, then their composition follows the composition of these sets of constraints. In fact, sectorial constraints exhibit a whole dagger compact structure -- that of finite relations --, completely consistent with the dagger compact structure of \FHilb. It is this kind of structural compatibility that we want to describe and exploit fully, whenever some notion of constraints features it.

\cite{vanrietvelde2020routed} provided a well-defined formal account of sectorial constraints; the structural and categorical features of this account, however, were toned down in order to make it suitable to working physicists. The framework we will present here aims to be a way more general, and fully structural, account, able to model the inclusion of, and reasoning about, constraints on morphisms in various theoretical contexts. After formalising the notion of a composable constraint as a lax functor $\mathcal{L}$ with a particular domain, we show that any such  lax functor can be used to construct a new \textit{constrained} category $\cat{C}_{\mathcal{L}}$ in which morphisms and their constraints are manipulated together as pairs $(\lambda, f)$ in which $\lambda$ (in \cat{Con}) is called the \textit{constraint}, and $f$ (in \cat{C}) satisfies the constraint $\lambda$ -- in the sense that $f \in \mathcal{L}(\lambda)$. Defining all operations pairwise, suitably \textit{monoidal} (or $\dagger$ or \textit{compact closed}, etc.) properties of such a lax-functor $\mathcal{L}$ lift to properties of the constrained category. In the constrained category, the whole calculus is thus doubled, and performed in parallel in \cat{C} on the one hand, and in \cat{Con} on the other hand, giving a circuit (or string, etc.) language for morphisms equipped with constraints:

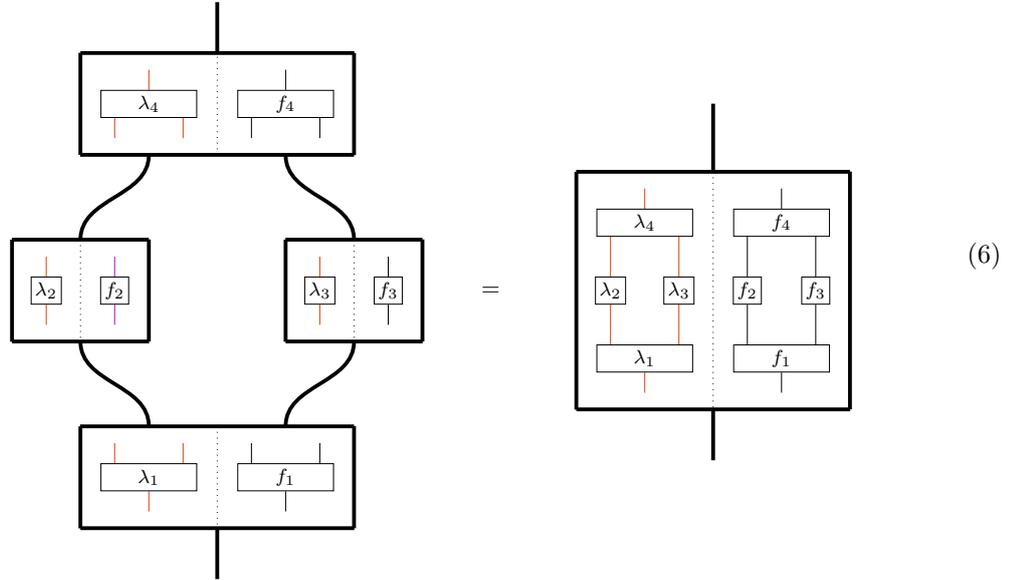
\be \tikzfigscale{0.9}{consistent}\ee

The structure of this paper is as follows. 
\begin{itemize}
    \item (Section \ref{sec:definitions}) First, we present a very general notion of composable (monoidal) constraint on a category and how it leads to the construction of constrained (monoidal) categories. We formalise the equivalence between the former and the latter concepts within the language of category theory.
    \item (Section \ref{sec:examples}) After presenting examples of \textit{thin} constraints, familiar in standard category theory, we then observe that the two motivating instances of constraints in quantum theory, those of \textit{sectorial constraints} and \textit{signalling constraints} are in fact instances of applying the same unified construction of \textit{relational constraints} to the direct sum $\oplus$ and the tensor product $\otimes$ respectively. The features of relational constraints, and in particular the consequences of time-symmetry for them, are then explored.
    \item (Section \ref{sec:Intersectable Constraints}) We discuss the important structural feature of \textit{intersectability} in constraint theory, showing that in the abstract it leads to practical simplifications for constraint reasoning, and that for relational constraints it is in fact entailed by time-symmetry.
    \item (Section \ref{sec:General Constructions}) We finally present more abstract and general constructions leading to theories of constraints: a generalised approach to relational constraints, and a category of constraint satisfaction problems.
\end{itemize}

\section{Composable Constraint Encoding} \label{sec:definitions}
In this section we present the concept of a composable constraint encoding from a category $\mathbf{Con}$ into a category $\cat{C}$. Our goal is to capture the two basic features of composable constraint encoding in the abstract:
\begin{itemize}
    \item if $f$ satisfies $\tau$ and $g$ satisfies $\lambda$ then $f \circ g$ satisfies $\tau \circ \lambda$: \be \tikzfigscale{1.0}{scheme_1} \,;\ee
    \item if $\tau$ is a `stronger' constraint than $\lambda$, then whenever $f$ satisfies $\tau$ it must furthermore be true that $f$ satisfies $\lambda$: \be \tikzfigscale{0.9}{scheme_4} \,. \ee
\end{itemize}
The notion of such a constraint encoding can be captured by the following definition.
\begin{definition}[Composable constraint encoding]
A composable constraint encoding for a category $\cat{C}$ is a lax functor $\mathcal{L}: \mathbf{Con} \longrightarrow \mathcal{P}[\cat{C}]$ from a locally posetal $2$-category $\mathbf{Con}$ into the power set category $ \mathcal{P}[\cat{C}].$ 
\end{definition}
We now unpack this definition, demonstrating that it indeed captures the two prior concepts. Firstly, from a category $\cat{C}$ one may construct a power-set category $\mathcal{P}[\cat{C}]$ into which constraints will be encoded. The point of the power set category (often referred to in the literature as the free quantaloid) is that its morphisms are subsets of morphisms, and so explicitly represent the most primitive notion of a constraint:
\begin{definition}[Power-Set Category]
For any category $\cat{C}$ the category $\mathcal{P}[\cat{C}]$ has objects given by the objects of $\cat{C}$ and for morphisms \be \mathcal{P}[\cat{C}](A,B) := \{S| \quad S \subseteq \cat{C}(A,B)\} \,,\ee  with the identity given by $\{id\}$ and the composition given by \be S \circ T := \{f \circ g| \quad  (f,g) \in S \times T\} \,.\ee 
$\mathcal{P}[\cat{C}]$ may furthermore be viewed as a $2$-Category (more precisely, a locally posetal 2-category) with $2$-Morphisms given by subset inclusion:
\begin{align*}
    \mathcal{P}[\cat{C}](S,T) := & \textrm{ } \{\subseteq \} \textrm{ if } S \subseteq T \,, \\
    & \textrm{  } \emptyset \quad otherwise.
\end{align*}
\end{definition}
Morphisms of $\cp[\mathbf{C}]$ may be represented with string diagrams such as
\be \tikzfigscale{1.0}{power_1} , \ee 
in terms of which an example of the composition rule is given by
\be \tikzfigscale{1.0}{power_2} . \ee 
For a category $\mathbf{Con}$ to encode constraints for $\mathbf{C}$, there should exist an encoding $\mathcal{L}(-)$ which maps each morphism $\tau$ of $\mathbf{Con}$ into a morphism $\mathcal{L}(\tau)$ of $\cp[\mathbf{C}]$, that is, to a subset $\mathcal{L}(\tau)$ of the possible morphisms of $\mathbf{C}$. In string diagrammatic language one may represent the action of $\mathcal{L}$ on the morphism $\tau$ by \be \tikzfigscale{1.0}{scheme_lax} , \ee anticipating that the shaded box will have properties similar to those of functor boxes \cite{functor_boxes}. Taking $f \in \mathcal{L}(\tau)$ to mean `$f$ satisfies $\tau"$  it is then expected that whenever $f \in \mathcal{L}(\tau)$ and $g \in \mathcal{L}(\lambda)$ then $f \circ g$ (which is for sure a member of $\mathcal{L}(\tau) \circ \mathcal{L}(\lambda)$ in $\mathcal{P}[\mathbf{C}]$) should be a member of $\mathcal{L}(\tau \circ \lambda)$. Secondly, it is expected that $\mathbf{Con}$ should be equipped with a notion of relaxation $\tau \leq \lambda$ expressing that the constraint $\tau$ is stronger than the constraint $\lambda$. For such a relaxation to be meaningful, the encoding $\mathcal{L}$ should preserve it, so that being a member of $\mathcal{L}(\tau)$ is a stronger condition than being a member of $\mathcal{L}(\lambda)$. Any lax functor \be \mathcal{L}: \mathbf{Con} \longrightarrow \mathcal{P}[\cat{C}] \ee  from a locally posetal $2$- category $\mathbf{Con}$ into $\mathcal{P}[\mathbf{C}]$ formally expresses each of these notions in the language of category theory\footnote{The authors are extremely grateful for the insight of an anonymous reviewer at ACT 2021 who pointed out the phrasing of composable constraints via lax-ness.}; indeed, a lax-functor  of locally posetal $2$-categories with codomain $\cp[\mathbf{C}]$ is precisely a mapping such that:
\begin{itemize}
    \item $\mathcal{L}(\tau) \circ \mathcal{L}(\lambda) \subseteq \mathcal{L}(\tau \circ \lambda) $, written graphically by box-merging: \be \tikzfigscale{1.0}{scheme_5} \,; \ee 
    \item $\tau \leq \lambda \implies \mathcal{L}(\tau) \subseteq \mathcal{L}(\lambda) \,.$
\end{itemize}
Note in particular that if $f \in \mathcal{L}(\tau)$ and $g \in \mathcal{L}(\lambda)$, then the composition $f \circ g \in \mathcal{L}(\tau) \circ \mathcal{L}(\lambda) \subseteq \mathcal{L}(\tau \circ \lambda)$. Graphically this can be seen by the following steps: \be \tikzfigscale{1.0}{scheme_5_1} \,. \ee


\subsection{Constrained categories}
From any composable constraint encoding (from now on termed composable constraint for brevity) $\mathcal{L}$, a category $\mathbf{C}_{\mathcal{L}}$ can be constructed in which the constraints of $\mathbf{Con}$  and the morphisms of $\cat{C}$  satisfying them live together, so that the diagrams of $\mathbf{C}_{\mathcal{L}}$ immediately capture a side-by-side calculus of constraints and morphisms.
\begin{definition}[Category of Constrained Morphisms]
Let $\mathcal{L}: \mathbf{Con} \longrightarrow \mathcal{P}[\cat{C}]$ be a composable constraint. The constrained category $\cat{C}_{\mathcal{L}}$ has as objects the objects of $\mathbf{Con} \times \mathbf{C}$ of the form $(a,\mathcal{L}(a))$ and as hom-sets \be \cat{C}_{\mathcal{L}}(a,b) := \{(\tau,f)| f \in \mathcal{L}(\tau:a \rightarrow b)\}\,,\ee  with composition defined component-wise by $(\tau,f) \circ (\lambda,g) := (\tau \circ \lambda, f \circ g)$; the latter morphism is well-defined since by laxity $f \in \mathcal{L}(\tau) $ and $ g \in \mathcal{L}(\lambda) \implies f \circ g \in \mathcal{L}(\tau \circ \lambda)$. The identity is given by $(id,id)$, which is well-defined since $\{id\} \subseteq \mathcal{L}(id)$. Taking the trivial identity-only locally posetal $2$- structure on the morphisms of the category $\mathbf{C}$, $\mathbf{C}_{\mathcal{L}}$ can furthermore be understood as a sub locally posetal category of $\mathbf{Con} \times \mathbf{C}$ where the $2$-morphisms of $\mathbf{C}_{\mathcal{L}}$ are defined by \be (\tau,f) \leq (\lambda,g) \iff \tau \leq \lambda \textrm{ and } f = g\ee 
\end{definition}
As will be expanded upon when circuit compatibility is discussed, the side-by-side calculus is essentially a practical product, in which a circuit theory has been enriched with constraint structure. Graphically, morphisms in the constrained category and their composition can be represented by
\be \tikzfigscale{1.0}{scheme_6} \,. \ee An immediate question is whether such a presentation of a constraint is equivalent to the notion from which it was constructed\footnote{The authors are grateful to James Hefford for asking this question.}.

\subsection{Equivalence between composable constraints and constrained categories}
Constrained categories have three key properties altogether which can be used to define them explicitly. In the following, $\mathbf{C}$ is identified with the trivial 2-category in which each homset $\mathbf{C}(A,B)$ is viewed as a discrete category.
\begin{definition}
A constrained category is a sub locally posetal category $\mathbf{M} \subseteq \mathbf{Con}\times \mathbf{C}$ equipped with a function $M$ such that:
\begin{itemize}
    \item every object $(x,y)$ of $\mathbf{M}$ has the form $(x,M(x))$;
    \item for every pair of objects $w,z$ the induced inclusion functor $\mathbf{M}(w,z) \subseteq_{w,z} \mathbf{Con}\times \mathbf{C}(w,z)$ is a discrete opfribration.
\end{itemize}
\end{definition}

These conditions are sufficient to recover the composable constraint from which a constrained category was constructed. This equivalence between constrained categories and composable constraints is conveniently expressed as an equivalence between the category of composable constraints and the category of constrained categories; we thus begin by defining those two categories, using the notation $\mathcal{L}:\mathbf{X} \rightsquigarrow \mathbf{Y}$ for a lax functor $\mathcal{L}:\mathbf{X} \rightarrow \cp[\mathbf{Y}]$.
\begin{definition}
The category $\mathbf{compCON}$ of composable constraints is given by taking, for any pair $\mathcal{L}:\mathbf{X} \rightsquigarrow \mathbf{Y}$ and $\mathcal{M}:\mathbf{Y} \rightsquigarrow \mathbf{Z}$, $\mathcal{M} \square \mathcal{L} (x) := \mathcal{M}(\mathcal{L}(x))$ on objects and $\mathcal{M} \square \mathcal{L}(\tau) = \cup_{\kappa \in \mathcal{L}(\tau)} \mathcal{M}(\kappa)$ on morphisms. This composition has identity $\mathcal{I}:\mathbf{X} \rightsquigarrow \mathbf{X}$ given by $\mathcal{I}(x) = x$ and $\mathcal{I}(f) = \{f\}$,  and the composition $\square$ is associative since 
\begin{align*}
    \mathcal{N} \square (\mathcal{M} \square \mathcal{L}) (\tau) :=  \cup_{\nu \in \mathcal{M} \square \mathcal{L}(\tau)} \mathcal{N}(\nu) \\
    = \cup_{\nu \in \cup_{\kappa \in \mathcal{L}(\tau)} \mathcal{M}(\kappa)} \mathcal{N}(\nu) \\
    = \cup_{\kappa \in \mathcal{L}(\tau)} \cup_{\eta \in \mathcal{M}(\kappa)} \mathcal{N}(\eta) \\
    = \cup_{\kappa \in \mathcal{L}(\tau)} \mathcal{N} \square \mathcal{M} (\mu)\\
    = (\mathcal{N} \square \mathcal{M}) \square \mathcal{L}(\tau)
\end{align*}
On $2$-Morphisms, $\mathcal{M} \square \mathcal{L}$ is defined by taking each morphism $\tau \leq \tau'$ to the inclusion $\mathcal{M} \square \mathcal{L} (\tau) \subseteq \mathcal{M} \square \mathcal{N} (\tau')$, which is guaranteed to exist since
\begin{align*}
    \tau \leq \tau' & \implies \mathcal{L}(\tau) \subseteq \mathcal{L}(\tau') \\
    & \implies \cup_{\kappa \in \mathcal{L}(\tau)} \mathcal{M}(\kappa) \subseteq  \cup_{\kappa \in \mathcal{L}(\tau')} \mathcal{M}(\kappa) \\
    & \implies \mathcal{M} \square \mathcal{L} (\tau) \subseteq \mathcal{M} \square \mathcal{N} (\tau')
\end{align*}
\end{definition}

We now turn to constrained categories. They are essentially a special class of relations internal to $\mathbf{CAT}$, with the relation action at the level of objects being single valued. Unsurprisingly, the composition of constrained categories is essentially the relational composition between sets of morphisms.
\begin{definition}
The category $\mathbf{conCAT}$ of constrained categories has categories as its objects, and morphisms $\mathbf{M}: \mathbf{X} \rightsquigarrow \mathbf{Y}$ given by constrained categories $\mathbf{M} \subseteq \mathbf{X} \times \mathbf{Y}$. Composition is given by taking $\mathbf{N} \triangle \mathbf{M}$ to be the subcategory of $\mathbf{X} \times \mathbf{Z}$ equipped with function $N \circ M$ with morphism $(f,h)$ existing if there exists some $g$ with $(f,g)$ in $\mathbf{M}$ and $(g,h)$ in $\mathbf{N}$. The identity $\mathcal{I}:\mathbf{X} \rightsquigarrow \mathbf{X}$ is given by the diagonal category $\triangle_X$, whose morphisms take the form $(f,f)$. Composition is easily checked to be associative:
\begin{align*}
    & (f,k)\in (N \triangle M) \triangle L[ (w,z),(w',z')]  \\
    \iff & \exists g:x \rightarrow x' \textrm{ s.t:  } (f,g)  \in L[ (w,x),(w',x')] \\ 
    & \wedge  (g,k) \in N \triangle M[(x,z),(x',z')] \\
    \iff & \exists (g,h):(x,y) \rightarrow (x',y') \textrm{ s.t:  } (f,g) \in L [ (w,x),(w',x')] \\
    & \wedge   (g,h) \in M [ (x,y),(x',y')] \\ 
    &\wedge 
     (h,k) \in N[(y,z),(y',z')] \\
     \iff & \exists h:y \rightarrow y' \textrm{ s.t:  } (f,h) \in M \triangle N [ (w,y),(w',y')] \\
    & \wedge   (h,k) \in N [ (y,z),(y',z')] \\
    \iff & (f,k)\in N \triangle (M \triangle L)[ (w,z),(w',z')] \,.
\end{align*}
Finally, the inclusion functor $\subseteq_{xy}$ for the composition $N \triangle M$ is a discrete opfibration since:
\begin{align*}
& (\alpha,f) \in \mathcal{N}  \triangle \mathcal{M} [(x,N(M(x))) , (y,N(M(y)))]  \textrm{ with } \alpha \leq \beta \\
\implies & \textrm{there exists } (\alpha,\lambda) \in \mathcal{M}[(x,M(x)) , (y,(My))] \\
& \textrm{and } (\lambda,f) \in \mathcal{N}[(M(x),N(M(x))) , (M(y),N(M(y)))] \textrm{ with } \alpha \leq \beta \\
\implies &  \textrm{there exists } (\beta,\lambda) \in \mathcal{M}[(x,M(x)) , (y,(My))] \\
& \textrm{and } (\lambda,f)\in \mathcal{N}[(M(x),N(M(x))) , (M(y),N(M(y)))]  \\
\implies & (\beta,f) \in \mathcal{N}  \triangle \mathcal{M} [(x,N(M(x))) , (y,N(M(y)))]
\end{align*}
\end{definition}
The assignment $\mathbf{C}_{\bullet}$ is functorial, with a mutually inverse functor $\mathcal{L}_{\bullet}$.
\begin{theorem}
The assignment $\mathbf{C}_{\bullet}: \mathbf{compCON} \longrightarrow \mathbf{conCAT}$ given on objects by $\mathbf{C}_{\bullet}(A) = A$ and on morphisms by $\mathbf{C}_{\bullet}(\mathcal{L}) = \mathbf{C}_{\mathcal{L}}$ is a functor with a strict inverse.
\end{theorem}
\begin{proof}
First we prove functoriality: 
\begin{align*}
    & (f,h) \in \mathbf{C}_{\mathcal{M}} \square \mathbf{C}_{\mathcal{L}}[(x,z),(x',z')] \\
    \iff & \exists g \in Y(y,y') \textrm{ s.t } (f,g) \in \mathbf{C}_{\mathcal{L}}[(x,y),(x',y')] \wedge (g,h) \in \mathbf{C}_{\mathcal{M}}[(y,z),(y',z')]\\
    \iff & (f,h) \in \bigcup_{g \in \mathcal{L}(f)} \mathcal{M}(g) \\
    \iff & (f,h) \in \mathbf{C}_{\mathcal{M} \triangle \mathcal{L}} \,.
\end{align*}
It was previously shown that the inclusion functors $\subseteq_{xy}$ for each category $\mathbf{C}_{\mathcal{L}}$ are always discrete opfibrations.
Conversely, each constrained category can be used to define a composable constraint: $\mathcal{L}_{\mathbf{C}}$ is defined on objects by taking $\mathcal{L}_{\mathbf{C}}(x) := \mathbf{C}(x)$, and on morphisms by taking \be \mathcal{L}_{\mathbf{C}}(\tau:x \rightarrow x') := \{f: (\tau,f) \in \mathbf{C}((x,C(x)),(y,C(y)))\} \,. \ee 
Lax functoriality of $\mathcal{L}_{\mathbf{C}}$ on 1-morphisms is guaranteed by the fact the $\mathbf{C}$ is a category. On $2$-morphisms, strict composition preservation is derived from each inclusion $\subseteq_{xy}$ being a discrete opfibration: if $\tau \leq \lambda$, then 
\begin{align*}
    f \in \mathcal{L}_{\mathbf{C}} (\tau:x \rightarrow x') & \iff (\tau,f) \in \mathbf{C}[(x,C(x),(y,C(y))] \\
    & \iff (\lambda,f) \in \mathbf{C}[(x,C(x),(y,C(y))] \\
    & \iff f \in \mathcal{L}_{\mathbf{C}}(\lambda:x \rightarrow x') \,.
\end{align*}
It is easy to see that $\mathcal{L}_{\bullet}$ and $\mathbf{C}_{\bullet}$ are mutually inverse, and so together define a strict equivalence of categories.
\end{proof}
The key takeaway from the above is that one can equivalently view compositonal constraints as either lax functors, or categories where constraints and the morphisms they constrain live in parallel.

\subsection{Full circuit compatibility for constraints}

In order to consider the kind of full circuit compatibility which will for example be observed for sectorial and signalling constraints in quantum circuit theory, and in order to construct constrained circuit calculi, one must take a further step and consider the notion of compatibility with parallel composition (monoidal structure). A monoidal composable constraint is in essence a composable constraint such that, whenever $f$ satisfies $\tau$ and $f'$ satisfies $\tau'$, then $f \otimes f' $ satisfies $\tau \otimes \tau'$;

\be \tikzfigscale{1.0}{scheme_7} . \ee

This basic principle for monoidal compatibility of constraint encoding can be written concisely as

\be (f,f') \in \mathcal{L}(\tau) \times \mathcal{L}(\tau') \implies f \otimes f' \in \mathcal{L}(\tau \otimes \tau') \,.\ee

For any monoidal category $\cat{C}$ the category $ \mathcal{P}[\cat{C}]$ can be viewed as a monoidal category with $S \otimes T := \{f \otimes g | (f,g) \in S \times T\} $, and with coherence isomorphisms inherited as singletons, e.g. $\alpha_{\mathcal{P}[\cat{C}]} := \{\alpha_{\cat{C}}\}$. A graphical example of parallel composition in $\cp[\mathbf{C}]$ is given by \be \tikzfigscale{1.0}{power_3} . \ee From this definition of the monoidal product, it follows that whenever $S \subseteq S'$ and $T \subseteq T'$ then $S \otimes T \subseteq S' \otimes T'$. It is tedious but simple to check that $ \mathcal{P}[\cat{C}]$ in fact forms a locally posetal monoidal $2$-Category \cite{Day1997MonoidalAlgebroids, shulman2010constructing}. The monoidal compatibility condition is then equivalent to requiring that $\mathcal{L}(\tau) \otimes \mathcal{L}(\lambda) \subseteq \mathcal{L} (\tau \otimes \lambda)$. We furthermore require that each coherence morphism such as $\lambda$, being morally the identity, should like the identity satisfy the corresponding constraint $\mathcal{L}(\lambda)$. These concepts, when relaxed to allow for the preservation of monoidal products on objects only-up-to isomorphism $\mathcal{L}(A \otimes B) \cong \mathcal{L}(A) \otimes \mathcal{L}(B)$, motivate the following definition.

\begin{definition}[Monoidal Composable Constraints]
A monoidal composable constraint encoding for a monoidal category $\cat{C}$ is a homomorphism of monoidal bicategories \cite{Day1997MonoidalAlgebroids} \be \mathcal{L}: \mathbf{Con} \longrightarrow \mathcal{P}[\cat{C}]\ee from a locally posetal monoidal $2$-category $\mathbf{Con}$ to $\cp[\mathbf{C}]$, that is, a lax functor equipped with an oplax natural transformation, \be  \phi : \mathcal{L}(-) \otimes \mathcal{L}(-) \Rightarrow \mathcal{L}(- \otimes -) \,, \ee  a morphism $ \phi_o : I \rightarrow \mathcal{L}(I)$ and $2$-Morphisms: \be 
\begin{tikzcd}
(\mathcal{L}(A) \otimes \mathcal{L}(B) ) \otimes \mathcal{L}(C) \arrow[rr, "\alpha"] \arrow[d, "{\phi_{A,B} \otimes i}"'] &  & \mathcal{L}(A) \otimes ( \mathcal{L}(B)  \otimes \mathcal{L}(C)) \arrow[d, "{i \otimes \phi_{B,C}}"] \arrow[lldd, Rightarrow] \\
\mathcal{L}(A \otimes B) \otimes \mathcal{L}(C) \arrow[d, "{\phi_{A \otimes B,C}}"']                                      &  & \mathcal{L}(A) \otimes \mathcal{L}(B \otimes C) \arrow[d, "{\phi_{A,B \otimes C}}"]                                           \\
\mathcal{L}((A \otimes B) \otimes C) \arrow[rr, "\mathcal{L}(\alpha)"]                                                    &  & \mathcal{L}(A \otimes (B \otimes C))                                                                                         
\end{tikzcd}\ee 
\be \begin{tikzcd}
I \otimes \mathcal{L}(A) \arrow[rr, "\phi_o \otimes i"]                 & {} \arrow[d, Rightarrow] & \mathcal{L}(I) \otimes \mathcal{L}(A)  \arrow[d, "{\phi_{I,A}}"] &  & \mathcal{L}(A) \otimes I \arrow[rr, "i \otimes \phi_o"]           & {} \arrow[d, Rightarrow] & \mathcal{L}(A) \otimes \mathcal{L}(I)  \arrow[d, "{\phi_{A,I}}"] \\
\mathcal{L}(A) \arrow[u, "\lambda"] \arrow[rr, "\mathcal{L}(\lambda)"'] & {}                       & \mathcal{L}(I \otimes A)                                         &  & \mathcal{L}(A) \arrow[u, "\rho"] \arrow[rr, "\mathcal{L}(\rho)"'] & {}                       & \mathcal{L}(A \otimes I)                                        
\end{tikzcd} \,.\ee 
\end{definition}
The coherence conditions required for a homomorphism of monoidal bicategories are left out of the above definition since they are automatically satisfied by uniqueness of $2$-morphisms in $\mathcal{P}[\mathbf{C}]$. We refer to a monoidal composable constraint as \textit{strong} if the components of the oplax natural transformation $\phi_{A,B}:   \mathcal{L}(A) \otimes \mathcal{L}(B) \rightarrow \mathcal{L}(A \otimes B)$ and $\phi_o$ are isomorphisms. Taking the extreme case, a \textit{strict} monoidal composable constraint is then one for which $\phi_{A,B}$ and $\phi_o$ are identities. The strict case can be concisely rephrased by requiring that the oplax natural transformation $\phi$ be an \textit{icon} \cite{icons_lack}. The strict case is the most easily interpreted: it entails precisely the basic principle for monoidal compatibility:

\be \mathcal{L}(\tau) \otimes \mathcal{L}(\lambda) \subseteq \mathcal{L} (\tau \otimes \lambda) \,,\ee

expressed in functor box notation by

\be \tikzfigscale{1.0}{scheme_8} .\ee 

In the strict case, coherence isomorphisms indeed satisfy $\lambda \leq \mathcal{L}(\lambda)$, $\rho \leq \mathcal{L}(\rho)$, and $\alpha \leq \mathcal{L}(\alpha)$, as required. Graphically, one can see that strict monoidal composable constraints capture our basic principle by the following steps:

\be \tikzfigscale{1.0}{power_mon} .\ee

For every strong monoidal composable constraint, one can construct a monoidal category (and so a circuit theory) of morphisms in parallel with their constraints. 
\begin{theorem}[Monoidal Category of Constrained Morphisms]
Let $\mathcal{L}: \mathbf{Con} \longrightarrow \mathcal{P}[\cat{C}]$ be a strong monoidal composable constraint. The constrained category $\cat{C}_{\mathcal{L}}$ is a monoidal category.
\end{theorem}
\begin{proof}
A full proof is given in the appendix. In the strict case, it is easy to check that one can define a monoidal category by: \begin{itemize}
    \item $(\tau,f) \otimes (\tau',f') := (\tau \otimes \tau',f \otimes f')$;
    \item all coherence isomorphisms defined component wise, e.g. $\alpha_{\cat{C}_{\mathcal{L}}} := (\alpha_{\mathbf{Con}},\alpha_{\cat{C}})$.
\end{itemize} 

Crucially, the proposed coherence morphisms are well-defined since $\alpha_{\cat{C}} \in \{\alpha_{\cat{C}}\} = \alpha_{\mathcal{P}[\cat{C}]} \subseteq \mathcal{L}(\alpha_{\mathbf{Con}})$, and similarly for left and right unitors.
\end{proof}
Beyond monoidal structure, quantum-like theories in particular have a plethora of other important categorical structures: among other things, they are often symmetric (braided), dagger monoidal, or compact closed. For each such structure, the corresponding notion of constraint compatibility can be introduced, by requiring the following conditions.
\begin{itemize}
    \item Braided monoidalness intuitively requires $\beta_{\mathcal{P}[\cat{C}]} \subseteq \mathcal{L}(\beta_{\mathbf{C}})$, where $\beta_X$ is the braiding natural transformation for category $X$, formally given by \be \begin{tikzcd}
\mathcal{L}(A) \otimes \mathcal{L}(B) \arrow[d, "{\phi_{A,B}}"'] \arrow[rr, "{\beta_{\mathcal{L}(A),\mathcal{L}(B)}}"] & {} \arrow[d, Rightarrow] & \mathcal{L}(B) \otimes \mathcal{L}(A) \arrow[d, "{\phi_{B,A}}"] \\
\mathcal{L}(A \otimes B) \arrow[rr, "{\mathcal{L}(\beta_{A,B})}"']                                                     & {}                       & \mathcal{L}(B \otimes A)                                       
\end{tikzcd} \,;\ee 
    \item $\dagger$-monoidalness: $\mathcal{L} \circ \dagger_{\mathbf{Con}} = \dagger_{\mathcal{P}[\cat{C}]} \circ \mathcal{L}$ where the coherence isomorphisms of $\mathcal{L}$ are furthermore unitary;
    \item Compact closure: $\cup_{\mathcal{P}[\cat{C}]} \subseteq \mathcal{L}(\cup_{\mathbf{Con}})$, where $\cup_{X}^{AA*}$ is the cup on object $A$ in category $X$, or formally, \be \begin{tikzcd}
I \arrow[d, "\phi"'] \arrow[rr, "{\cup_{\mathcal{L}(A),\mathcal{L}(B)}}"] & {} \arrow[d, Rightarrow] & \mathcal{L}(A) \otimes \mathcal{L}(A^{*}) \arrow[d, "{\phi_{B,A}}"] \\
\mathcal{L}(I) \arrow[rr, "\mathcal{L}(\cup_{AA^{*}})"']                  & {}                       & \mathcal{L}(A \otimes A^{*})                                       
\end{tikzcd}\,.\ee 
\end{itemize}
\begin{theorem}[Inheritance of structure to constrained categories]
If a composable constraint $\mathcal{L}$ is:
\begin{itemize}
    \item a braided monoidal constraint, then $\mathbf{C}_{\mathcal{L}}$ is braided monoidal;
    \item a $\dagger$-constraint, then $\mathbf{C}_{\mathcal{L}}$ is a $\dagger$-category;
    \item a compact closed constraint, then $\mathbf{C}_{\mathcal{L}}$ is compact closed.
\end{itemize}
\end{theorem}
\begin{proof}
Given in the appendix.
\end{proof}
Whenever one has a theory of compact constraints over a compact category, one may as a result construct a new category, in which the constraints may be computed in parallel with the morphisms they constrain, \textit{and} which features a string diagram calculus with cups and caps.

\section{Examples} \label{sec:examples}
We now explore examples of composable constraints with increasing complexity. Two common and distinct classes are observed:
\begin{itemize}
    \item \textit{object-based type systems} which allow for using labelling of the objects that morphisms go between to specify particular subsets of the morphisms between those underlying objects;
    \item \textit{relational type systems} in which relations, as morphisms themselves, are used to specify the behaviour of morphisms.
\end{itemize}
The former notions can be understood as arising from constraints which are \textit{thin}, whereas relational constraints are formally more complex in that they require \textit{thick} constraint categories for their expression. In this sense, relational and in general composable constraints can be seen as a natural generalisation of sub-categories and type-systems; intuitively, rather than prescribing the \textit{preservation} of the structure of objects, relational constraints prescribe the \textit{distribution} of structure of objects.

\subsection{Thin constraints}
A thin category is a category in which there is at most one arrow between each pair of objects. Here, we will begin by considering constraints whose domains are thin categories:

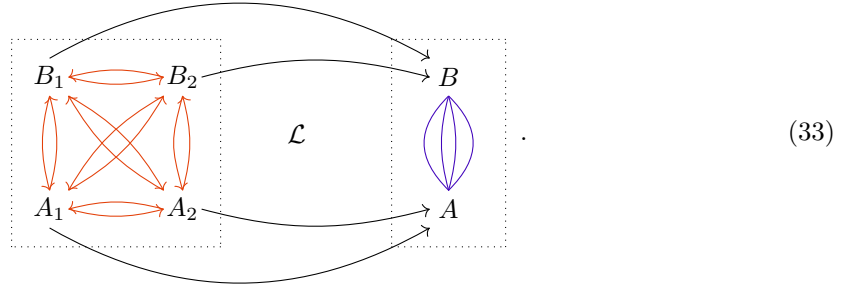
\be \tikzfigscale{1.0}{thin} \,.\ee

In general we will call constraints $\mathcal{L}$ with thin domains `thin constraints'; whilst being far simpler than the general behaviour we wish to model, thin constraints already contain many familiar examples. In each case, the resulting constrained category of a thin constraint has morphisms of the form $(\bullet,f)$; where $\bullet$ (representing the unique morphism between the relevant objects of the domain category) plays no role in the specification of the constraint beyond the specification of its input and output spaces. As a result, thin constraints can be viewed as the class of composable constraints which can easily be converted into \textit{structure} encoded within objects by moving to the constrained category.
The most familiar notion of a theory of composable constraints is that of a subcategory: each hom-set of $\mathbf{C}$ is restricted to a sub-hom-set, so that the hom-sets are still closed under composition of $\mathbf{C}$. Subcategories are examples of faithful functors and so of functors. We first show that each functor into $\mathbf{C}$ can be used to define a simple constrained category.

\begin{example}
Let $\mathcal{F}:\mathbf{C}' \rightarrow \mathbf{C}$ be a functor; we then define a locally discrete $2$-category $\mathbf{Sub}_{\mathcal{F}}$ with objects given by objects of $\mathbf{C}'$ and hom-sets given by singletons $\mathbf{Sub}_{\mathcal{F}}(A,B) := \{ \bullet \}$ with the obvious composition rule $\bullet \circ \bullet = \bullet$. There exists a composable constraint $\mathcal{L}:\mathbf{Sub}_{\mathcal{F}} \rightarrow \mathbf{C}$ given by $\mathcal{L}(A) = \mathcal{F}(A)$ and $\mathcal{L}(\bullet:A \rightarrow B) = \mathcal{F}(\mathbf{C}'(A,B))$. The resulting constrained category $\mathbf{C}_{\mathcal{L}}$ has objects given by pairs $(A,\mathcal{F}(A))$ and  hom-sets given by $\mathbf{C}_{\mathcal{L}}((A,\mathcal{F}(A)),(B,\mathcal{F}(B)) = \{(\bullet,f) | \exists g: A \rightarrow B, f = \mathcal{F}(g)\}$; it is as a result an intermediate category that can be used in the factorisation of a functor $\mathcal{F}:\mathbf{C}' \rightarrow \mathbf{C}$ into a full and bijective-on-objects functor $\mathcal{K}:\mathbf{C}' \rightarrow \mathbf{C}_{\mathcal{L}}$ and a faithul functor $M:\mathbf{C}_{\mathcal{L}} \rightarrow \mathbf{C}$.
\end{example}
A less trivial example is the idempotent completion of any category, which constructs from an arbitrary category a new type system in which objects are given by idempotents and morphisms are required to be invariant under the action of those idempotents.
\begin{example}
Define a category $\Pi$ with objects of the form $(A,\pi)$, with $A$ an object of $\mathbf{C}$ and $\pi:A \rightarrow A$ an idempotent morphism on $A$, and hom-sets again taken to be singletons $\Pi(A,B) := \{ \bullet \}$. A composable constraint is defined by $\mathcal{L}(A,\pi) := A$ and \be \mathcal{L}(\bullet:(A,\pi) \rightarrow (B,\pi')) := \{f|f=\pi' \circ f \circ \pi\} \,.\ee  The constrained category $\mathbf{C}_{\mathcal{L}}$ resulting from this thin constraint is isomorphic to the Karoubi envelope (idempotent completion) of $\mathbf{C}$.
\end{example}

Finally, we consider type systems constructed from specifying subsets of state spaces and their preservation: we give an example  of a constraint whose equivalent constrained category is isomorphic to the $\mathbf{Caus}[\mathbf{C}]$ construction \cite{Kissinger2019AStructure}, used to construct higher order processes from raw material pre-causal categories. For any set $S$, we define the category $S_{\circlearrowleft} $ to be the category with objects given by members of $S$ and homsets given by singletons $S_{\circlearrowleft}  (A,B) := \{ \bullet \}$. For any subset $a \subseteq \mathbf{C}(I,A)$ and $b \subseteq \mathbf{C}(I,B)$ and morphism $f:A \rightarrow B$, we write $f \circ a \subseteq b$ to express the statement $\forall \rho \in a, \quad f \circ \rho \in b $.

\begin{example}[State preservation]
Let $\mathbf{C}$ be a category. For every $S \subseteq \sqcup_{A \in o(\mathbf{C})} P[ \mathbf{C}(I,A)]$, there exists a compositional constraint with domain given by $S_{ \circlearrowleft} $ and lax functor given by $\mathcal{L}_{S}(\bullet:c \rightarrow d) := \{f \textrm{ s.t }f \circ c \subseteq d \} $. It is clear to see that $i \in \mathcal{L}_{S}(i:c \rightarrow c)$ and that \be  f \circ c \subseteq d  \wedge g \circ d \subseteq e \implies (g \circ f) \circ c \subseteq e \,.\ee 
\end{example}
An example of such a construction is the construction of higher order causal categories (HOCCs) \cite{Kissinger2019AStructure}. The condition for a state set to be a member of $S$ is that it ought to be closed and flat; remarkably, the conditions of closure and flatness are enough to ensure a variety of structures on the induced constrained category, all the way to $*$-autonomous structure. In more general settings, one can produce a monoidal constraint by requiring `complete' state preservation for morphisms.
\begin{example}[Complete state preservation]
Let $S \subseteq \sqcup_{A \in o(\mathbf{C})} P[ \mathbf{C}(I,A)]$ be equipped with a function $\otimes : S \times S \rightarrow S$ such that 
\begin{itemize}
    \item $o(c \otimes d) = c \otimes d$
    \item $\rho \circ (c \otimes c_I) = c$
    \item $\lambda \circ (c_I \otimes c) = c$
    \item $\alpha \circ (c \otimes d) \otimes e = c \otimes (d \otimes e)$
\end{itemize}
Then there exists a monoidal compositional constraint given by 
\be  \mathcal{L}_{S}^{\otimes}(\bullet : c \rightarrow d) = \{f \textrm{ s.t } \forall e : \textrm{ } (f \otimes i_{e}) \circ (c \otimes e) = (d \otimes e) \} \,.  \ee  This notion of preservation will be referred to as complete preservation, in analogy with complete positivity.
\end{example}
Having shown that basic notions of composable constraints fit within the framework, we now explore more elaborate notions of composable constraint by looking at examples with richer domains. The new relational types for morphisms express how morphisms must distribute structure, as opposed to expressing that they must preserve structure.

\subsection{Constraints on monoids by relations}

We now observe a first basic example of constraining one-object categories with endorelations. Far from being a toy example, this construction captures compositionality of causality constraints in dynamical quantum networks \cite{Arrighi2021QuantumTheory, Arrighi2017QuantumDynamics}. We will denote the full subcategory of a category $\mathbf{C}$ with single object $S$ by $\mathbf{C}_{S}$. 
\begin{example}
For every function $f:S \rightarrow M$ from a set $S$ into a monoid $M$, there exists a compositional constraint $\mathcal{L}_{f}:  \mathbf{REL}_{S} \rightarrow \mathcal{P}[\mathbf{M}] $, given by $ \mathcal{L}_{f}(S) = M $ and by
\be \mathcal{L}_{f}(\tau:S \rightarrow S) = \{m \in M \textrm{ s.t }  \forall x \overset{\tau}{\sim} y, \textrm{ } \exists m' \textrm{ s.t } f(y)m = m'f(x)    \} \,. \ee 

Note that $\mathcal{L}_{f}(i) = \{m \textrm{ s.t } \forall x,  \exists m' \textrm{ s.t } f(x)m = f(x)m'   \} \supseteq \{i \} $ and that

\be f,g \in \mathcal{L}(\tau) \times \mathcal{L}(\lambda) \implies \forall x \overset{\tau \circ \lambda}{\sim} z : f(z) m  n = m' f(y) n = m'n'f(x) \,, \ee 

so the constraint is indeed compositional.
\end{example}
A concrete example of a relational constraint over a monoid is the description of causality constraints on theories of dynamical quantum networks \cite{Arrighi2021QuantumTheory, Arrighi2017QuantumDynamics}. For instance, signalling constraints between subsystems $\chi$ and $\zeta$ in quantum networks theory \cite{Arrighi2021QuantumTheory} are defined over a sub-monoid of all linear maps  generated by unitary actions $U(-)U^{\dagger}$, and restriction-parametrised partial traces $(-)_{|\chi}$ by the condition $(U(-)U^{\dagger})_{|\zeta} = (U(-)_{|\chi}U^{\dagger})_{|\zeta}$. The relevant function $f$ in this context is given by $f(\chi) = (-)_{|\chi}$, assigning each restriction to its corresponding partial trace. The equivalence between causality constraints in quantum networks and the above construction for general monoids is given by noting that
\be  f(\zeta) \circ \hat{U}  = f(\zeta) \circ \hat{U} \circ f(\chi) \iff \exists m'\textrm { s.t }  f(\zeta) \circ \hat{U}  =  m' \circ f(\chi) \,, \ee

with $\hat{U}(-):= U(-)U^{\dagger}$. Signalling constraints in more familiar circuit based models will be constructed in section \ref{sec:Relational Constraints}. The relevance of introducing composable constraints to monoids to physics will in more general terms likely come from applying them to theories which take the view of the process-theoretic church of the larger Hilbert space \cite{Gogioso2019ASpace}, in which the universe is modelled as a single object from which subsystems are then carved out.

\subsection{Sectorial constraints in categories enriched over commutative monoids}\label{sec:Sectorial Constraints}

A less trivial, and quite general, example of composable constraints is that of \textit{sectorial constraints}. Conceptually, sectorial constraints correspond to the idea of stating that some possible alternatives for an input state cannot lead to some alternatives for the output state. The seminal example is that of sectors (i.e. orthogonal subspaces) of Hilbert spaces: a natural interpretation of a partition into sectors in quantum theory is that each sector corresponds to one of the possible outcomes of a measurement.

Mathematically, the idea is to look at  finite biproducts (when these exist), and to restrict the morphisms between them to feature no connection between some of the underlying objects of the biproduct (which we call sectors):

\begin{equation}\label{eq:sector1}
\tikzfigscale{1}{scheme_2} \, .
\end{equation}.

This can be defined in any theory enriched over commutative monoids, which includes semiadditive categories and preadditive categories, such as those of modules over rings. The category in which sectorial constraints live is then isomorphic a category of finite relations.

\begin{example} \label{example:Sectorial Constraints}
Let \cat{C} be a category enriched over commutative monoids. We define the category $\cat{Sec}_{\cat{C}}$ whose objects are finite lists $(A^k)_{k \in K}$ of objects of \cat{C} for which a biproduct exists\footnote{Note that it is sufficient that a (co)product of these objects exists, as enrichment over commutative monoids implies that it will actually be a biproduct.}, together with a specific choice $A$ of such a biproduct; and whose morphisms are (finite) relations between the sets of indices of the lists. There exists a composable constraint $\cl_{\cat{Sec}_{\cat{C}}} : \cat{Sec}_{\cat{C}} \to \cat{C}$ defined by
\be \cl(\lambda) = \{ f: A \to B \, | \, \forall k,l \textrm{ not related by } \lambda, p^l_B \circ f \circ i^k_A = 0 \} \, , \ee
where $i_A^k$ is the injection $A^k \to A$ and $p_B^l$ is the projection $B^l \to B$.
\end{example}

\begin{proof}
Let us prove the inclusion $\cl_{\cat{Sec}_{\cat{C}}}(\sigma) \circ \cl_{\cat{Sec}_{\cat{C}}}(\lambda) \subseteq \cl_{\cat{Sec}_{\cat{C}}}(\sigma \circ \lambda)$. Suppose $f: A \to B$ and $g: B \to C$ satisfy the respective constraints. Then it is easy to prove that $\textrm{id}_{B} = \sum_l p_B^l \circ i_B^l$, and therefore,
$$\forall k,m, p^m_C \circ g \circ f \circ i^k_A = \sum_l (p^m_C \circ g \circ i_B^l) \circ (p_B^l \circ f \circ i^k_A) \, .$$

In addition, taking $k$ and $m$ not related by $\sigma \circ \lambda$, one has that for any $l$, it is true that $k$ is not related to $l$ by $\lambda$ or that $l$ is not related to $m$ by $\sigma$ (or both); from there one can deduce using the previous equation that $p^m_C \circ g \circ f \circ i^k_A = 0$.
\end{proof}

If \cat{C} is semiadditive (i.e. if it has all finite biproducts), then it is easy to see that $\cat{Sec}[\cat{C}]$, once equipped with the disjoint-union monoidal structure of finite relations, forms a monoidal theory of constraints with respect to the biproduct monoidal structure of  \cat{C}. Moreover, if \cat{C} is bimonoidal \cite{johnson2021bimonoidal}, then one can see that sectorial constraints also form monoidal composable constraints with respect to the cartesian product of relations; the same goes for dagger structure and compact closure.

The use of finite relations to encode sectorial constraints was originally developed for the case $\cat{C} = \FHilb$ in \cite{vanrietvelde2020routed}, in order to model more accurately some quantum-theoretical scenarios, giving rise to so-called \textit{routed maps}; this in turn was the main motivation for the generalisation to arbitrary composable constraints presented here.

\subsection{Relational constraints in semicartesian categories} \label{sec:Relational Constraints}

We now consider relational constraints for processes in semicartesian categories, that is, monoidal categories for which the unit object is terminal (in simple terms, this means that there is only one effect on each object). Typical examples are the theory of quantum channels (i.e. deterministic quantum processes) and that of classical stochastic maps; for instance, in  the former, the only deterministic quantum channel into a $1$-dimensional system is the trace operation. Another example is that of categories with all finite coproducts. In the first case, the unique effect on a given object can be interpreted as the discarding process, through which one forgets about a system; in the second one, it can be interpreted as a negation, meaning that one of several possible alternatives is false. The first of these two interpretations leads to the introduction of \textit{signalling conditions} on multiparty processes. For example, in the following diagram we specify a constraint on which sub-systems of the input system can influence subsystems of the output system:

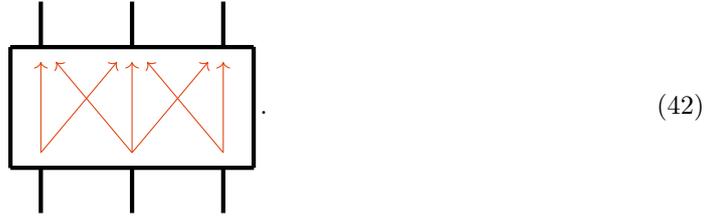
\be \begin{tikzpicture}[scale=0.8, {every node/.style}={scale=0.8}]
	\begin{pgfonlayer}{nodelayer}
		\node [style=none] (0) at (-2, 1) {};
		\node [style=none] (1) at (2, 1) {};
		\node [style=none] (2) at (2, -1) {};
		\node [style=none] (3) at (-2, -1) {};
		\node [style=none] (4) at (-1.5, 1.75) {};
		\node [style=none] (5) at (0, 1.75) {};
		\node [style=none] (6) at (1.5, 1.75) {};
		\node [style=none] (8) at (-1.5, 1) {};
		\node [style=none] (9) at (0, 1) {};
		\node [style=none] (10) at (1.5, 1) {};
		\node [style=none] (12) at (-1.5, -1) {};
		\node [style=none] (13) at (0, -1) {};
		\node [style=none] (14) at (1.5, -1) {};
		\node [style=none] (16) at (-1.5, -1.75) {};
		\node [style=none] (17) at (0, -1.75) {};
		\node [style=none] (18) at (1.5, -1.75) {};
		\node [style=none] (19) at (-1.5, -0.75) {};
		\node [style=none] (20) at (-1.5, 0.75) {};
		\node [style=none] (21) at (-0.25, 0.75) {};
		\node [style=none] (22) at (1.5, 0.75) {};
		\node [style=none] (23) at (1.5, -0.75) {};
		\node [style=none] (24) at (0, -0.75) {};
		\node [style=none] (25) at (-1.25, 0.75) {};
		\node [style=none] (26) at (0.25, 0.75) {};
		\node [style=none] (27) at (1.25, 0.75) {};
		\node [style=none] (28) at (0, 0.75) {};
	\end{pgfonlayer}
	\begin{pgfonlayer}{edgelayer}
		\draw [style=double edge] (0.center) to (1.center);
		\draw [style=double edge] (1.center) to (2.center);
		\draw [style=double edge] (2.center) to (3.center);
		\draw [style=double edge] (3.center) to (0.center);
		\draw [style=double edge] (4.center) to (8.center);
		\draw [style=double edge] (5.center) to (9.center);
		\draw [style=double edge] (6.center) to (10.center);
		\draw [style=double edge] (12.center) to (16.center);
		\draw [style=double edge] (13.center) to (17.center);
		\draw [style=double edge] (14.center) to (18.center);
		\draw [style=new edge style 8] (22.center) to (23.center);
		\draw [style=new edge style 8] (26.center) to (23.center);
		\draw [style=new edge style 8] (27.center) to (24.center);
		\draw [style=new edge style 8] (28.center) to (24.center);
		\draw [style=new edge style 8] (25.center) to (24.center);
		\draw [style=new edge style 8] (20.center) to (19.center);
		\draw [style=new edge style 8] (21.center) to (19.center);
	\end{pgfonlayer}
\end{tikzpicture}\,.\ee 

The same could be said about the second interpretation -- in fact, as we will see, this will then yield a generalisation of sectorial constraints. In general, we will call such constraints \textit{relational constraints}, as they naturally have the form of (finite) relations. Note that the actual constraints here come from the \textit{absences} of arrows.

This condition on the possibility of influence between input and output systems can be formalised in any semicartesian category in terms of discarding processes, by requiring that whenever all arrows emerging from a particular input are discarded, then the input system itself can be discarded. For example, there are two non-trivial constraints encoded by the above diagram. The first is a specification of the result of discarding all arrows emerging from the left-hand input: \be 
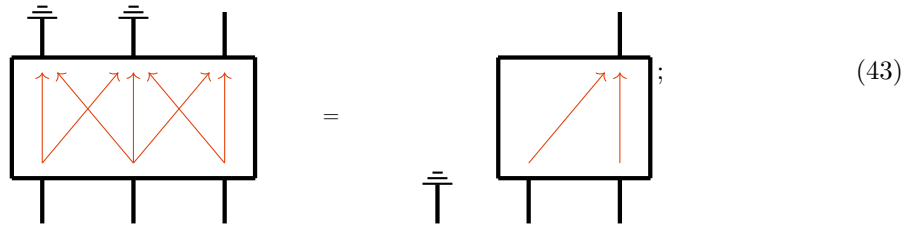
\begin{tikzpicture}[scale=0.8, {every node/.style}={scale=0.8}]
	\begin{pgfonlayer}{nodelayer}
		\node [style=none] (55) at (-0.25, -0.75) {$=$};
		\node [style=none] (56) at (-5.5, 0.25) {};
		\node [style=none] (57) at (-1.5, 0.25) {};
		\node [style=none] (58) at (-1.5, -1.75) {};
		\node [style=none] (59) at (-5.5, -1.75) {};
		\node [style=smalldiscarding] (60) at (-5, 1) {};
		\node [style=smalldiscarding] (61) at (-3.5, 1) {};
		\node [style=none] (62) at (-2, 1) {};
		\node [style=none] (63) at (-5, 0.25) {};
		\node [style=none] (64) at (-3.5, 0.25) {};
		\node [style=none] (65) at (-2, 0.25) {};
		\node [style=none] (66) at (-5, -1.75) {};
		\node [style=none] (67) at (-3.5, -1.75) {};
		\node [style=none] (68) at (-2, -1.75) {};
		\node [style=none] (69) at (-5, -2.5) {};
		\node [style=none] (70) at (-3.5, -2.5) {};
		\node [style=none] (71) at (-2, -2.5) {};
		\node [style=none] (72) at (-5, -1.5) {};
		\node [style=none] (73) at (-5, 0) {};
		\node [style=none] (74) at (-3.75, 0) {};
		\node [style=none] (75) at (-2, 0) {};
		\node [style=none] (76) at (-2, -1.5) {};
		\node [style=none] (77) at (-3.5, -1.5) {};
		\node [style=none] (78) at (-4.75, 0) {};
		\node [style=none] (79) at (-3.25, 0) {};
		\node [style=none] (80) at (-2.25, 0) {};
		\node [style=none] (81) at (-3.5, 0) {};
		\node [style=none] (108) at (2.5, 0.25) {};
		\node [style=none] (109) at (5, 0.25) {};
		\node [style=none] (110) at (5, -1.75) {};
		\node [style=none] (111) at (2.5, -1.75) {};
		\node [style=smalldiscarding] (112) at (1.5, -1.75) {};
		\node [style=none] (114) at (4.5, 1) {};
		\node [style=none] (115) at (1.5, -2.5) {};
		\node [style=none] (117) at (4.5, 0.25) {};
		\node [style=none] (119) at (3, -1.75) {};
		\node [style=none] (120) at (4.5, -1.75) {};
		\node [style=none] (122) at (3, -2.5) {};
		\node [style=none] (123) at (4.5, -2.5) {};
		\node [style=none] (127) at (4.5, 0) {};
		\node [style=none] (128) at (4.5, -1.5) {};
		\node [style=none] (129) at (3, -1.5) {};
		\node [style=none] (132) at (4.25, 0) {};
	\end{pgfonlayer}
	\begin{pgfonlayer}{edgelayer}
		\draw [style=double edge] (56.center) to (57.center);
		\draw [style=double edge] (57.center) to (58.center);
		\draw [style=double edge] (58.center) to (59.center);
		\draw [style=double edge] (59.center) to (56.center);
		\draw [style=double edge] (60) to (63.center);
		\draw [style=double edge] (61) to (64.center);
		\draw [style=double edge] (62.center) to (65.center);
		\draw [style=double edge] (66.center) to (69.center);
		\draw [style=double edge] (67.center) to (70.center);
		\draw [style=double edge] (68.center) to (71.center);
		\draw [style=new edge style 8] (75.center) to (76.center);
		\draw [style=new edge style 8] (79.center) to (76.center);
		\draw [style=new edge style 8] (80.center) to (77.center);
		\draw [style=new edge style 8] (81.center) to (77.center);
		\draw [style=new edge style 8] (78.center) to (77.center);
		\draw [style=new edge style 8] (73.center) to (72.center);
		\draw [style=new edge style 8] (74.center) to (72.center);
		\draw [style=double edge] (108.center) to (109.center);
		\draw [style=double edge] (109.center) to (110.center);
		\draw [style=double edge] (110.center) to (111.center);
		\draw [style=double edge] (111.center) to (108.center);
		\draw [style=double edge] (112) to (115.center);
		\draw [style=double edge] (114.center) to (117.center);
		\draw [style=double edge] (119.center) to (122.center);
		\draw [style=double edge] (120.center) to (123.center);
		\draw [style=new edge style 8] (127.center) to (128.center);
		\draw [style=new edge style 8] (132.center) to (129.center);
	\end{pgfonlayer}
\end{tikzpicture}\,;\ee 
and the second case is the specification of the result of discarding all arrows emerging from the right-hand input:
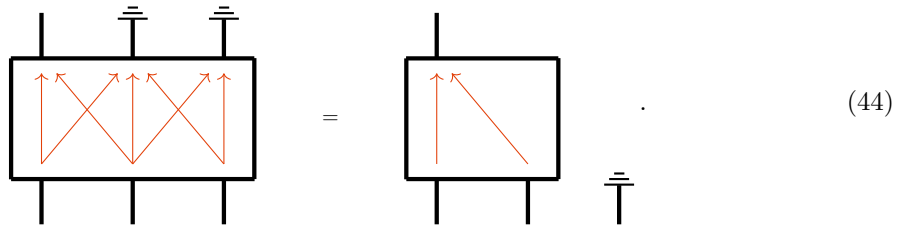
\be \begin{tikzpicture}[scale=0.8, {every node/.style}={scale=0.8}]
	\begin{pgfonlayer}{nodelayer}
		\node [style=none] (29) at (-5.75, 0.75) {};
		\node [style=none] (30) at (-1.75, 0.75) {};
		\node [style=none] (31) at (-1.75, -1.25) {};
		\node [style=none] (32) at (-5.75, -1.25) {};
		\node [style=none] (33) at (-5.25, 1.5) {};
		\node [style=smalldiscarding] (34) at (-3.75, 1.5) {};
		\node [style=smalldiscarding] (35) at (-2.25, 1.5) {};
		\node [style=none] (36) at (-5.25, 0.75) {};
		\node [style=none] (37) at (-3.75, 0.75) {};
		\node [style=none] (38) at (-2.25, 0.75) {};
		\node [style=none] (39) at (-5.25, -1.25) {};
		\node [style=none] (40) at (-3.75, -1.25) {};
		\node [style=none] (41) at (-2.25, -1.25) {};
		\node [style=none] (42) at (-5.25, -2) {};
		\node [style=none] (43) at (-3.75, -2) {};
		\node [style=none] (44) at (-2.25, -2) {};
		\node [style=none] (133) at (-0.5, -0.25) {$=$};
		\node [style=none] (134) at (0.75, 0.75) {};
		\node [style=none] (135) at (3.25, 0.75) {};
		\node [style=none] (136) at (3.25, -1.25) {};
		\node [style=none] (137) at (0.75, -1.25) {};
		\node [style=none] (138) at (1.25, 1.5) {};
		\node [style=smalldiscarding] (140) at (4.25, -1.25) {};
		\node [style=none] (141) at (1.25, 0.75) {};
		\node [style=none] (143) at (4.25, -2) {};
		\node [style=none] (144) at (1.25, -1.25) {};
		\node [style=none] (145) at (2.75, -1.25) {};
		\node [style=none] (147) at (1.25, -2) {};
		\node [style=none] (148) at (2.75, -2) {};
		\node [style=none] (150) at (1.25, -1) {};
		\node [style=none] (151) at (1.25, 0.5) {};
		\node [style=none] (155) at (2.75, -1) {};
		\node [style=none] (156) at (1.5, 0.5) {};
		\node [style=none] (157) at (-5.25, -1) {};
		\node [style=none] (158) at (-5.25, 0.5) {};
		\node [style=none] (159) at (-4, 0.5) {};
		\node [style=none] (160) at (-2.25, 0.5) {};
		\node [style=none] (161) at (-2.25, -1) {};
		\node [style=none] (162) at (-3.75, -1) {};
		\node [style=none] (163) at (-5, 0.5) {};
		\node [style=none] (164) at (-3.5, 0.5) {};
		\node [style=none] (165) at (-2.5, 0.5) {};
		\node [style=none] (166) at (-3.75, 0.5) {};
	\end{pgfonlayer}
	\begin{pgfonlayer}{edgelayer}
		\draw [style=double edge] (29.center) to (30.center);
		\draw [style=double edge] (30.center) to (31.center);
		\draw [style=double edge] (31.center) to (32.center);
		\draw [style=double edge] (32.center) to (29.center);
		\draw [style=double edge] (33.center) to (36.center);
		\draw [style=double edge] (34) to (37.center);
		\draw [style=double edge] (35) to (38.center);
		\draw [style=double edge] (39.center) to (42.center);
		\draw [style=double edge] (40.center) to (43.center);
		\draw [style=double edge] (41.center) to (44.center);
		\draw [style=double edge] (134.center) to (135.center);
		\draw [style=double edge] (135.center) to (136.center);
		\draw [style=double edge] (136.center) to (137.center);
		\draw [style=double edge] (137.center) to (134.center);
		\draw [style=double edge] (138.center) to (141.center);
		\draw [style=double edge] (140) to (143.center);
		\draw [style=double edge] (144.center) to (147.center);
		\draw [style=double edge] (145.center) to (148.center);
		\draw [style=new edge style 8] (156.center) to (155.center);
		\draw [style=new edge style 8] (151.center) to (150.center);
		\draw [style=new edge style 8] (160.center) to (161.center);
		\draw [style=new edge style 8] (164.center) to (161.center);
		\draw [style=new edge style 8] (165.center) to (162.center);
		\draw [style=new edge style 8] (166.center) to (162.center);
		\draw [style=new edge style 8] (163.center) to (162.center);
		\draw [style=new edge style 8] (158.center) to (157.center);
		\draw [style=new edge style 8] (159.center) to (157.center);
	\end{pgfonlayer}
\end{tikzpicture} \,.\ee 
The above constraints, when interpreted as signalling constraints for quantum processes, are of broad interest in quantum information and quantum foundations \cite{Allen2017, barrett2019, lorenz2020, barrett2020cyclic, Coecke2013CausalProcesses, Kissinger2019AStructure,  wilsoncausal2021, semi_causal, Beckman_causal_local, perinotti_causal_influence}. We will now show that such constraints fit within the framework of composable constraints. The result of the constrained category construction here will be a formal circuit theory in which signalling constraints are expressed within circuit diagrams, essentially giving a formalisation of the above schematic diagrams.

\begin{example} \label{example:SemicartesianRelational}
For any symmetric semicartesian category $(\mathbf{C},\otimes,I)$ there exists a symmetric monoidal composable constraint $\mathcal{L}_{Rel}: \mathbf{Rel}_{\mathbf{C}} \rightarrow \mathbf{C}$, where the symmetric monoidal category $(\mathbf{Rel}_{\mathbf{C}},\sqcup,[\bullet])$ of relational constraints for $\mathbf{C}$ is defined as the category whose objects are finite lists $H := [H_1 , \dots , H_n]$ of objects of $\mathbf{C}$; and whose morphisms are relations between the sets of labels of the lists (we refer to these sets as $\texttt{label}(H)$). Graphically, one can picture these relations as bi-graphs on objects-labelled nodes:

\be \begin{tikzpicture}[scale=0.8, {every node/.style}={scale=0.8}]
	\begin{pgfonlayer}{nodelayer}
		\node [style=none] (36) at (-2, 1.25) {};
		\node [style=none] (37) at (-0.5, 1.25) {};
		\node [style=none] (38) at (1, 1.25) {};
		\node [style=none] (39) at (-2, -1.25) {};
		\node [style=none] (40) at (-0.5, -1.25) {};
		\node [style=none] (41) at (1, -1.25) {};
		\node [style=none] (55) at (2.25, 1.25) {};
		\node [style=none] (56) at (-2, 1.5) {$K_1$};
		\node [style=none] (57) at (-0.5, 1.5) {$K_2$};
		\node [style=none] (58) at (1, 1.5) {$K_3$};
		\node [style=none] (59) at (2.25, 1.5) {$K_4$};
		\node [style=none] (60) at (-2, -1.5) {$H_1$};
		\node [style=none] (61) at (-0.5, -1.5) {$H_2$};
		\node [style=none] (62) at (1, -1.5) {$H_3$};
	\end{pgfonlayer}
	\begin{pgfonlayer}{edgelayer}
		\draw [style=new edge style 5] (36.center) to (40.center);
		\draw [style=new edge style 5] (39.center) to (36.center);
		\draw [style=new edge style 5] (37.center) to (41.center);
		\draw [style=new edge style 5, in=90, out=-90] (38.center) to (41.center);
		\draw [style=new edge style 5] (40.center) to (38.center);
		\draw [style=new edge style 5] (40.center) to (55.center);
	\end{pgfonlayer}
\end{tikzpicture} \,.\ee 

The monoidal product $\sqcup$ in $\mathbf{Rel}_{\mathbf{C}}$ is taken to be a coproduct structure, given on objects by list concatenation: $[H_1 \dots H_n] \sqcup [G_1 \dots G_m] := [H_1 , \dots , H_n , G_1 , \dots , G_m]$, with a strict unit object given by the empty list $[\bullet]$; and on morphisms by the coproduct of $\mathbf{Rel}$. Graphically, this monoidal product admits a very simple graphical representation:

\be \begin{tikzpicture}[scale=0.8, {every node/.style}={scale=0.8}]
	\begin{pgfonlayer}{nodelayer}
		\node [style=none] (57) at (-1.75, 1.25) {};
		\node [style=none] (58) at (-0.25, 1.25) {};
		\node [style=none] (61) at (-1.75, -1.25) {};
		\node [style=none] (62) at (-0.25, -1.25) {};
		\node [style=none] (71) at (-3, 0) {$\sqcup$};
		\node [style=none] (72) at (2, 0) {$:=$};
		\node [style=none] (97) at (-7.25, 1.25) {};
		\node [style=none] (98) at (-5.75, 1.25) {};
		\node [style=none] (99) at (-4.25, 1.25) {};
		\node [style=none] (100) at (-7.25, -1.25) {};
		\node [style=none] (101) at (-5.75, -1.25) {};
		\node [style=none] (104) at (1.25, -1.25) {};
		\node [style=none] (118) at (-1.75, 1.5) {$M_1$};
		\node [style=none] (119) at (-0.25, 1.5) {$M_2$};
		\node [style=none] (120) at (-7.25, 1.5) {$K_1$};
		\node [style=none] (121) at (-5.75, 1.5) {$K_2$};
		\node [style=none] (122) at (-4.25, 1.5) {$K_3$};
		\node [style=none] (123) at (7.25, 1.5) {$M_1$};
		\node [style=none] (124) at (8.75, 1.5) {$M_2$};
		\node [style=none] (125) at (3, 1.5) {$K_1$};
		\node [style=none] (126) at (4.5, 1.5) {$K_2$};
		\node [style=none] (127) at (6, 1.5) {$K_3$};
		\node [style=none] (128) at (-1.75, -1.5) {$G_1$};
		\node [style=none] (129) at (-0.25, -1.5) {$G_2$};
		\node [style=none] (130) at (-7.25, -1.5) {$H_1$};
		\node [style=none] (131) at (-5.75, -1.5) {$H_2$};
		\node [style=none] (132) at (1.25, -1.5) {$G_3$};
		\node [style=none] (133) at (6, -1.5) {$G_1$};
		\node [style=none] (134) at (7.5, -1.5) {$G_2$};
		\node [style=none] (135) at (3, -1.5) {$H_1$};
		\node [style=none] (136) at (4.5, -1.5) {$H_2$};
		\node [style=none] (137) at (9, -1.5) {$G_3$};
		\node [style=none] (138) at (3, 1.25) {};
		\node [style=none] (139) at (4.5, 1.25) {};
		\node [style=none] (140) at (6, 1.25) {};
		\node [style=none] (141) at (3, -1.25) {};
		\node [style=none] (142) at (4.5, -1.25) {};
		\node [style=none] (143) at (7.25, 1.25) {};
		\node [style=none] (144) at (8.75, 1.25) {};
		\node [style=none] (145) at (6, -1.25) {};
		\node [style=none] (146) at (7.5, -1.25) {};
		\node [style=none] (147) at (9, -1.25) {};
	\end{pgfonlayer}
	\begin{pgfonlayer}{edgelayer}
		\draw [style=new edge style 5] (97.center) to (101.center);
		\draw [style=new edge style 5] (100.center) to (97.center);
		\draw [style=new edge style 5] (101.center) to (99.center);
		\draw [style=new edge style 5] (57.center) to (62.center);
		\draw [style=new edge style 5] (58.center) to (62.center);
		\draw [style=new edge style 5] (61.center) to (57.center);
		\draw [style=new edge style 5] (58.center) to (104.center);
		\draw [style=new edge style 5] (98.center) to (101.center);
		\draw [style=new edge style 5] (138.center) to (142.center);
		\draw [style=new edge style 5] (141.center) to (138.center);
		\draw [style=new edge style 5] (142.center) to (140.center);
		\draw [style=new edge style 5] (139.center) to (142.center);
		\draw [style=new edge style 5] (143.center) to (146.center);
		\draw [style=new edge style 5] (144.center) to (146.center);
		\draw [style=new edge style 5] (145.center) to (143.center);
		\draw [style=new edge style 5] (144.center) to (147.center);
	\end{pgfonlayer}
\end{tikzpicture} \,.\ee 

The composable constraint $\mathcal{L}_{Rel}$ is given on objects by $\mathcal{L}_{Sig}([H_1 \dots H_n]) = \bigotimes_{i} H_{i}$. To define its action on morphisms, we introduce some preliminary notation. For a relation $\tau: K \to L$ and an input element $i \in \texttt{label}(K)$, we define the set $\tau_i := \{j|i \overset{\tau}{\sim} j\} \subseteq \texttt{label}(L)$ of the elements that $i$ is related to. We also define, for every object $K$ of $\mathbf{Rel}_{\mathbf{C}}$ and every subset $S \subseteq \texttt{label}(K)$, a function $S(-):\texttt{label}(K) \rightarrow \texttt{morph}(\mathbf{C})$ into the morphisms of $\mathbf{C}$ defined by $S(i) := \discard$ if $i \in S$ and $S(i) = \identity$ otherwise. Finally, for such an $S$, we define the morphism $d_{S}: = \otimes_{i}S(i)$.

Put simply, $S$ is used to decide which objects should be discarded; as an example, consider the following:
\be \begin{tikzpicture}[scale=0.8, {every node/.style}={scale=0.8}]
	\begin{pgfonlayer}{nodelayer}
		\node [style=none] (65) at (-3, 1) {$K:= [K_1, K_2,K_3,K_4,K_5]$};
		\node [style=none] (67) at (-3, 0) {$S:= \{2,3,5\}$};
		\node [style=none] (68) at (0.75, 1.5) {$S(K_1) = \identity$};
		\node [style=none] (69) at (0.75, 1) {$S(K_2) = \discard$};
		\node [style=none] (70) at (0.75, 0.5) {$S(K_3) = \discard$};
		\node [style=none] (71) at (0.75, 0) {$S(K_4) = \identity$};
		\node [style=none] (72) at (0.75, -0.5) {$S(K_5) = \discard$};
		\node [style=none] (73) at (3.5, 0.5) {$d_{S}:=$};
		\node [style=none] (74) at (4.5, 1.25) {};
		\node [style=none] (75) at (4.5, 0) {};
		\node [style=smalldiscarding] (76) at (5.25, 0.75) {};
		\node [style=none] (77) at (5.25, 0) {};
		\node [style=smalldiscarding] (78) at (6, 0.75) {};
		\node [style=none] (79) at (6, 0) {};
		\node [style=none] (80) at (6.75, 1.25) {};
		\node [style=none] (81) at (6.75, 0) {};
		\node [style=smalldiscarding] (82) at (7.5, 0.75) {};
		\node [style=none] (83) at (7.5, 0) {};
	\end{pgfonlayer}
	\begin{pgfonlayer}{edgelayer}
		\draw [style=double edge] (74.center) to (75.center);
		\draw [style=double edge] (76) to (77.center);
		\draw [style=double edge] (78) to (79.center);
		\draw [style=double edge] (80.center) to (81.center);
		\draw [style=double edge] (82) to (83.center);
	\end{pgfonlayer}
\end{tikzpicture}\,.\ee  

 The action of $\mathcal{L}_{Rel}$ on morphisms is then given by
 \be \label{eq:SatisfactionRelConstraints} \mathcal{L}_{\mathbf{Rel}_{\mathbf{C}}}(\tau: K \rightarrow L) = \{f \, | \, \forall i \in \texttt{label}(K), \, \exists f'_i \textrm{ s.t } :d_{\tau_i} \circ f = f_i' \circ d_{\{i\}}\} \, . \ee
 
 \end{example}
 
This is a formal way of defining the signalling constraints as described in the informal introduction. Before proving that this defines a composable constraint encoding, we display a useful lemma, which shows a property we call \textit{domain-atomicity}: if several input elements are not connected to some subset of outputs, then one can say that the monoidal product of these input elements isn't connected to this subset of outputs either.

\begin{lemma} \label{lemma:Domain Atomicity}
Consider $f \in \mathcal{L}_{\mathbf{Rel}_{\mathbf{C}}}(\tau: K \rightarrow L)$, and $T \subseteq \texttt{label}(L)$. Then there exists $f_T'$ such that
\be 
d_T \circ f = f_T' \circ d_{\{i | \tau_i \subseteq T\}} \, .
\ee
\end{lemma}

\begin{proof}
Let us introduce some notation. For every object $M$ of $\mathbf{Rel}_{\mathbf{C}}$ and every subset $S \subseteq \texttt{label}(M)$, we define the function $\tilde{S}:\texttt{label}(K) \rightarrow \texttt{morph}(\mathbf{C})$ by $\tilde{S}(i) := \mu_i $, where $\mu_i$ is an arbitrary state on $M_i$, if $i \in S$; and $S(i) = \identity$ otherwise. For such an $S$, we then define the morphism $e_{S}: = \otimes_{i} \tilde{S}(i)$. Basically, $e_S$ is a variation on the morphism $d_S$ previously introduced, with the difference being that it `reprepares' an arbitrary state in the discarded objects. From terminality of the unit object, one then has 
$$\forall i \in S, S(i) \circ \tilde{S}(i) \circ S(i) = \discard \circ \mu_i \circ \discard = \discard = S(i) \, ;$$

As this is also the case for $i \in \texttt{label}(K) \setminus S$, we find
$$\forall S, d_S \circ e_S \circ d_S = \otimes_i S_i \circ \tilde{S}_i \circ S_i = \otimes_i S_i = d_S \, .$$

By the same method, one can also prove that $S \subseteq T \implies d_T = d_T \circ e_S \circ d_S$.

We now consider $f \in \mathcal{L}_{\mathbf{Rel}_{\mathbf{C}}}(\tau: K \rightarrow L)$ and $T \subseteq \texttt{label}(L)$. We note $\{i | \tau_i \subseteq T\} = \{i_1, \dots, i_n\}$. The precious equations entail
$$    d_T \circ f = d_T \circ e_{\tau_{i_1}} \circ d_{\tau_{i_1}} \circ f = d_T \circ e_{\tau_{i_1}} \circ f_{i_1}' \circ d_{\{i_1\}} \, .$$

Furthermore it can also be computed that $f_{i_1}' = d_{\tau_{i_1}} \circ f \circ e_{\{i_1\}}$; thus one further has $d_T \circ f = d_T \circ e_{\tau_{i_1}} \circ d_{\tau_{i_1}} \circ f \circ e_{\{i_1\}} \circ d_{\{i_1\}} = d_T \circ f \circ e_{\{i_1\}} \circ d_{\{i_1\}}$. As the same can be said of all of the $i_k$'s, this leads to $d_T \circ f = d_T \circ f \circ e_{\{i_1\}} \circ d_{\{i_1\}} \circ \dots \circ e_{\{i_n\}} \circ d_{\{i_n\}} = d_T \circ f \circ e_{\{i_1, \dots, i_n \}} \circ d_{\{i_1, \dots, i_n \}}$. Taking $f_T' := d_T \circ f \circ e_{\{i_1, \dots, i_n \}}$ ends the proof.

\end{proof}

With this lemma in mind, we can prove composability.

\begin{proof}

We now show that relational constraints, defined as in Example \ref{example:SemicartesianRelational}, are composable. Consider $f \in \mathcal{L}_{\mathbf{Rel}_{\mathbf{C}}}(\tau: K \rightarrow L)$ and $g \in \mathcal{L}_{\mathbf{Rel}_{\mathbf{C}}}(\sigma: L \rightarrow M)$, and take $i \in \texttt{label}(K)$. Then
\be \begin{split}
d_{(\sigma \circ \tau)_i} \circ (g \circ f) &= d_{(\sigma \circ \tau)_i} \circ g \circ e_{\{j| \sigma_j \subseteq (\sigma \circ \tau)_i\}} \circ d_{\{j| \sigma_j \subseteq (\sigma \circ \tau)_i\}} \circ f \\
&= \underbrace{d_{(\sigma \circ \tau)_i} \circ g \circ e_{\{j| \sigma_j \subseteq (\sigma \circ \tau)_i\} \setminus \tau_i} \circ d_{\{j| \sigma_j \subseteq (\sigma \circ \tau)_i\}\setminus \tau_i} \circ e_{\tau_i}}_{\tilde{g}_i} \circ d_{\tau_i} \circ f \\
&= \tilde{g}_i \circ d_{\tau_i} \circ f \\
&= \tilde{g}_i \circ f_i' \circ d_{\{i\}}\,.
\end{split} \ee
Therefore, $g \circ f \in \mathcal{L}_{\mathbf{Rel}_{\mathbf{C}}}(\sigma \circ \tau)$. That the constraints are also monoidal can be proven in an analogous way.
\end{proof}
As mentioned in the preamble to this section, relational constraints capture a multitude of different notions depending on the semi-cartesian structures which are chosen.
\begin{example}[Signalling constraints]
 The category of completely positive trace-preserving maps represents the most general possible causal evolutions of quantum degrees of freedom. In this category, made semicartesian by the standard tensor product (with the unique effect given by the trace operation), relational constraints indeed encode signalling relations.
\end{example}
In the case of semiadditive categories, relational constraints actually recover sectorial constraints.

\begin{lemma}[Sectorial constraints]
In a semiadditive category equipped with the monoidal structure given by the biproduct, relational constraints are (equivalent to) sectorial constraints. 
\end{lemma}
\begin{proof}
First, note that in such a category the unit object (better referred to as the zero object here) is both initial and terminal. Note also that $e_S \circ d_S = \sum_{l \not\in S} i^l \circ p^l $, where  the $p^l$'s are the projections $\oplus_k H^k \to H^l$, and the $i^l$'s are the injections $H^l \to \oplus_k H^k$.

Suppose now that $f \in \cl_{\cat{Sec}_{\cat{C}}}(\tau) $. Then $f = \id \circ f \circ \id = \sum_{k,l} i^l \circ p^l \circ f \circ i^k \circ p^k = \sum_{k,l| k \overset{\tau}{\sim} l} i^l \circ p^l \circ f \circ i^k \circ p^k$. Therefore, fixing an arbitrary $k$, $d_{\tau_k} \circ f = \left( \sum_{l| k \overset{\tau}{\not\sim} l} i^{l} \circ p^{l} \right) \circ \left(\sum_{k',l'| k' \overset{\tau}{\sim} l'} i^{l'} \circ p^{l'} \circ f \circ i^{k'} \circ p^{k'} \right) = \sum_{l| k \overset{\tau}{\not\sim} l} i^{l} \circ p^{l} \circ f \circ \sum_{k'| k' \overset{\tau}{\sim} l}i^{k'} \circ p^{k'} = \sum_{l| k \overset{\tau}{\not\sim} l} i^{l} \circ p^{l} \circ f \circ d_{\{ k'| k' \overset{\tau}{\not\sim} l\}} $. Now for all $l$ such that $k \overset{\tau}{\not\sim} l$, one has $\{k\} \subseteq \{ k'| k' \overset{\tau}{\not\sim} l\}$, and therefore $d_{\{ k'| k' \overset{\tau}{\not\sim} l\}} = d_{\{ k'| k' \overset{\tau}{\not\sim} l\}} \circ e_{\{k\}} \circ d_{\{k\}}$, which leads to $d_{\tau_k} \circ f = f_k' \circ d_{\{k\}}$. Therefore, $f \in \cl_{\cat{Rel}_{\cat{C}}}(\tau) $.

Reciprocally, suppose $f \in \cl_{\cat{Rel}_{\cat{C}}}(\tau) $, and take $k,l$ such that $k \overset{\tau}{\not\sim} l$. Then $p^l \circ f \circ i^k = p^l \circ i^l \circ p^l \circ f \circ i^k = p^l \circ d_{\{l' \neq l\}} \circ f \circ i^k$. Yet $\tau_k \subseteq \{l' \neq l\}$, so $p^l \circ f \circ i^k = p^l \circ d_{\{l' \neq l\}} \circ e_{\tau_k} \circ d_{\tau_k} \circ f \circ i^k = p^l \circ d_{\{l' \neq l\}} \circ e_{\tau_k} \circ f_k' \circ d_{\{k\}} \circ i^k = p^l \circ d_{\{l' \neq l\}} \circ e_{\tau_k} \circ f_k' \circ \left( \sum_{k' \neq k} i^{k'} \circ p^{k'} \circ i^k \right) = 0$, as all the terms in the bracketed sum are null. Therefore, $f \in \cl_{\cat{Sec}_{\cat{C}}}(\tau) $.
\end{proof}

Finally, we note that even though we focused on semicartesian categories here, all of our constructions could be symmetrically applied to cosemicartesian categories, i.e. categories in which the unit is initial. Indeed one can then see $\cat{Rel}$ as encoding constraints for the semicartesian category $\cat{C}^\textrm{op}$, meaning that $\cat{Rel}^\textrm{op} \cong \cat{Rel}$ encodes constraints for $\cat{C}$, of the form: 

\be \tikzfigscale{1.0}{signal_time_sym} \,.\ee

For instance, this yields a theory of relational constraints for the category of sets and functions with respect to the coproduct (i.e. disjoint union), which can be thought of as constraints on which alternatives among the possible inputs of the function can lead to which alternatives among its possible outputs.

\begin{example}[Relational constraint satisfaction]
The prototypical category of sets and functions is equipped with an initial unit object $\emptyset$ for the monoidal structure given by the coproduct  (disjoint union). Using this structure, the corresponding relational constraint can be parsed in the following way. Taking $\tau(i)$ to be the set of indices in the output of $\tau$ which are related to entry $i$ by $\tau$, then
\begin{align*}
    f:\cup_{i}X_{i} \rightarrow \cup_{j}Y_{j} \in \mathcal{L}(\tau:X \rightarrow Y) \iff \forall x \in X_{i} : f(x) \in \cup_{j \in \tau(i)} Y_{j} \,.
\end{align*}
Notably, for lists $X$ and $Y$ built from singletons, this simplifies to 
\begin{align*}
    f:\cup_{i}\{x_{i}\} \rightarrow \cup_{j}\{y_{j}\} \in \mathcal{L}(\tau:X \rightarrow Y) \iff f(x_i) \in \cup_{j \in \tau(i)} \{y_{j}\} \,.
\end{align*}
So for the empty set and the disjoint union, relational constraints simply express some imperfect knowledge about the output elements to which each input is sent.
\end{example}
In principle, one could view the above constraints as essentially constraint satisfaction problems. However, the specification of a constraint to be satisfied by a single explicit relation means all that is required is to check that such a relation is total. A precise definition of constraint satisfaction problems and their expression in quantaloidal and general categorical settings can be found in \cite{fujii2021quantaloidal, Morton2015GeneralizedCircuits}. In Section \ref{sec:CSP}, we will present a faithful generalisation of the spirit of constraint satisfaction problems to a general categorical setting, in which a composition rule is defined at the level of the individual constraints which are specified within a problem.

\subsection{Time-symmetric relational constraints and their properties} \label{sec:TimeSymmetry}

When the unit object is both initial and terminal, the possibility to encode constraints in both the backward direction and the forward one leads to the question of whether those options are equivalent procedures. When they indeed are, it leads to valuable simplifications. Let us first define what we mean by such an equivalence.

\begin{definition} \label{def:TimeSymmetric}
Let $(\mathbf{C},\otimes, I)$ be a symmetric semicartesian and cosemicartesian category. We equip $\cat{Rel}$ with its standard dagger structure given by the relational converse, and we call ${\cl'}_{\cat{Rel}_{\cat{C}}}$ the constraint encoding on $\cat{C}$ obtained from the fact that $\cat{Rel}$ encodes constraints for $\cat{C}^\textrm{op}$. We say that relational constraints on $\mathbf{C}$ are \emph{time-symmetric} if 

\be \forall \tau, \left(\cl_{\cat{Rel}_{\cat{C}}}(\tau) \right)^\textrm{op}:= \{f^\textrm{op} | f \in \cl_{\cat{Rel}_{\cat{C}}}(\tau) \} = {\cl'}_{\cat{Rel}_{\cat{C}}}(\tau^\dagger) \,. \ee

\end{definition}

Two subcategories of the category of quantum channels which have terminal and initial monoidal units are the category $\mathbf{unitalQC}$ of `unital quantum channels` (quantum channels which send maximally mixed states to maximally mixed states) and the category $\mathbf{mixU}$ of `unitary channels and noise' (generated by unitary channels together with completely mixed states and discarding effects). Whilst relational constraints for $\mathbf{mixU}$ are time-symmetric, relational constraints for $\mathbf{unitalQC}$ are not; this is why a simplification can be made for the presentation of relational constraints over $\mathbf{mixU}$, but not for $\mathbf{unitalQC}$.

\begin{lemma} \label{lem:SemicartWithInitial}
If a relational theory of constraints is time-symmetric, the condition for the satisfaction of constraints can be simplified: the action of $\mathcal{L}_{Rel_\cat{C}}$, as defined in (\ref{eq:SatisfactionRelConstraints}), can be rewritten as
 \be \label{eq:SactisfactionRelConstraints when time-symmetric} \mathcal{L}_{\mathbf{Rel}_{\mathbf{C}}}(\tau: K \rightarrow L) = \{f \, | \, \forall i,j \textrm{ s.t } i \overset{\tau}{\not\sim} j, \, \exists f'_{ij} \textrm{ s.t } :d_{\texttt{label}(L) \setminus \{j\}} \circ f = f_{ij}' \circ d_{\{i\}}\} \, . \ee

\end{lemma}

\begin{proof}
Let us call $\tilde{\mathcal{L}}_{Rel_\cat{C}}(\tau)$ the set defined in (\ref{eq:SactisfactionRelConstraints when time-symmetric}), and prove that, when the relational constraints are time-symmetric, it is equal to $\mathcal{L}_{Rel_\cat{C}}(\tau)$ as defined in (\ref{eq:SatisfactionRelConstraints}). First, suppose $f \in \mathcal{L}_{Rel_\cat{C}}(\tau)$. Then, for given $i$ and $j$ not related by $\tau$, one has by Lemma \ref{lemma:Domain Atomicity}: $d_{\texttt{label}(L) \setminus \{j\}} \circ f = f \circ d_{\{i'| \tau_{i'} \subseteq \texttt{label}(L) \setminus \{j\}\}}$; yet $\{i\} \subseteq \{i'| \tau_{i'} \subseteq \texttt{label}(L) \setminus \{j\}\}$, so $d_{\texttt{label}(L) \setminus \{j\}} \circ f = f \circ d_{\{i'| \tau_{i'} \subseteq \texttt{label}(L) \setminus \{j\}\}} \circ e_{\{i\}} \circ d_{\{j\}}$ and thus $f  \in \tilde{\mathcal{L}}_{Rel_\cat{C}}(\tau)$. Therefore $\mathcal{L}_{Rel_\cat{C}}(\tau) \subseteq \tilde{\mathcal{L}}_{Rel_\cat{C}}(\tau)$; note that this inclusion doesn't depend on the constraints being time-symmetric.

Suppose now that the relational constraints are time-symmetric, and take $f \in \tilde{\mathcal{L}}_{Rel_\cat{C}}(\tau)$. Fix an $i \in \texttt{label}(K)$. Then, for $j$ not related to $i$ by $\tau$, one has $f \in \mathcal{L}_{Rel_\cat{C}}(\lambda_{ij})$, where $\lambda_{ij} : (K_i, \bigotimes_{i' \neq i} K_{i'}) \to (L_j, \bigotimes_{j' \neq j} L_{j'})$ is defined by $K_i \overset{\lambda_{ij}}{\not\sim} L_j$. Therefore $f^\tr{op} \in \cl_{\cat{Rel}_{\cat{C}^\textrm{op}}}(\lambda_{ij}^\dagger)$, so $e_{\texttt{label}(K) \setminus \{i\}}^\tr{op} \circ f^\tr{op} = \tilde{f}_{ij}^\tr{op} \circ e_{\{j\}}^\tr{op}$. As this is true for all the $j \not\in \tau_i$, this leads to  $e_{\texttt{label}(K) \setminus \{i\}}^\tr{op} \circ f^\tr{op} = F_{ij}^\tr{op} \circ e_{\texttt{label}(L) \setminus \tau_i}^\tr{op}$. Therefore $f^\tr{op} \in \mathcal{L}_{\cat{Rel}_{\cat{C}^\textrm{op}}}(\mu_{i}^\dagger)$, where $\mu_i: (K_i, \bigotimes_{i' \neq i} K_{i'}) \to (\bigotimes_{j \in \tau_i j} L_{j'}, \bigotimes_{j \not\in \tau_i j} L_{j'})$ is defined by $K_i \overset{\mu_{i}}{\not\sim} \bigotimes_{j \not\in \tau_i} L_{j'}$. Thus $f \in \mathcal{L}_{Rel_\cat{C}}(\mu_{i})$, i.e. $d_{\tau_i} \circ f = f_i' \circ d_{\{i\}}$. As this is true for any $i$, $f  \in \mathcal{L}_{Rel_\cat{C}}(\tau)$.
\end{proof}

Lemma \ref{lem:SemicartWithInitial} has an important interpretation: is sufficient to check \textit{separately} each absence of arrow in order to verify that a morphism satisfies the constraint given by a relation. The constraints thus become collection of independent statements, each of these statements being the absence of one arrow in the relation. This important property, as well as the connection with discussions of causal structure in quantum channels, will be further discussed and formalised in Section \ref{sec:Intersectable Constraints}.

It is worth noticing that the proof of Lemma \ref{lem:SemicartWithInitial} essentially relies on the fact that, when the relational theory of constraints is time-symmetric, then by symmetry one does not only have domain-atomicity, but also \textit{codomain-atomicity}: if several output elements are not connected to some subset of inputs, then one can say that the monoidal product of these output elements isn't connected to this subset of inputs either. The conjunction of domain-atomicity and codomain-atomicity is what leads to the simplified form (\ref{eq:SactisfactionRelConstraints when time-symmetric}) for constraints.

\section{Intersectable constraints}\label{sec:Intersectable Constraints}

Whenever talking about constraints, a natural notion arises: that of intersecting two constraints (i.e. asking that both be satisfied). This notion of \textit{intersectability} has a natural description in the existence of meets (given by the intersection of sets) with respect to the locally posetal structure of the powerset category $\cp[\cat{C}]$. It is then natural to ask whether this notion can also be expressed in the constraint category \cat{Con}. For instance, relational constraints admit a simple notion of intersection:

\be \tikzfigscale{1.0}{meet} \,,\ee

so it is natural to ask whether this intersection simulates the following intersection of sets within the powerset category:

\be \tikzfigscale{1.0}{meet_2} \,. \ee

Here, we will describe how such a property can be formalised in the abstract, how this simplifies the description of constraints, and whether this property is satisfied in our main examples.

Seeing constraints as forming a partially ordered set, the ability to combine them corresponds to that of taking their \textit{meet} $\wedge$. Categorically, meets can be simply understood as products in poset categories. This leads to the following definition.

\begin{definition}
A 2-poset \cat{D} is a \emph{2-meet-semilattice} if for every $A,B$ the poset $\cat{D}(A,B)$ has products.
\end{definition}

    In particular, the powerset category $\mathcal{P}[\cat{C}]$ is a 2-meet-semilattice, with the meet given by the intersection of sets $\cap$. The question of having a faithful way of intersecting constraints then becomes that of having a 2-meet-semilattice structure on \cat{Con} that is compatible with that of $\mathcal{P}[\cat{C}]$. 

\begin{definition}
An \emph{intersectable} composable constraint for a category \cat{C} is given by a lax 2-functor $\cl$ from a 2-meet-semilattice \cat{Con} into the 2-meet-semilattice $\mathcal{P}[\cat{C}]$, such that products are preserved:

\be \forall \tau, \sigma, \quad \cl(\tau \wedge \sigma) = \cl(\tau) \cap \cl(\sigma) \,. \ee
\end{definition}

Note that for any composable constraint, the inclusion $\cl(\tau \wedge \sigma) \subseteq \cl(\tau) \cap \cl(\sigma)$ is guaranteed because of the lax 2-functoriality; what is at stake here is the reverse inclusion.

A major advantage of working with intersectable constraints is that they can be checked by solely looking at \textit{meet-generators} of the hom-set semilattice. A similar notion can be found in the study of generative effects from a categorical perspective \cite{adam2019generativity,fong2018seven}, where the compatibility of an observation with the action of joining systems together is expressed in terms of the preservation of joins by that observation. In a given meet-semilattice, a set $\{\sigma_i\}_{i \in I}$ of elements is a set of meet-generators if any other element is a meet of a subset of them, i.e., $\forall \tau, \exists I_\tau \subseteq I, \tau = \bigwedge_{i \in I_\tau} \sigma_i$. Typically, for finite relations, the relations with exactly one arrow missing meet-generate the hom-sets.

\begin{example} \label{example:GeneratorsRel}
In the category \Rel of finite relations, any hom-set $\Rel(A,B)$ is meet-generated by $\{\sigma_{ab}\}_{a \in A, b \in B}$, where for given $a,b$, $\sigma_{ab}$ is the relation corresponding to the subset $A \times B \setminus \{(a,b)\}$ of $A \times B$. 
\end{example}

\begin{proof}
Any relation $\tau: A \to B$, corresponding to a subset $\tau \subseteq A \times B$, satisfies $\tau = \bigwedge_{(a,b) \not\in \tau} \sigma_{ab}$.
\end{proof}

The point is now that any constraint in an intersectable theory of constraints is nothing more than the combination of the constraints corresponding to the meet-generators: it becomes a collection of independent elementary statements.

\begin{lemma}
If \cat{Con} is a theory of intersectable composable constraints for \cat{C}, then for a hom-set $\cat{Con}(A,B)$ and meet-generators $\sigma_i$ of this hom-set, one has $\cl(\tau = \bigwedge_{i \in I_\tau} \sigma_i) = \bigcap_{i \in I_\tau} \cl(\sigma_i)$.
\end{lemma}
\begin{proof}
Direct.
\end{proof}

Such a result greatly simplifies, in practice, the computation of whether a given morphism satisfies a given constraint. In particular, when the theory of constraints is that of finite relations, it is sufficient to check separately each absence of arrow whenever the constraints are intersectable. This is possible in some cases, although not for all relational constraints.

\begin{example}
Sectorial constraints are intersectable.
\end{example}
\begin{proof}
This is baked into their definition in Example \ref{example:Sectorial Constraints}, which can be rewritten as $\cl(\lambda) = \bigcap_{(a,b) \not\in \lambda} \cl(\sigma_{ab})$, with $\forall a,b, \, \cl(\sigma_{ab}) := \{f \, | \, p_B^l \circ f \circ i_A^k = 0\}$: i.e., we already defined sectorial constraints as intersections of the constraints given by the meet-generators of the hom-sets described in Example \ref{example:GeneratorsRel}.
\end{proof}

\begin{example}
Relational constraints are not necessarily intersectable.
\end{example}
\begin{proof}
A counter-example is the theory of relational constraints (better called signalling constraints in this case) for noisy quantum channels, with respect to the tensor product \cite{barrett2019}: looking at the constraints on channels $A \to B_1 \otimes B_2$, and taking, for $i \in \{1,2\}$, $\sigma_i$ to be the relation not featuring an arrow $A \to B_i$, there exist channels in $ \cl(\sigma_1) \cap \cl(\sigma_2) \setminus \cl(\sigma_1 \wedge \sigma_2)$, for example the classical channel sending $\ket{0}\bra{0}$ to $\ket{0}\bra{0} \otimes \ket{0}\bra{0} + \ket{1}\bra{1} \otimes \ket{1}\bra{1}$ and $\ket{1}\bra{1}$ to $\ket{0}\bra{0} \otimes \ket{1}\bra{1} + \ket{1}\bra{1} \otimes \ket{0}\bra{0}$.
\end{proof}
The above counterexample is a member of the theory $\mathbf{unitalQC}$ of `unital quantum channels', a theory for which relational constraints are not time-symmetric.  
\begin{theorem}
If a theory of relational constraints is time-symmetric, then these constraints are intersectable.
\end{theorem}
\begin{proof}
This can be seen in Lemma \ref{lem:SemicartWithInitial}, which in fact directly proves that in this case relational constraints can be expressed as intersections of the constraints given by the meet-generators.
\end{proof}

In particular, the discussion of \cite{barrett2019} on the causal structure of unitary quantum channels can be recast into the statement that relational constraints for the theory $\mathbf{mixU}$ of `unitary channels and noise' are time-symmetric and therefore intersectable. An important conceptual consequence in this case is that the causal structure of unitary channels can be fully described by the absence of causal connections between individual input and output factors: the absences of causal connections between \textit{groups} of input factors and groups of output factors can be derived from this data via the intersectability of the signalling constraints. As mentioned in Section \ref{sec:TimeSymmetry}, this can also be ascribed to the presence of both domain-atomicity and codomain-atomicity in this case, the latter being derivable from the former by time-symmetry.

\section{General categorical constructions} \label{sec:General Constructions}
Finally, we present generalisations in the language of classical category theory both for relational constraints and for constraint satisfaction problems. In turn we present the latter in their usual form.

\subsection{Generalised approach to relational constraints}
The construction of relational constraints can be recognised as a consequence of a higher level categorical construction. We will first outline the key categorical features of the construction of relational constraints which can be used to construct abstract composable constraints, and then prove that those features are indeed sufficient for such a construction. Let $\mathbf{Con}$ be an arbitrary category, and let $\mathbf{Cat}_i$ denote the full subcategory of $\mathbf{Cat}$ whose objects are categories with initial objects. A functor of type $\mathcal{K}:\mathbf{Con} \longrightarrow \mathbf{Cat}_i$ assigns to each object of $\mathbf{Con}$ a category with an initial object. For example, for any semicartesian monoidal category $\mathbf{C}$ there is a functor $\mathcal{K}:\mathbf{Rel}_{C}^{op} \longrightarrow \mathbf{Cat}_i$ which assigns to each list $H := [H_1, \dots , H_n]$ the category with objects given by sublists of $H$ and morphisms given by inclusion. On morphisms, $\mathcal{K}$ sends $\tau:A \rightarrow B$ to the functor $\mathcal{K}(\tau):\mathcal{K}(B) \rightarrow \mathcal{K}(A)$ given on objects by $\mathcal{K}(\tau)(S):= \texttt{pre}_{\tau}(S) := \{i:\forall j \textrm{ s.t } i \overset{\tau}{\sim} j \textrm{ then }j \in S\}$ and sending each unique inclusion $S \subseteq T$ to the unique inclusion $\texttt{pre}_{\tau}(S) \subseteq \texttt{pre}_{\tau}(T)$.

\be \begin{tikzpicture}
	\begin{pgfonlayer}{nodelayer}
		\node [style=whitephase] (36) at (-9.25, 1.25) {${H'}_1$};
		\node [style=whitephase] (37) at (-7.75, 1.25) {${H'}_2$};
		\node [style=whitephase] (38) at (-6.25, 1.25) {${H'}_3$};
		\node [style=whitephase] (39) at (-9.25, -1.25) {$H_1$};
		\node [style=whitephase] (40) at (-7.75, -1.25) {$H_2$};
		\node [style=whitephase] (41) at (-6.25, -1.25) {$H_3$};
		\node [style=none] (42) at (-7.75, 0) {$\tau$};
		\node [style=whitephase] (43) at (-3.25, 1) {${H'}_1$};
		\node [style=whitephase] (45) at (-2.25, 1) {${H'}_3$};
		\node [style=whitephase] (46) at (-2.75, -1.5) {${H'}_3$};
		\node [style=whitephase] (47) at (1.75, 1.25) {$H_1$};
		\node [style=whitephase] (48) at (2.5, 1.25) {$H_2$};
		\node [style=whitephase] (49) at (3.25, 1.25) {$H_3$};
		\node [style=whitephase] (50) at (2.125, -1.5) {$H_2$};
		\node [style=whitephase] (51) at (2.875, -1.5) {$H_3$};
		\node [style=none] (52) at (-4.25, 2.5) {};
		\node [style=none] (53) at (-1.25, 2.5) {};
		\node [style=none] (54) at (-1.25, -2.25) {};
		\node [style=none] (55) at (-4.25, -2.25) {};
		\node [style=none] (56) at (0.75, 2.5) {};
		\node [style=none] (57) at (0.75, -2.25) {};
		\node [style=none] (58) at (3.75, -2.25) {};
		\node [style=none] (59) at (3.75, 2.5) {};
		\node [style=none] (60) at (0.75, -2.25) {};
		\node [style=none] (61) at (-2.75, 0.5) {};
		\node [style=none] (62) at (-2.75, -1) {};
		\node [style=none] (63) at (2.5, -1) {};
		\node [style=none] (64) at (2.5, 0.75) {};
		\node [style=none] (65) at (-2.25, -1.5) {};
		\node [style=none] (66) at (1.75, -1.5) {};
		\node [style=none] (67) at (1.5, 1.25) {};
		\node [style=none] (68) at (-1.75, 1) {};
		\node [style=none] (69) at (0, -0.25) {$\mathcal{K}_{\tau}$};
		\node [style=none] (71) at (-3.5, 2) {$\mathcal{K}(H')$};
		\node [style=none] (72) at (1.5, 2) {$\mathcal{K}(H)$};
		\node [style=none] (73) at (-3.75, 1.25) {};
		\node [style=none] (74) at (-1.75, 1.25) {};
		\node [style=none] (75) at (-1.75, 0.75) {};
		\node [style=none] (76) at (-3.75, 0.75) {};
		\node [style=none] (77) at (1.5, 1.5) {};
		\node [style=none] (78) at (1.5, 1) {};
		\node [style=none] (79) at (3.5, 1) {};
		\node [style=none] (80) at (3.5, 1.5) {};
		\node [style=none] (81) at (1.75, -1.25) {};
		\node [style=none] (82) at (1.75, -1.75) {};
		\node [style=none] (83) at (3.25, -1.75) {};
		\node [style=none] (84) at (3.25, -1.25) {};
		\node [style=none] (85) at (-3.25, -1.25) {};
		\node [style=none] (86) at (-2.25, -1.25) {};
		\node [style=none] (87) at (-2.25, -1.75) {};
		\node [style=none] (88) at (-3.25, -1.75) {};
		\node [style=none] (89) at (-5.25, 0) {$:$};
		\node [style=none] (90) at (-2.5, -0.25) {};
		\node [style=none] (91) at (2.25, 0) {};
	\end{pgfonlayer}
	\begin{pgfonlayer}{edgelayer}
		\draw (36) to (40);
		\draw (39) to (36);
		\draw (37) to (41);
		\draw (38) to (41);
		\draw (40) to (38);
		\draw [style=dottededge] (56.center) to (59.center);
		\draw [style=dottededge] (59.center) to (58.center);
		\draw [style=dottededge] (60.center) to (58.center);
		\draw [style=dottededge] (60.center) to (56.center);
		\draw [style=dottededge] (53.center) to (54.center);
		\draw [style=dottededge] (54.center) to (55.center);
		\draw [style=dottededge] (55.center) to (52.center);
		\draw [style=dottededge] (52.center) to (53.center);
		\draw [style=new edge style 0] (61.center) to (62.center);
		\draw [style=new edge style 0] (64.center) to (63.center);
		\draw [style=new edge style 4, bend right=15] (66.center) to (65.center);
		\draw [style=new edge style 4, bend right=15] (67.center) to (68.center);
		\draw [style=new edge style 1] (73.center) to (76.center);
		\draw [style=new edge style 1] (76.center) to (75.center);
		\draw [style=new edge style 1] (75.center) to (74.center);
		\draw [style=new edge style 1] (74.center) to (73.center);
		\draw [style=new edge style 1] (79.center) to (78.center);
		\draw [style=new edge style 1] (78.center) to (77.center);
		\draw [style=new edge style 1] (77.center) to (80.center);
		\draw [style=new edge style 1] (80.center) to (79.center);
		\draw [style=new edge style 1] (84.center) to (83.center);
		\draw [style=new edge style 1] (83.center) to (82.center);
		\draw [style=new edge style 1] (82.center) to (81.center);
		\draw [style=new edge style 1] (81.center) to (84.center);
		\draw [style=new edge style 1] (85.center) to (88.center);
		\draw [style=new edge style 1] (88.center) to (87.center);
		\draw [style=new edge style 1] (87.center) to (86.center);
		\draw [style=new edge style 1] (86.center) to (85.center);
		\draw [style=new edge style 4, bend right=15] (91.center) to (90.center);
	\end{pgfonlayer}
\end{tikzpicture}\ee 

We interpret $\mathcal{K}(A)$ as being the category which represents all of the choices of subsystems that could be removed from $A$, with each arrow $S \rightarrow S'$ representing the instruction to transition from having removed $S$ to having furthermore removed $S'$. In our example, the initial object of $\mathcal{K}(A)$ is the empty set $\emptyset$. Initiality means that there is one and only one arrow $\emptyset \rightarrow S$ for each $S$: intuitively, there is only one way to throw away each system.

The second categorical ingredient required is the notion of a cone. The functor $\mathcal{K}:\mathbf{Rel}_{C}^{op} \longrightarrow \mathbf{Cat}_{i}$ can be completed by the inclusion $ \mathbf{Cat}_i \subseteq \mathbf{Cat}$ and the inclusion $\texttt{ob}[\mathbf{Rel}_{C}^{op}] \subseteq \mathbf{Rel}_{C}^{op}$ to a functor $\mathcal{K}_{\subseteq}:\texttt{ob}[\mathbf{Rel}_{C}^{op}] \longrightarrow \mathbf{Cat}$, where $\texttt{ob}[\mathbf{M}]$ for some category $\mathbf{M}$ is the trivial discrete category with objects given by the objects of $\mathbf{M}$. In our previous example, there exists a cone over this functor, with its apex given by $\mathbf{C}$. Unpacking this claim, there is a functor $\mathcal{F}_{A}: \mathcal{K}(A) \longrightarrow \mathbf{C}$ for every object $A$ of $\mathbf{Rel}_{C}$: for some fixed $A$ this functor is given in our example by assigning to each object $S$ of $\mathcal{K}(A)$ the object $\mathcal{F}_{A}(S) = \otimes_{A_i \in \bar{S}}A_i$, where $\bar{S}$ is the complement of $S$ in $A$, and by assigning to each morphism $S \subseteq T$ in $\mathcal{K}(A)$ a morphism $\mathcal{F}_{A}(S \subseteq T):\otimes_{A_i \in \bar{S}}A_i \longrightarrow \otimes_{A_i \in \bar{S}}A_i$ given by discarding all of those elements of $T$ which are not in $S$:

\be \begin{tikzpicture}
	\begin{pgfonlayer}{nodelayer}
		\node [style=whitephase] (43) at (-3.25, 1) {${H'}_1$};
		\node [style=whitephase] (45) at (-2.25, 1) {${H'}_3$};
		\node [style=whitephase] (46) at (-2.75, -1.5) {${H'}_3$};
		\node [style=none] (52) at (-4.25, 2.5) {};
		\node [style=none] (53) at (-1.25, 2.5) {};
		\node [style=none] (54) at (-1.25, -2.25) {};
		\node [style=none] (55) at (-4.25, -2.25) {};
		\node [style=none] (56) at (0.75, 2.5) {};
		\node [style=none] (57) at (0.75, -1.75) {};
		\node [style=none] (58) at (3.5, -2.25) {};
		\node [style=none] (59) at (3.5, 2.5) {};
		\node [style=none] (60) at (0.75, -2.25) {};
		\node [style=none] (61) at (-2.75, 0.5) {};
		\node [style=none] (62) at (-2.75, -1) {};
		\node [style=none] (65) at (-2.25, -1.5) {};
		\node [style=none] (66) at (1.25, -1.25) {};
		\node [style=none] (67) at (2.25, 0.75) {};
		\node [style=none] (68) at (-1.75, 1) {};
		\node [style=none] (69) at (0, -0.25) {$\mathcal{F}_{H'}$};
		\node [style=none] (71) at (-3.5, 2) {$\mathcal{K}(H')$};
		\node [style=none] (72) at (1.5, 2) {$\mathbf{C}$};
		\node [style=none] (73) at (-3.75, 1.25) {};
		\node [style=none] (74) at (-1.75, 1.25) {};
		\node [style=none] (75) at (-1.75, 0.75) {};
		\node [style=none] (76) at (-3.75, 0.75) {};
		\node [style=none] (85) at (-3.25, -1.25) {};
		\node [style=none] (86) at (-2.25, -1.25) {};
		\node [style=none] (87) at (-2.25, -1.75) {};
		\node [style=none] (88) at (-3.25, -1.75) {};
		\node [style=none] (90) at (-2.5, -0.25) {};
		\node [style=none] (91) at (1.5, -0.25) {};
		\node [style=none] (100) at (2, -1.25) {${H'}_1$};
		\node [style=none] (101) at (2.75, -1.25) {${H'}_2$};
		\node [style=none] (113) at (2.75, 0.75) {${H'}_2$};
		\node [style=smalldiscarding] (116) at (2, 0) {};
		\node [style=none] (117) at (2, -0.75) {};
		\node [style=none] (119) at (2.75, 0.25) {};
		\node [style=none] (120) at (2.75, -0.75) {};
	\end{pgfonlayer}
	\begin{pgfonlayer}{edgelayer}
		\draw [style=dottededge] (56.center) to (59.center);
		\draw [style=dottededge] (59.center) to (58.center);
		\draw [style=dottededge] (60.center) to (58.center);
		\draw [style=dottededge] (60.center) to (56.center);
		\draw [style=dottededge] (53.center) to (54.center);
		\draw [style=dottededge] (54.center) to (55.center);
		\draw [style=dottededge] (55.center) to (52.center);
		\draw [style=dottededge] (52.center) to (53.center);
		\draw [style=new edge style 0] (61.center) to (62.center);
		\draw [style=new edge style 4, bend right=15] (66.center) to (65.center);
		\draw [style=new edge style 4, bend right=15] (67.center) to (68.center);
		\draw [style=new edge style 1] (73.center) to (76.center);
		\draw [style=new edge style 1] (76.center) to (75.center);
		\draw [style=new edge style 1] (75.center) to (74.center);
		\draw [style=new edge style 1] (74.center) to (73.center);
		\draw [style=new edge style 1] (85.center) to (88.center);
		\draw [style=new edge style 1] (88.center) to (87.center);
		\draw [style=new edge style 1] (87.center) to (86.center);
		\draw [style=new edge style 1] (86.center) to (85.center);
		\draw [style=new edge style 4, bend right=15] (91.center) to (90.center);
		\draw [style=double edge] (116) to (117.center);
		\draw [style=double edge] (119.center) to (120.center);
	\end{pgfonlayer}
\end{tikzpicture}\,.\ee 

This assignment is indeed functorial, that is, $\mathcal{F}_{A}(S \subseteq T \subseteq V) = \mathcal{F}_{A}(S \subseteq T) \circ \mathcal{F}_{A}(T \subseteq V)$. It turns out that the structure of a $\mathbf{C}$-valued cone over a functor which factors in the way mentioned above, is all that is required construct a composable constraint encoding which, for the specific functors mentioned above, reduces exactly to relational constraints for semicartesian categories.
\begin{theorem}[Cones over factored functors]
For every pair 
$(\mathbf{C},\mathbf{Con})$ of categories equipped with
\begin{itemize}
    \item a functor $\mathcal{K}: \mathbf{Con}^{op} \rightarrow \mathbf{Cat}_{i}$;
    \item a cone with apex $\mathbf{C}$ for the completion of the functor $\mathcal{K}$ to $ \mathcal{K}_{\subseteq} :\texttt{ob}[\mathbf{Con}_{C}^{op}] \longrightarrow \mathbf{Cat}$,
\end{itemize}
there exists is a composable constraint encoding $\mathcal{L}:\mathbf{Con} \longrightarrow \mathcal{P}[\mathbf{C}]$ defined by taking each object $A$ of $\mathbf{Con}$ to the object $\mathcal{L}(A) := \mathcal{F}_{A}(\mathcal{K}_{A}^{i})$ where $\mathcal{K}_{A}^{i}$ is the initial object of the category $K(B)$, and by taking each morphism $\tau : A \rightarrow B$ to the set $\mathcal{L}(\tau)$ of morphisms $f:\mathcal{F}_{A}(\mathcal{K}_{A}^{i}) \rightarrow \mathcal{F}_{B}(\mathcal{K}_{B}^{i})$ such that, for every $S$, there exists $f_{S}$ which makes the following square commute:
\be \begin{tikzcd}
\mathcal{F}_{C}(\mathcal{K}_{B}^{i}) \arrow[rr, "\mathcal{F}_{B}(u_S)", rightarrow]                                                  &  & \mathcal{F}_{B}(S)                                                 \\
\mathcal{F}_{A}(\mathcal{K}_{A}^{i}) \arrow[u, "f", rightarrow] \arrow[rr, "\mathcal{F}_{A}(u_{\mathcal{K}_{\tau}(S)})", rightarrow] &  & \mathcal{F}_{A}\mathcal{K}_{\tau}(S) \arrow[u, "f_S"', rightarrow]
\end{tikzcd}\,,\ee 
where for each object $X$ of a category equipped with an initial object, $u_X$ is the unique morphism from that initial object to $X$. 
\end{theorem}
\begin{proof}
We must show that if $f \in \mathcal{L}(\tau)$ and $g \in \mathcal{L}(\lambda)$ then $ f \circ g \in \mathcal{L}(\tau \circ \lambda)$. Indeed the following diagram commutes:
\be \begin{tikzcd}
\mathcal{F}_{C}(\mathcal{K}_{C}^{i}) \arrow[rr, "\mathcal{F}_{C}(u_S)", rightarrow]                                                                       &  & \mathcal{F}_{C}(S)                                                                         \\
\mathcal{F}_{B}(\mathcal{K}_{B}^{i}) \arrow[u, "f", rightarrow] \arrow[rr, "\mathcal{F}_{B}(u_{\mathcal{K}_{\tau}(S)})", rightarrow]                      &  & \mathcal{F}_{B}\mathcal{K}_{\tau}(S) \arrow[u, "f_S"', rightarrow]                         \\
\mathcal{F}_{A}(\mathcal{K}_{A}^{i}) \arrow[u, "g", rightarrow] \arrow[rr, "\mathcal{F}_{A}(u_{\mathcal{K}_{\lambda}\mathcal{K}_{\tau}(S)})", rightarrow] &  & \mathcal{F}_{C}\mathcal{K}_{\lambda}\mathcal{K}_{\tau}(S) \arrow[u, "g_{S}"', rightarrow]
\end{tikzcd} \,.\ee 
The commutativity of this diagram follows from the commutativivity of the diagrams for $f$ and $g$ individually, and since $\mathcal{K}$ is a contravariant functor then $\mathcal{K}_{\lambda} \circ \mathcal{K}_{\tau} = \mathcal{K}_{\tau \circ \lambda}$, so the bottom morphism is the correct one for the constraint $\tau \circ \lambda$. Next, we need to show that for each $2$-morphism $\alpha:\tau \rightarrow \tau'$ of $\mathbf{Con}$ one can assign a $2$-morphism $\mathcal{L}(\alpha):\mathcal{L}(\tau) \rightarrow \mathcal{L}(\tau')$; indeed, there can only be one arrow from $\mathcal{L}(\tau)$ to $\mathcal{L}(\tau')$ since the two morphisms of the power-set category are simply subset-inclusions; so all that is required is to show that indeed whenever there is a morphism $\alpha:\tau \rightarrow \tau'$ then it is the case that $\mathcal{L}(\tau) \subseteq \mathcal{L}(\tau')$. Since $\mathcal{K}$ is a contravariant $2$-functor, every morphism $\alpha:\tau \rightarrow \tau'$ defines a natural transformation $\mathcal{K}(\alpha):\mathcal{K}(\tau') \Rightarrow \mathcal{K}(\tau)$. This means in turn that for every object $S \in \mathcal{K}(A)$ there exists a morphism $\mathcal{K}(\alpha)_{S}:\mathcal{K}(\tau)(S) \rightarrow \mathcal{K}(\tau')(S)$, which can be used to show that $f \in \mathcal{L}(\tau)$ implies $f \in \mathcal{L}(\tau')$. Indeed for every such $\alpha$, if $f \in \mathcal{L}(\tau)$ then there exists $f_S$ such that the outer path of the following diagram commutes: \be \begin{tikzcd}
\mathcal{F}_{C}(\mathcal{K}_{B}^{i}) \arrow[rr, "\mathcal{F}_{B}(u_S)"]                                                                                                               &  & \mathcal{F}_{B}(S)                                                                                                                                  &  &                                                                                   \\
\mathcal{F}_{A}(\mathcal{K}_{A}^{i}) \arrow[u, "f"] \arrow[rr, "\mathcal{F}_{A}(u_{\mathcal{K}_{\tau'}(S)})"] \arrow[rrrr, "\mathcal{F}_{A}(u_{\mathcal{K}_{\tau}(S)})"', bend right] &  & \mathcal{F}_{A}\mathcal{K}_{\tau'}(S) \arrow[u, "f_S \circ \mathcal{F}_{A}\mathcal{K}_{\alpha}"'] \arrow[rr, "\mathcal{F}_{A}\mathcal{K}_{\alpha}"] &  & \mathcal{F}_{A}\mathcal{K}_{\tau}(S) \arrow[llu, "f_S"', bend right, shift right]
\end{tikzcd}\,.\ee 
The bottom path commutes by functorality of $\mathcal{F}_{A}$ and terminality of $\mathcal{K}_{A}^{i}$, and the top right path commutes by definition; as a result the top left path commutes and so there indeed exists a morphism $f_S \circ \mathcal{F}_{A} \mathcal{K}_{\alpha}$ which makes the required diagram commute for $f$ to be in $\mathcal{L}(\tau')$. 
\end{proof}


\subsection{Constraint satisfaction problems} \label{sec:CSP}
In a constraint satisfaction problem (CSP), a function $f:V \rightarrow D$ is said to satisfy constraint $(k,\mathbf{x} \in V^{k},\rho \subseteq D^{k})$ if $(f(\mathbf{x}_1), \dots , f(\mathbf{x}_n)) \in \rho$, and is more generally said to satisfy a set of constraints $\mathbf{K} \subseteq \{(k,\mathbf{x} \in V^{k},\rho \subseteq D^{k})\}$ if it individually satisfies each such constraint. Here we will show first that a category $\texttt{CSP}_{S}[\mathbf{C}]$ of generalised CSPs can be constructed from any category $\mathbf{C}$, and that there always exists a composable constraint from $\texttt{CSP}_{S}[\mathbf{C}]$ to $\mathbf{C}$. We will see furthermore how to recover set-based CSPs as a special case of categorical CSPs on $\mathbf{Set}$. The natural number $k$ can be interpreted as an instance of an endofunctor $k(-)::S \mapsto S^{k}$, the element $\mathbf{x} \in k(V)$ can be re-stated as a state $\mathbf{x}:I \rightarrow k(V)$ and generalised to a morphism $\mathbf{a}: A \rightarrow k(V)$, the subset $B \subseteq k(D)$ can be generalised to a morphism $\mathbf{b}:B \rightarrow k(D)$. Using these observations we define for every category $\mathbf{C}$ and source object $S$ the category $\texttt{CSP}_{S}[\mathbf{C}]$ by the following data: 
\begin{itemize}
    \item objects of $\texttt{CSP}_{S}[\mathbf{C}]$ given by objects of $\mathbf{C}$;
    \item morphisms from $V$ to $V'$ given by subsets $\cat{C} \subseteq \{(\mathcal{F},\mathbf{a}:A \rightarrow \mathcal{F}(V),\mathbf{b}:B \rightarrow \mathcal{F}(V')\}$ where each $\mathcal{F}$ is an endofunctor.
\end{itemize}
Composition in $\texttt{CSP}_{S}[\mathbf{C}]$ is given by 
\begin{align}
    \cat{C}' \circ \cat{C} :=  \{(\mathcal{F},\mathbf{a},\mathbf{d}) | & 
     \exists{(\mathcal{F},\mathbf{a},\mathbf{b}) \in \cat{C}} \quad s.t \quad \forall m:S \rightarrow dom(\mathbf{b}) \\ & \exists{(\mathcal{F},\mathbf{c},\mathbf{d})} \in \cat{C}' \quad and \quad n:S \rightarrow dom(\mathbf{c}) \\
     & s.t \quad  \mathbf{b} \circ m=\mathbf{c} \circ n \quad \}\,.
\end{align}
As a commutative diagram the last condition reads \be \begin{tikzcd}
                                         & \mathcal{F}(V')                               &                                           \\
dom(\mathbf{b}) \arrow[ru, "\mathbf{b}"] & {} \arrow[loop, distance=2em, in=125, out=55] & dom(\mathbf{c}) \arrow[lu, "\mathbf{c}"'] \\
                                         & S \arrow[lu, "m"] \arrow[ru, "n"']            &                                          
\end{tikzcd}\,.\ee 
One can then assign a composable constraint in the following way.
\begin{theorem}
There exists a lax functor $\mathcal{L}:\texttt{CSP}_{S}[\mathbf{C}] \rightarrow \mathcal{P}[\mathbf{C}]$ given on objects by $\mathcal{L}(V) = V$ and on morphisms by \begin{align*}
\mathcal{L}(\cat{C}) := \{f|\textrm{ }\forall{(\mathcal{F},\mathbf{a},\mathbf{b}) \in \cat{C}} \textrm{ and } S \xrightarrow{x} dom(\mathbf{a}) \textrm{, } \exists{S \xrightarrow{m} dom(\mathbf{b})} \textrm{ s.t the diagram } \\
\begin{tikzcd}[ampersand replacement=\&,
]
\mathcal{F}(V) \arrow[rr, "\mathcal{F}(f)"] \&                                               \& \mathcal{F}(V')                          \\
dom(\mathbf{a}) \arrow[u, "\mathbf{a}"]     \& {} \arrow[loop, distance=2em, in=125, out=55] \& dom(\mathbf{b}) \arrow[u, "\mathbf{b}"'] \\
                                            \& S \arrow[lu, "x"] \arrow[ru, "m"']            \&                                         
\end{tikzcd} \textrm{ commutes } \}\,.
\end{align*}

\end{theorem}
\begin{proof} That this is a composable constraint follows from the commutativity of the following diagram. For every $x$, the existence of an $m$ making the bottom left diagram commute is given by $f$ satisfying the constraint $\cat{C}$; the existence of $n$ making the bottom diagram commute is given by the definition of composition in \texttt{CSP}; and the existence of $z$ making the bottom right diagram commute is given by $g$ satisfying constraint $\cat{C}'$.
\be \begin{tikzcd}
                                                                                                                                         &                   & {} \arrow[loop, distance=2em, in=305, out=235]                                               &                                                &                                         \\
\mathcal{F}(V) \arrow[rr, "\mathcal{F}(f)", bend left] \arrow[rrrr, "\mathcal{F}(g \circ f)", bend left=49] & {} \arrow[loop, distance=2em, in=305, out=235] & \mathcal{F}(V') \arrow[rr, "\mathcal{F}(g)", bend left]                                      & {} \arrow[loop, distance=2em, in=305, out=235] & \mathcal{F}(V')                        \\
dom(\mathbf{a}) \arrow[u, "\mathbf{a}"']                                                                    & dom(\mathbf{b}) \arrow[ru, "\mathbf{b}"']      & {} \arrow[loop, distance=2em, in=125, out=55]                                                & dom(\mathbf{c}) \arrow[lu, "\mathbf{c}"]       & dom(\mathbf{d}) \arrow[u, "\mathbf{d}"] \\
                                                                                                            &                                                & S \arrow[lu, "m"'] \arrow[ru, "n"] \arrow[llu, "x"', bend left] \arrow[rru, "z", bend right] &                                                &                                        
\end{tikzcd}\ee  The upper commutative diagram is functoriality.
\end{proof}
The subcategory $\texttt{CSP}[\mathbf{Set}]$ of $\texttt{CSP}_{I}[\mathbf{Set}]$ defined by keeping only those morphisms $\cat{C}:V \rightarrow V'$ such that $(\mathcal{F},\mathbf{a},\mathbf{b}) \in \cat{C}$ implies that there is a $k$ such that $\mathcal{F}(-) = (-)^{k}$ and $dom(\mathbf{a}) = \{ \bullet \}$ is the category whose morphisms from $V$ to $V'$ are exactly the constraint satisfaction problems from $V$ to $V'$. To see this, note that the previous constraint definition reduces to the existence of an $m$ such that:
\be \begin{tikzcd}
V^k \arrow[rr, "f^k"]                &                                                 & {V'}^k                       \\
\{\bullet \} \arrow[u, "\mathbf{a}"] & {} \arrow[loop, distance=2em, in=125, out=55]   & \rho \arrow[u, "\subseteq"'] \\
                                     & \{ \bullet \} \arrow[lu, "id"] \arrow[ru, "m"'] &                             
\end{tikzcd}\,,\ee 
i.e. such that $f(a_1) \dots f(a_k) \in \rho$. The composition rule then reduces to the  condition \be  (k,\mathbf{a},\sigma) \in \mathcal{C}' \circ \mathcal{C} \iff \exists (k,\mathbf{a},\rho)\textrm{ s.t } \forall \mathbf{m} \in \rho \textrm{ } \exists (k,\mathbf{m},\sigma) \in \mathcal{C}' \ee  between standard constraint instances.

\section{Conclusion} \label{sec:conclusion}

In this work, we described \textit{composable constraint encoding}, a general notion that allows to endow a (possibly symmetric monoidal, or $\dagger$-compact, etc.) category \cat{Con} with the interpretation of representing constraints for the morphisms of another $\dagger$-symmetric monoidal (or $\dagger$-compact) category \cat{C}. We showed that, given such a composable constraint, one can combine \cat{C} and \cat{Con} into a \textit{constrained category} that features a parallel calculus, taking place both in \cat{Con} and in \cat{C}. In addition to being an intuitive representation for the user, the notion of a constrained category can be viewed as equivalent to the notion of a composable constraint. We presented several general examples, those of thin constraints, constraint satisfaction problems, sectorial constraints, and signalling constraints; the latter two were in fact found to arise from the same general categorical construction of relational constraints applied to distinct categorical structures. We proceeded to develop the formal language of constraints by introducing and studying additional notions of time-symmetry and of intersectability, proving that the former entails the latter.

As a first and focused outcome, providing a simple notion which encompasses both the constructions introduced by \cite{vanrietvelde2020routed} to describe sectorial constraints in \FHilb and the study of signalling conditions in unitary channels \cite{barrett2019}, the theory of composable constraints gives a formal backing to their intuitive similarity. In addition, the abstract phrasing of the framework leaves room for developing a language which can be used to study and understand their essential differences, for instance by expanding the observations of \cite{barrett2019} in which links are noted between important features of constraints and their time-symmetry. A more general outcome is the definition of the abstract structure underlying composable constraints, paving the way to the modelisation and discovery of examples beyond the current familiarity of the authors. 

A first future research direction focused specifically on the study of quantum theory, would be to take advantage of their general nature of the language of composable constraints and the constructions presented to model more elaborate quantum scenarios. For instance, sectorial constraints and signalling constraints ought to be extendable to infinite-dimensional Hilbert spaces; this may require generalisation of the construction of relational constraints, and give an opportunity to study the consequences of infinite dimensionality on intersectability and time-symmetry of constraints. Another direction would be to study both the relation of these constructions with, and their applications to, the study of causality in quantum theory. Besides causal decompositions, whose description requires the use of matching routes, indefinite causal order, which has recently been the subject of a lot of investigation, has deep connections with both signalling constraints on channels \textit{and} routed circuits \cite{barrett2020cyclic}; it may then be fruitful to apply the idea of routes to recent categorical investigations into causal structure \cite{Kissinger2019AStructure}. In any case the framework turns out to be sufficiently flexible to allow for theorems on the decomposition of a process into sectors based on signalling constraints \cite{lorenz2020} to be expressed neatly as genuine equations within a single category.

In the general non-quantum setting, a natural future direction of investigation is to model and construct constraints in yet other contexts; it is expected in particular that instances of composable constraints which leverage allegories and bicategories of relations \cite{Carboni1987CartesianI,Carboni2008CARTESIANII, Freyd1990CategoriesAllegories}, as opposed to their motivating instance of $\mathbf{Rel}$ are waiting to be observed within the current literature.

\begin{acknowledgments}
It is a pleasure to thank Jonathan Barrett, Bob Coecke, James Hefford, Aleks Kissinger, Hlér Kristjánsson, Nick Ormrod, and Vincent Wang for helpful discussions, advice and comments. Furthermore, the authors are extremely grateful for the valuable insight of an anonymous reviewer from ACT 2021, who suggested the categorical phrasing of the concept of a constraint as a lax functor. AV thanks Alexandra Elbakyan for her help to access the scientific literature. AV is supported by the EPSRC Centre for Doctoral Training in Controlled Quantum Dynamics. MW was supported by University College London and the EPSRC [grant number EP/L015242/1]. This publication was made possible through the support of the grant 61466 `The Quantum Information Structure of Spacetime (QISS)' (qiss.fr) from the John Templeton Foundation. The opinions expressed in this publication are those of the authors and do not necessarily reflect the views of the John Templeton Foundation.
\end{acknowledgments}

\newpage
\bibliographystyle{utphys.bst}
\bibliography{references}

\appendix

\begin{theorem}[Monoidal Category of Constrained Morphisms]
Let $\mathcal{L}: \mathbf{Con} \longrightarrow \mathcal{P}[\cat{C}]$ be a strong monoidal composable constraint, The constrained category $\cat{C}_{\mathcal{L}}$ is a monoidal category.
\end{theorem}
\begin{proof}
The monoidal product is given on objects by $(A,\mathcal{L}(A)) \odot (B,\mathcal{L}(B)) := (A\otimes B, \mathcal{L}(A \otimes B))$, on morphisms $(\tau,f) \odot (\lambda,g)$ is given by $(\tau \otimes \lambda, f \odot g)$ where $f \odot g$ is defined by requiring the following diagram to commute.

\be \begin{tikzcd}
\mathcal{F}(A \otimes B) \arrow[d, "\phi_{AB}^{-1}"'] \arrow[r, "f \odot g"] & \mathcal{F}(A' \otimes B') \arrow[d, "\phi_{A'B'}^{-1}"] \\
\mathcal{F}(A) \otimes \mathcal{F}(B) \arrow[r, "f \otimes g"]               & \mathcal{F}(A') \otimes \mathcal{F}(B')                 
\end{tikzcd}\ee 

First we must show that $\odot$ is a bifunctor, that is, that it satisfies the interchange law: \be \begin{tikzcd}
\mathcal{F}(A \otimes B) \arrow[d, "\phi_{AB}^{-1}"'] \arrow[r, "f_1 \odot g_1"] \arrow[rrr, "(f_2 \circ f_1)\odot(g_2 \circ g_1)", bend left] & \mathcal{F}(A' \otimes B') \arrow[r, "id"]                                       & \mathcal{F}(A' \otimes B') \arrow[d, "\phi_{A'B'}^{-1}"] \arrow[r, "f_2 \odot g_2"] & \mathcal{F}(A' \otimes B')                                          \\
\mathcal{F}(A) \otimes \mathcal{F}(B) \arrow[r, "f_1 \otimes g_1"']                                                                            & \mathcal{F}(A') \otimes \mathcal{F}(B') \arrow[u, "\phi_{A'B'}"] \arrow[r, "id"] & \mathcal{F}(A') \otimes \mathcal{F}(B') \arrow[r, "f_2 \otimes g_2"']               & \mathcal{F}(A') \otimes \mathcal{F}(B') \arrow[u, "\phi_{A'B'}"']
\end{tikzcd} \,.\ee 

The commutativity of the top follows from commutativity of the outer diagram and the lower diagrams. We then take the unit object $I_{\mathbf{C}_{\mathcal{L}}}$ of $\mathbf{C}_{\mathcal{L}}$ to be the unit object $I_{\mathbf{Con}}$ of $\mathbf{Con}$. A natural transformation $I \odot - \Rightarrow -$ can be defined by the components $(\lambda_A,\Lambda_A):I \odot A \rightarrow A$ where $\Lambda_A$ is defined by the following diagram:
\be  \begin{tikzcd}
\mathcal{L}IA \arrow[r, "\Lambda_{A}"] \arrow[d, "\psi_{IA}"']       & \mathcal{L}A                           \\
\mathcal{L}I\mathcal{L}A \arrow[r, "\psi \otimes i_{\mathcal{L}A}"'] & I \mathcal{L}A \arrow[u, "\lambda_A"']
\end{tikzcd} \ee 
and similarly for a right unitor $(\rho,P)$. Naturality of the components of $(\lambda,\Lambda)$ follows from from the following commutative diagrams:
\be  \begin{tikzcd}
                                                                                                        & \mathcal{L}A \arrow[r, "f"]                                                    &  \mathcal{L}B                                                               &                                         \\
                                                                                                        & I \mathcal{L}A \arrow[u, "\lambda_{\mathcal{L}A}"'] \arrow[r, "i_{I}f"']       & I \mathcal{L}B \arrow[u, "\lambda_{\mathcal{L}B}"]                          &                                         \\
\mathcal{L}IA \arrow[r, "\phi_{IA}"] \arrow[ruu, "\Lambda_A"] \arrow[rrr, "i_{I} \odot f"', bend right] & \mathcal{L}I\mathcal{L}A \arrow[u, "\lambda_A"] \arrow[r, "i_{\mathcal{L}I}f"] & \mathcal{L}I\mathcal{L}B \arrow[u, "\lambda_B"] \arrow[r, "\phi_{IB}^{-1}"] & \mathcal{L}IB \arrow[luu, "\Lambda_B"']
\end{tikzcd} \ee 
A natural transformation $(- \odot -) \odot - \Rightarrow - \odot (- \odot -)$ can be defined by the components $(\alpha,\bar{\alpha}): (a \odot b) \odot c \rightarrow a \odot (b \odot c)$ where $\bar{\alpha}:\mathcal{L}((a \otimes b) \otimes c)) \rightarrow \mathcal{L}(a \otimes (b \otimes c))$ is defined by the following diagram: \be \begin{tikzcd}
\mathcal{L}(AB)C \arrow[r, "\mathrm{A}"] \arrow[d, "{\psi_{AB,C}}"'] & \mathcal{L}A(BC)                                                 \\
\mathcal{L}AB\mathcal{L}C \arrow[d, "{\psi_{A,B}i}"']                & \mathcal{L}A\mathcal{L}BC \arrow[u, "{\psi_{A,BC}}"']            \\
(\mathcal{L}A\mathcal{L}B)\mathcal{L}C \arrow[r, "\alpha"]           & \mathcal{L}A(\mathcal{L}B\mathcal{L}C) \arrow[u, "i \psi_{BC}"']
\end{tikzcd}\ee 

and naturality follows from the following commutative diagram.
\be \begin{tikzcd}[scale cd=0.7]
\mathcal{L}(AB)C \arrow[d, "\mathrm{A}"] \arrow[r, "\psi^{-1}"'] \arrow[rrrrr, "(f \odot g) \odot h", bend left=49] & \mathcal{L}AB \mathcal{L}C \arrow[r, "\psi^{-1}i"'] \arrow[rrr, "(f \odot g)h", bend left]   & ( \mathcal{L}A \mathcal{L}B )\mathcal{L}C \arrow[r, "(fg)h"'] \arrow[d, "\alpha"] & ( \mathcal{L}A' \mathcal{L}B' )\mathcal{L}C' \arrow[d, "\alpha"] \arrow[r, "\psi i"'] & \mathcal{L}A'B' \mathcal{L}C' \arrow[r, "\psi"'] & \mathcal{L}(A'B')C' \arrow[d, "\mathrm{A}"'] \\
\mathcal{L}A(BC) \arrow[r, "\psi^{-1}"] \arrow[rrrrr, "f \odot (g \odot h)"', bend right=49]                        & \mathcal{L}A \mathcal{L}BC \arrow[r, "i \psi^{-1}"] \arrow[rrr, "f(g \odot h)"', bend right] &  \mathcal{L}A ( \mathcal{L}B \mathcal{L}C) \arrow[r, "f(gh)"]                     &  \mathcal{L}A' ( \mathcal{L}B' \mathcal{L}C') \arrow[r, "i \psi"]                     & \mathcal{L}A' \mathcal{L}B'C' \arrow[r, "\psi"]  & \mathcal{L}A'(B'C')                         
\end{tikzcd}\ee 
For the coherence conditions between associators and unitors, coherence is again trivial in the first component of each $(\tau,f)$, and coherence in the second component between the unitors and associators is guaranteed by the commutativity of the following diagram: \be \begin{tikzcd}[scale cd=0.7]
\mathcal{L}(AI)C \arrow[rrrddddd, "\mathrm{P} \odot i", bend right=49] \arrow[rrrrrr, "{\mathrm{A}_{A,I,C}}"'] \arrow[rd, "{\psi_{AI,C}}"] &                                                                                                   &                                                                                                                              &                                                 &                                                                &                                                                                                & \mathcal{L}A(IC) \arrow[lllddddd, "i \odot \Lambda"', bend left=49] \arrow[ld, "{\psi_{A,IC}}"'] \\
                                                                                                                                           & \mathcal{L}AI\mathcal{L}C \arrow[rrddd, "\mathrm{P}i", bend right=49] \arrow[rd, "{\psi_{A,I}i}"] &                                                                                                                              &                                                 &                                                                & \mathcal{L}A\mathcal{L}IC \arrow[llddd, "i \Lambda"', bend left=49] \arrow[ld, "i \psi_{IC}"'] &                                                                                                  \\
                                                                                                                                           &                                                                                                   & (\mathcal{L}A\mathcal{L}I)\mathcal{L}C \arrow[d, "(i \psi)i"] \arrow[rr, "{\alpha_{\mathcal{L}A\mathcal{L}I,\mathcal{L}C}}"] &                                                 & \mathcal{L}A(\mathcal{L}I\mathcal{L}C) \arrow[d, "i(\psi i)"'] &                                                                                                &                                                                                                  \\
                                                                                                                                           &                                                                                                   & ((\mathcal{L}A)I)\mathcal{L}C \arrow[rr, "{\alpha_{(\mathcal{L}A)I,\mathcal{L}C}}"] \arrow[rd, "\rho_{A} i"]                 &                                                 & \mathcal{L}A(I(\mathcal{L}C)) \arrow[ld, "i \lambda_A"']       &                                                                                                &                                                                                                  \\
                                                                                                                                           &                                                                                                   &                                                                                                                              & \mathcal{L}A\mathcal{L}C \arrow[d, "\psi_{AC}"] &                                                                &                                                                                                &                                                                                                  \\
                                                                                                                                           &                                                                                                   &                                                                                                                              & \mathcal{L}(AC)                                 &                                                                &                                                                                                &                                                                                                 
\end{tikzcd}\ee               
Finally, the pentagon identity follows from the following. 
\be \begin{tikzcd}[scale cd=0.7]
                                                                                                                 &                                                                                                   & \mathcal{L}(AB)(CD) \arrow[d, "\psi^{-1}"] \arrow[rrd, "\mathrm{A}"]                            &                                                                                                       &                                                                                               \\
\mathcal{L}((AB)C)D \arrow[rru, "\mathrm{A}"] \arrow[d]                                                          &                                                                                                   & \mathcal{L}AB\mathcal{L}CD \arrow[d, "\psi^{-1} i"]                                             &                                                                                                       & \mathcal{L}A(B(CD))                                                                           \\
(\mathcal{L}(AB)C)\mathcal{L}D \arrow[d]                                                                         & \mathcal{L}AB(\mathcal{L}C \mathcal{L}D) \arrow[ru, "i \psi"] \arrow[r, "\psi^{-1} \psi"']        & (\mathcal{L}A \mathcal{L} B ) \mathcal{L} CD \arrow[d, "(i i) \psi^{-1}"] \arrow[rrd, "\alpha"] &                                                                                                       & \mathcal{L}A\mathcal{L}B(CD) \arrow[u, "\psi"]                                                \\
(\mathcal{L}AB\mathcal{L}C) \mathcal{L}D \arrow[rd, "(\psi i) i"] \arrow[ru, "\alpha"']                          &                                                                                                   & (\mathcal{L}A \mathcal{L}B) (\mathcal{L}C \mathcal{L} D) \arrow[rd, "\alpha"]                   &                                                                                                       & \mathcal{L}A(\mathcal{L}A \mathcal{L}CD)) \arrow[u, "i \psi"]                                 \\
                                                                                                                 & ((\mathcal{L}A\mathcal{L}B) \mathcal{L}C) \mathcal{L}D \arrow[ru, "\alpha"] \arrow[d, "\alpha i"] &                                                                                                 & \mathcal{L}A( \mathcal{L}B(\mathcal{L}C \mathcal{L}D)) \arrow[ru, "i(i\psi)"]                         &                                                                                               \\
                                                                                                                 & (\mathcal{L}A(\mathcal{L}B\mathcal{L}C))\mathcal{L}D \arrow[rr, "\alpha"]                         &                                                                                                 & \mathcal{L}A ((\mathcal{L}B \mathcal{L}C) \mathcal{L}D) \arrow[u, "i \alpha"] \arrow[rd, "i(\psi i)"] &                                                                                               \\
(\mathcal{L}A\mathcal{L}BC) \mathcal{L}D \arrow[ru, "(i \psi)i"] \arrow[rrrr, "\alpha"']                         &                                                                                                   &                                                                                                 &                                                                                                       & \mathcal{L}A(\mathcal{L}BC\mathcal{L}D)) \arrow[d, "i \psi"']                                 \\
\mathcal{L}A(BC)\mathcal{L}D \arrow[u, "\psi i"'] \arrow[d, "\psi"] \arrow[uuuuu, "\mathrm{A} i"', bend left=49] &                                                                                                   &                                                                                                 &                                                                                                       & \mathcal{L}A \mathcal{L}(BC)D \arrow[d, "\psi"'] \arrow[uuuuu, "i \mathrm{A}", bend right=49] \\
\mathcal{L}(A(BC))D \arrow[uuuuuuu, "\mathrm{A} \odot i", bend left=49] \arrow[rrrr, "\mathrm{A}"]               &                                                                                                   &                                                                                                 &                                                                                                       & \mathcal{L}A((BC)D) \arrow[uuuuuuu, "i \odot \mathrm{A}"', bend right=49]                    
\end{tikzcd}\ee 
As a result, the constrained category is indeed monoidal.
\end{proof}

\begin{theorem}[Inheritance of structure to constrained categories]
If a composable constraint $\mathcal{L}$ is:
\begin{itemize}
    \item a $\dagger$-constraint, then $\mathbf{C}_{\mathcal{L}}$ is a $\dagger$-category;
    \item a braided monoidal constraint, then $\mathbf{C}_{\mathcal{L}}$ is braided monoidal;
    \item a compact closed constraint, then $\mathbf{C}_{\mathcal{L}}$ is compact closed.
\end{itemize}
\end{theorem}
\begin{proof}
For $\dagger$-constraints, define $(\tau,f)^{\dagger}:=(\tau^{\dagger},f^{\dagger})$; this morphisms exists for any $(f,\tau)$ since $f \in \mathcal{L}(\tau) \implies f^{\dagger} \in \mathcal{L}(\tau)^{\dagger} = \mathcal{L}(\tau^{\dagger})$. If the constraint is furthermore $\dagger$-monoidal then the constrainted category is also $\dagger$-monoidal: in particular one can show that $((\tau,f) \odot (\tau',f'))^{\dagger} = (\tau,f)^{\dagger} \odot (\tau',f')^{\dagger}$, by the commutativity of the following diagram: \be \begin{tikzcd}
\mathcal{L}AB \arrow[d, "\phi^{\dagger}"]  & \mathcal{L}A'B' \arrow[l, "f^{\dagger} \odot g^{\dagger}"']                                                                   \\
\mathcal{L}A\mathcal{L}B \arrow[d, "\phi"] & \mathcal{L}A' \mathcal{L}B' \arrow[l, "(fg)^{\dagger}"] \arrow[l, "f^{\dagger}g^{\dagger}"', bend right=49] \arrow[u, "\phi"] \\
\mathcal{L}AB                              & \mathcal{L}A'B' \arrow[l, "(f \odot g)^{\dagger}"] \arrow[u, "\phi^{\dagger}"]                                               
\end{tikzcd} \,.\ee 

For braided constraints, the coherence conditions for $\mathbf{B}$ as defined in the main text are given by the following commutative diagram: \be \begin{tikzcd}[scale cd=0.7]
                                                                                         &                                                                                              & \mathcal{L}(AB)C \arrow[d, "\phi^{-1}"] \arrow[r, "\mathbf{A}"]                 & \mathcal{L}A(BC) \arrow[d, "\phi^{-1}"] \arrow[r, "\mathbf{B}"]                & \mathcal{L}(BC)A \arrow[d, "\phi^{-1}"] \arrow[rrdd, "i"]                       &                                       &                                           \\
                                                                                         &                                                                                              & (\mathcal{L}AB)\mathcal{L}C \arrow[d, "\phi^{-1}i"]                             & \mathcal{L}A(\mathcal{L}BC) \arrow[d, "i \phi^{-1}"] \arrow[r]                 & \mathcal{L}(BC)\mathcal{L}A \arrow[d, "i \phi^{-1}"'] \arrow[rd, "i"]           &                                       &                                           \\
\mathcal{L}(AB)C \arrow[rruu, "i"] \arrow[r, "\phi^{-1}"] \arrow[d, "\mathbf{B}\odot i"] & (\mathcal{L}AB)\mathcal{L}C \arrow[ru, "i"] \arrow[r, "\phi^{-1}i"] \arrow[d, "\mathbf{B}i"] & (\mathcal{L}A\mathcal{L}B)\mathcal{L}C \arrow[r, "\alpha"] \arrow[d, "\beta i"] & \mathcal{L}A(\mathcal{L}B\mathcal{L}C) \arrow[r]                               & (\mathcal{L}B\mathcal{L}C)\mathcal{L}A \arrow[d, "\alpha"] \arrow[r, "i \phi"'] & \mathcal{L}(BC)\mathcal{L}A \arrow[r] & \mathcal{L}(BC)A \arrow[d, "\mathbf{A}"'] \\
\mathcal{L}(BA)C \arrow[r, "\phi^{-1}"] \arrow[rrdd, "i"]                                & (\mathcal{L}BA)\mathcal{L}C \arrow[r, "\phi^{-1}i"] \arrow[rd, "i"]                          & (\mathcal{L}B\mathcal{L}A)\mathcal{L}C \arrow[r, "\alpha"]                      & \mathcal{L}B(\mathcal{L}A\mathcal{L}C) \arrow[r, "i \beta"]                    & \mathcal{L}B(\mathcal{L}C\mathcal{L}A) \arrow[r, "i \phi"]                      & \mathcal{L}B(\mathcal{L}CA) \arrow[r] & \mathcal{L}B(CA)                          \\
                                                                                         &                                                                                              & (\mathcal{L}BA)\mathcal{L}C \arrow[u, "\phi^{-1}i"']                            & \mathcal{L}B(\mathcal{L}AC) \arrow[u, "i \phi^{-1}"] \arrow[r, "i \mathbf{B}"] & \mathcal{L}B(\mathcal{L}CA) \arrow[u, "i \phi^{-1}"] \arrow[ru, "i"]            &                                       &                                           \\
                                                                                         &                                                                                              & \mathcal{L}(BA)C \arrow[u, "\phi^{-1}"'] \arrow[r, "\mathbf{A}"']               & \mathcal{L}B(AC) \arrow[u, "\phi^{-1}"] \arrow[r, "i \odot \mathbf{B}"']       & \mathcal{L}B(CA) \arrow[u, "\phi^{-1}"] \arrow[rruu, "i"]                       &                                       &                                          
\end{tikzcd}\ee

Finally, a compact closed category is one in which every object has a dual. We can show that whenever a constraint is itself compact closed as defined in the main text, then the resulting constraint category is also compact closed; this follows from the commutativity of the following diagram: \be \begin{tikzcd}
                                                                             &                                                     & \mathcal{L}IA \arrow[d, "\phi^{-1}"] \arrow[lld, "\Lambda"']                 &                                                                                 & \mathcal{L}(AA^{*})A \arrow[ld, "\phi^{-1}"'] \arrow[ll, "\mathbf{E} \odot i"'] \\
\mathcal{L}A                                                                 &                                                     & \mathcal{L}I \mathcal{L}A \arrow[d, "\phi^{-1}i"]                            & \mathcal{L}AA^{*}\mathcal{L}A \arrow[d, "\phi^{-1}i"] \arrow[l, "\mathbf{E}i"'] &                                                                                 \\
                                                                             & \mathcal{L}A \arrow[lu, "i"]                        & I \mathcal{L}A \arrow[l, "\lambda"]                                          & (\mathcal{L}A\mathcal{L}A^{*}) \mathcal{L}A \arrow[l, "\epsilon i"]             &                                                                                 \\
                                                                             &  \mathcal{L}A \arrow[r, "\rho^{-1}"] \arrow[u, "i"] & (\mathcal{L}A) I \arrow[r, "i \eta"]                                         & \mathcal{L}A (\mathcal{L}A^{*} \mathcal{L}A) \arrow[u, "\alpha"]                &                                                                                 \\
\mathcal{L}A \arrow[ru, "i"] \arrow[uuu, "i"] \arrow[rrd, "\mathbf{P}^{-1}"] &                                                     & \mathcal{L}A \mathcal{L}I \arrow[u, "i \phi^{-1}"] \arrow[r, "i \mathbf{N}"] & \mathcal{L}A\mathcal{L}A^{*}A \arrow[u, "i \phi^{-1}"]                          &                                                                                 \\
                                                                             &                                                     & \mathcal{L}AI \arrow[u, "\phi^{-1}"'] \arrow[rr, "i \odot \mathbf{N}"]       &                                                                                 & \mathcal{L}A(A^{*}A) \arrow[lu, "\phi^{-1}"'] \arrow[uuuuu, "\mathbf{A}"]      
\end{tikzcd}\,.\ee 

Commutativity of the diagram required for duals to be both left \textit{and} right duals is derived almost identically.

\end{proof}

\end{document}